\documentclass[a4paper,10pt,reqno]{amsart}

\usepackage{xcolor,calc, blindtext}
\usepackage{amsmath,amssymb,mathtools,amsthm,mathrsfs,bm}
\numberwithin{equation}{section}

\usepackage{YDLthesissymbols}

\usepackage[colorlinks]{hyperref}

\usepackage{setspace}
\usepackage[margin=1.2in]{geometry} 
\usepackage{graphicx}     
 
\usepackage{fancyhdr}  
\usepackage{appendix}
\usepackage{subfiles}
\usepackage{bbm}
\usepackage{tikz-cd,hf-tikz}
\usepackage[utf8]{inputenc}
\usepackage[T1]{fontenc}
\usepackage[normalem]{ulem}
\usepackage{indentfirst} 
\usepackage{enumerate} 
\usepackage{enumitem} 
\usepackage{lscape}
\usepackage{etoolbox}
\usepackage{stmaryrd}
\usepackage{tabularx}
\usepackage{import}
\usepackage{multirow}

\newlength{\perspective}

\usepackage[french,english]{babel}

\newcommand{\newabstract}[1]{%
  \par\bigskip
  \csname otherlanguage*\endcsname{#1}%
  \csname captions#1\endcsname
  \item[\hskip\labelsep\scshape\abstractname.]
}

\allowdisplaybreaks

\title{Intégrale orbitale pondérée via l'induite de Lusztig-Spaltenstein généralisée}
\author{Yan-Der LU}
\address{Université Paris-Cité, IMJ-PRG}
\date{\today}

\begin{document}

{\selectlanguage{english}\begin{abstract}
In this article, we present two novel approaches to constructing weighted orbital integrals of an inner form of a general linear group. Our method utilizes generalized Lustig-Spaltenstein induction. Furthermore, we will prove that a weighted orbital integral on the Lie algebra constitutes a tempered distribution. We also demonstrate that our new definitions and Arthur's original definition are consistent.
\end{abstract}}

\selectlanguage{french}

\maketitle

\tableofcontents

\section{Introduction}

\subsection{Préface}
La formule des traces non-invariante d'Arthur est un outil puissant en théorie des formes automorphes et en fonctorialité de Langlands. On trouve dans le développement fin du côté géométrique des intégrales orbitales pondérées $J_M^G(x,f)$, qui sont définies, sur un corps local par exemple, comme suit : soient $G$ un groupe réductif sur un corps local $F$, $M$ un sous-groupe de Levi semi-standard relativement à un sous-groupe de Levi minimal fixé, $x\in M(F)$, et $f$ une fonction test sur $G(F)$. Supposons d'abord que le centralisateur de $x$ dans $G$ est inclus dans $M$, i.e. $M_x=G_x$, alors on pose 
\[J_M^G(x,f)=|D^G(x)|_{F}^{1/2}\int_{M_x^\circ(F)\backslash G(F)}f\left(\Ad (g^{-1})x\right)v_M^G(g)\,dg,\]
avec $M_x^\circ$ la composante neutre de $M_x$, $|D^G(x)|_{F}^{1/2}$ un facteur de normalisation, et $v_M^G(g)$ est un « poids » obtenu par la théorie des $(G,M)$-familles d'Arthur. Le cas général invoque la limite d'une somme d'intégrales orbitales pondérées  décrites ci-dessus : on sait que $M_{ax}=G_{ax}$ pour $a$ un élément en position générale dans le groupe des $F$-points du sous-tore central $F$-déployé maximal de $M$, donc l'intégrale $J_L^G(ax,f)$ est définie pour tout $L$ sous-groupe de Levi contenant $M$. On pose maintenant
\[J_M^G(x,f)=\lim_{a\to 1}\sum_{L\in\L^G(M)}r_M^L(x,a)J_L^G(ax,f),\]
avec $r_M^L(x,a)
$ des fonctions obtenues également par la théorie des $(G,M)$-familles. 

Pourtant avec cette formulation que met en vedette Arthur, le nom « intégrale orbitale pondérée » semble moins pertinent, puisque $J_M^G(x,f)$ n'a plus l'air d'être une intégrale sur une orbite, contre une mesure pondérée. Le problème ne s'arrête pas là, il est opportun de prendre une fonction test qui n'est ni lisse ni à support compact dans la formule des traces pour des certaines finalités, la convergence et l'existence de la limite ci-dessus ne semblent pas évidentes dans cette situation. C'est ce qui nous pousse à la recherche d'une nouvelle définition d'une intégrale orbitale pondérée.

Une nouvelle construction d'une intégrale orbitale pondérée unipotente, pour un groupe général linéaire, est déjà découverte par Chaudouard \cite{Ch17}. Généralisant son approche, nous proposons une définition des intégrales orbitales pondérées, au niveau de l'algèbre de Lie, des groupes généraux linéaires ainsi que leurs formes intérieures, pour les fonctions de classe Schwartz-Bruhat. 

\subsection{Résultats principaux}

On commence par généraliser la notion de l'orbite induite.

\begin{proposition}[{{proposition \ref{chap3prop:indprop}}}]
Soit $G$ un groupe réductif défini sur un corps $F$. Soient $M$ un sous-groupe de Levi de $G$ et $X\in \m(F)$. On note $\o= (\Ad M)X$.
\begin{enumerate}
    \item Il existe une unique orbite $\Ind_M^G(\o)=\Ind_M^G(X)$ dans $\g$ pour l’action adjointe de $G$ telle que l’intersection
    \begin{equation}\label{chap3eq:introindorbdef}
    \Ind_M^G(\o)\cap \left(\o+\mathfrak{n}_{P}\right)    
    \end{equation}
    soit un ouvert Zariski dense dans $\o+\mathfrak{n}_{P}$ pour tout $P$ sous-groupe parabolique de $G$ ayant $M$ comme facteur de Levi ;
    \item si $F$ est un corps parfait, alors l'orbite induite commute à la décomposition de Jordan, i.e.
    \begin{equation}\label{chap3eq:introindJordan}
    \Ind_M^G(X)=\Ad(G)(\Ind_{M_{X_{\ss}}^\circ}^{G_{X_{\ss}}^\circ}(X))=\Ad(G)(X_{\ss}+\Ind_{M_{X_{\ss}}^\circ}^{G_{X_{\ss}}^\circ}(X_{\nilp})) ;   
    \end{equation}
    \item $\text{codim}_{\mathfrak{m}}(\o)=\text{codim}_{\g}(\Ind_M^G(\o))$ ;
    \item pour toute partie $S\subseteq \g$, on note par $S_{G-\reg}$ le lieu régulier pour l'action adjointe de $G$, autrement dit
    \[
    S_{G-\reg}\eqdef\{X\in S\mid \dim (\Ad G)X=\max_{Y\in S}\left( \dim (\Ad G)Y\right)\}.\]  
    Alors l'intersection \eqref{chap3eq:introindorbdef} est également le lieu régulier de $\o+\mathfrak{n}_{P}$ ;
    \item soit $P$ un sous-groupe parabolique de $G$ ayant $M$ comme facteur de Levi, alors $\Ind_M^G(\o)\cap \left(\o+\mathfrak{n}_{P}\right)$ est une orbite dans $\p$ pour l'action adjointe de $P$ ;
    \item soit $P$ un sous-groupe parabolique de $G$ ayant $M$ comme facteur de Levi, soit $Y\in \Ind_M^G(\o)\cap (\o+\mathfrak{n}_{P})$, alors $G_{Y}^\circ\subseteq P$ ;
    \item soit $L$ un sous-group de Levi de $H$ contenant $M$, alors
    \[\Ind_{L}^G(\Ind_{M}^L(\o))=\Ind_{M}^G(\o).\]
\end{enumerate}
\end{proposition}

On peut également induire une orbite d'un sous-groupe parabolique.

\begin{definition}[{{définition \ref{YDLiopdef:IndPG}}}]
Soient $P$ un sous-groupe de Levi de $G$ et $X \in\mathfrak{p}(F)$. On définit
\[\Ind_P^G(X)\eqdef \Ind_M^G(\pi_{\p,\mathfrak{m}}(X)).\]
Ici $M$ est un facteur de Levi de $P$ et $\pi_{\p,\mathfrak{m}}:\mathfrak{p}\rightarrow \mathfrak{m}$ est la projection.

La définition de $\Ind_P^G(X)$ est indépendante du choix de $M$ car la sous-variété $(\Ad M)\pi_{\p,\mathfrak{m}}(X)+\n_P$ de $\g$ l'est. 
\end{definition}

Expliquons maintenant nos démarches pour définir une intégrale orbitale pondérée. Soit $G$ une forme intérieure d'un groupe groupe général linéaire sur un corps local $F$. Notre nouvelle définition d'une intégrale orbitale pondérée repose sur la manipulation des « sous-groupes paraboliques de Richardson généralisés ». 

Soit $X\in \g(F)$. On note $X_{\ss}$ sa partie semi-simple et $X_{\nilp}$ sa partie nilpotente. Définissons (équation \eqref{eq:cRG'X})
\[\cR^{G}(X)'=\{P\subseteq G\text{ sous-}\text{groupe parabolique} \mid X_{\nilp}\in \n_{P}, X\in \p, X\in\Ind_{P}^{G}(X_{\ss})\},\]
puis (équation \eqref{eq:cRGX})
\[\cR^{G}(X)=\{P\in \cR^{G}(X)' \mid P_v\text{ un élément minimal (pour l'inclusion) dans }\cR^{G}(X)'\}.\]
Les éléments de $\cR^{G}(X)$ sont appelés les sous-groupes paraboliques de Richardson généralisés de $X$.

On fixe maintenant $M_{0}$ un sous-groupe de Levi minimal de $G$, et $K$ un sous-groupe compact de $G(F)$ en bonne position par rapport à $M_{0}$. On note $\uW^{G,0}\subseteq K$ le sous-groupe des matrices de permutations (numéro \ref{YDLiopsubsec:elementwP}).
Soient $L$ un sous-groupe de Levi de $G$ contenant $M_{0}$, $Q$ un sous-groupe parabolique de $G$ contenant $L$ et $\o$ une $L(F)$-orbite dans l'algèbre de Lie $\mathfrak{l}(F)$.  Notre objectif est de définir l'intégrale orbitale pondérée locale $J_{L}^{Q}(\o,-)$.

Soit $X\in \Ind_{L}^{G}(\o)(F)$ le « représentant standard » (numéro \ref{subsec:ssssschoix}). On sait que tout sous-groupe parabolique de Richardson généralisé de $X$ contient $M_{0}$ (lemme \ref{lem:LSsstand}). Fixons $\MR^{X}$ le facteur de Levi contenant $M_0$ d'un élement de $\cR^G(X)$. On prouve, pour tout sous-groupe parabolique $P$ ayant $\MR^{X}$ comme facteur de Levi, qu'il existe un élément $w_P\in \uW^{G,0}$, unique à translation à gauche près par $\uW^{G,0}\cap \MR^{X}(F)$, tel que $(\Ad w_P^{-1})P\in \cR^G(X)$ (lemme \ref{lem:wP}). Pour $H$ un groupe algébrique défini sur $F$, on note $X^\ast(H)$ le groupe des caractères de $H$ définis sur $F$ et $a_H$ (resp. $a_H^\ast$) l’espace vectoriel reél $\Hom(X^\ast(H),\R)$ (resp.$X^\ast(H)\otimes_\Z\R$). Aussi $ia_{\MR^{X}}^\ast$ est le sous-espace  reél évident de $a_{\MR^{X}}^\ast\otimes_\R\C$. En outre, on introduit la $(G,\MR^{X})$-famille $v_{P,X}(\lambda,g)\eqdef \exp(\langle\lambda,-H_P(w_Pg)\rangle)$ (équation \ref{YDLiopeq:defnewGMfamilyinIOP}), où $H_P:G(F)\to a_{\MR^{X}}$ est l'application de Harish-Chandra et $\lambda\in ia_{\MR^{X}}^\ast$. De plus, on prouve qu'il existe $w\in\uW^{G,0}$ tel que $(\Ad w^{-1})\MR^{X}\subseteq L$ et $\o=\Ind_{(\Ad w^{-1})\MR^{X}}^{L}((\Ad w^{-1})X_{\ss})$ (corollaire \ref{coro:choixtow}). On peut de ce fait regarder le « poids » $v_{(\Ad w)L_v,X_v}^{(\Ad w)Q}$ associé à cette $(G,\MR^{X})$-famille. Pour finir, on prouve que ce poids est invariant à gauche par le groupes des $F$-points du centralisateur de $X$ dans $G$ (lemme \ref{lem:poidsinvnorm}). Ceci nous fournit les ingrédients appropriés pour définir une intégrale orbitale pondérée.

Soient $C_c^\infty(\g(F))$ l'espace des fonctions lisses à support compact de $\g(F)$, et $\S(\g(F))$ l'espace de Schwartz-Bruhat de $\g(F)$, tous les deux munis de leurs topologies usuelles. Une fonctionnelle sur $\g(F)$ est dite une distribution si elle est continue sur $C_c^\infty(\g(F))$, et elle est dite une distribution tempérée si elle est continue sur $\S(\g(F))$. 

\begin{theorem}[définition \ref{def:NIOPl}, théorème \ref{thm:locIOPtemperee}]
Soient $f\in \S(\g(F))$, $L$ un sous-groupe de Levi de $G$ contenant $M_0$, $Q$ un sous-groupe parabolique de $G$ contenant $L$, et $\o$ une $L(F)$-orbite dans l'algèbre de Lie $\mathfrak{l}(F)$. Soit $X\in \Ind_L^G(\o)(F)$ le représentant standard (numéro \ref{subsec:ssssschoix}). Soit $\MR^
{X}$ le facteur de Levi contenant $M_0$ d'un élement de $\cR^G(X)$. Il existe $w\in \uW^{G,0}$ tel que $(\Ad w^{-1})\MR^{X}\subseteq L$ et $\o=\Ind_{(\Ad w^{-1})\MR^X}^L((\Ad w^{-1})X_\ss)$ (corollaire \ref{coro:choixtow}). On pose,
\[
J_L^Q(\o,f)\eqdef |D^{\g}(X)|_{F}^{1/2} \int_{G_{X}(F)\backslash G(F)} f((\Ad g^{-1})X)v_{(\Ad w)L,X}^{(\Ad w)Q}(g)\,dg,
\]
avec $D^{\g}(X)$ le discriminant de Weyl de $X$ (cf. sous-section \ref{YDLiopsubsec:normalizationmeasure}) et $|\cdot|_F$ la valeur absolue usuelle sur $F$. Alors l'intégrale $J_L^Q(\o,-)$ définit une distribution tempérée.
\end{theorem}

Dans son article \cite{Art86}, Arthur établit le développement fin du côté géométrique de la formule des traces, qui stipule essentiellement que ce dernier est entièrement déterminé par les intégrales orbitales pondérées et les coefficients $a^G(S,\o)$. Cependant, sa méthode ne fournit que peu d'éléments concernant concernant ces coefficients, à l'exception des cas où $\o$ serait une orbite semi-simple elliptique. Ce n'est que récemment que des formules intégrales, conjecturées par Hoffmann (cf. \cite{Hoff16}), ont été établies (cf. \cite{Ch17,Ch18} sur l'algèbre de Lie ; et \cite{HoffWaka18,FiHoffWaka18} sur le groupe), permettant ainsi de calculer ces coefficients dans certains cas particuliers. Les propriétés de ces coefficients ont plusieurs applications en théorie des nombres ; par exemple, elles peuvent être utilisées pour étudier le comportement asymptotique des valeurs propres des opérateurs de Hecke (cf. \cite{Ma17,MaTe15,KimWakaYama20}). Il convient de noter que c'est par le biais de la définition plus haute au cas nilpotent que Chaudouard a pu réaliser à travers ses textes les calculs de ces coefficients. Nous disposons d'une description purement combinatoire des éléments $w, w_P \in \uW^{G,0}$ mentionnés précédemment dans la définition du poids (cf. remarque \ref{rem:w_Psimple} et proposition \ref{prop:paradesccomp}), ce qui constitue un avantage de notre définition de l'intégrale orbitale pondérée. Il serait par exemple envisageable d'étendre les calculs dans la  pour déterminer explicitement les coefficients $a^G(S,\o)$ pour une orbite $\o$ non-nilpotente ou non-unipotente, rendant ainsi la formule de descente au centralisateur semi-simple de $a^G(S, \o)$ (\cite[équation (8.1)]{Art86}) plus directe, sans avoir recours aux méthodes techniques de \cite{Art86}.

En outre, la compréhension du fait qu'une intégrale orbitale pondérée se présente comme une distribution tempérée sur l'algèbre de Lie ouvre la voie à l'étude de sa transformée de Fourier. Cette perspective nous permettrait de tenter de généraliser les résultats classiques de l'analyse harmonique de Waldspurger en \cite{Walds95,Walds97}, qui ont été démontrés dans le contexte d'un corps $p$-adique, à tout corps local de caractéristique 0.

La raison de prendre l'espace de Schwartz-Bruhat en tant que l'espace des fonctions tests est naturellement issue de l'analyse harmonique. Toujours est-il que des calculs explicites sont accomplis (cf. les lemmes \ref{lem:diffadjacentR_P} et \ref{lem:centralizerOIformula} et la discussion après ce dernier) pour que l'on puisse s'interroger sur la convergence de $J_{L}^Q(\o,f)$ pour $f$ une fonction test vivant dans un espace plus large.

On souhaite dans un second temps comparer cette définition d'une intégrale orbitale pondérée avec celle d'Arthur. Dans \cite{Art88}, Arthur définit une intégrale orbitale pondérée en utilisant des certaines $(G,M)$-familles $w_P$ et $r_P$. Il commence par définir ces objets pour les orbites nilpotentes, puis les généralise au cas général en utilisant une descente au centralisateur semi-simple. Cependant, grâce aux travaux de Finis-Lapid sur le groupe (\cite
[section 7]{FiLa16}), ainsi qu'à ceux de Chaudouard sur l'algèbre de Lie (\cite[section 3]{Ch18}), le côté géométrique de la formule des traces est désormais partitionné selon les classes de conjugaison usuelles, via le processus d'induction de l'orbite. Il semble donc plus naturel de définir ces $(G,M)$-familles $w_P$ et $r_P$ directement à partir du processus d'induction de l'orbite, sans recourir à la descente au centralisateur semi-simple. Nous généralisons donc l'approche d'Arthur aux orbites quelconques, obtenant ainsi une autre nouvelle définition $\widetilde{J}_L^Q(\o,-)$ (théorème \ref{YDLiopthm:ArtdefdirectIOP}), qui est également une distribution tempérée. Finalement, nous notons $\widetilde{J}_L^Q[\o,-]$ (théorème \ref{YDLiopthm:ArtdefIOPCVCinftyc}) la définition d'Arthur, qui est une distribution.   

On peut comparer ces trois définitions.
\begin{theorem}[{{théorèmes \ref{thm:IOPcomparaisonI-II}, \ref{thm:IOPcomparaisonII-III}}}]
Soient $L$ un sous-groupe de Levi de $G$ contenant $M_0$, $Q$ un sous-groupe parabolique de $G$ contenant $L$, et $\o$ une $L(F)$-orbite dans l'algèbre de Lie $\mathfrak{l}(F)$. Alors
\[J_L^Q(\o,-)=\widetilde{J}_L^Q(\o,-)=\widetilde{J}_L^Q[\o,-]\]
sur $C_c^\infty(\g(F))$. 
\end{theorem}


\subsection{Plan de l'article}
Le contenu du présent texte est organisé dans cet ordre : en section \ref{sec:1preliminaire} on introduit des notations fondamentales et on généralise la notion de l'orbite induite de Lusztig-Spaltenstein. En section \ref{sec:2nouvdef} on propose une nouvelle définition d'une intégrale orbitale pondérée sur l'algèbre de Lie. En section \ref{sec:3cv} on aborde la question de convergence de notre définition. En section \ref{sec:4comparedef}, on généralise l'approche d'Arthur d'une intégrale orbitale pondérée, on discute d'autres possibles définitions et on procède à la comparaison des définitions données. 

\subsection*{Remerciements}
L'auteur tient particulièrement à remercier Pierre-Henri Chaudouard, son directeur de thèse. Sans ses connaissances, conseils, temps et patience, cet article n'aurait pas été possible. L'auteur tient également à remercier l'école doctorale 386 Sciences Mathématiques de Paris Centre ainsi que l'université Paris-Cité pour avoir financé ce projet de thèse.

\section{
Préliminaires}\label{sec:1preliminaire}
\subsection{Notations}

Soit $F$ un corps. Soit $G$ un groupe algébrique défini sur $F$. On note par la même lettre en minuscule gothique son algèbre de Lie, en l'occurrence $\g=\Lie(G)$. On note $Z(G)$ le centre de $G$. L'action adjointe de $G$ sur $\g$ est notée $\text{Ad}$. On écrit $G^\circ$ la composante neutre de $G$. On note $A_G$ le sous-tore central $F$-deployé maximal de $G$. Pour toute partie $S\subseteq G$ ou $S\subseteq \g$, $\Cent(S,G)$ est son centralisateur dans $G$. Sauf mention contraire, un sous-groupe de $G$ signifie un sous-groupe algébrique de $G$ défini sur $F$. 

Supposons désormais que $G$ est réductif et connexe. 

Pour toute partie $S\subseteq \g$, on note par $S_
{G-\reg}$ le lieu régulier pour l'action adjointe de $G$, autrement dit
\begin{equation}\label{eq:deflieuG-reg}
S_{G-\reg}
\eqdef\{X\in S\mid \dim (\Ad G)X=\max_{Y\in S}\left( \dim (\Ad G)Y\right)\}.    
\end{equation}
En général $S_{G-\reg}\not= \g_{G-\reg}\cap S$. En revanche $\mathfrak{c}_{G-\reg}= \g_{G-\reg}\cap \mathfrak{c}$ pour $\mathfrak{c}$ une sous-algèbre de Cartan, et $\mathcal{N}_{\g,G-\reg}= \g_{G-\reg}\cap \mathcal{N}_{\g}$ où $\mathcal{N}_{\g}$ est le cône nilpotent de $\g$.

On appelle sous-groupe de Levi de $G$ un groupe qui est une composante de Levi d'un sous-groupe parabolique de $G$. Soient $M$ un sous-groupe de Levi de $G$ et $H$ un sous-groupe de Levi ou parabolique de $G$ qui contient $M$. On note $\L^H(M)$ l’ensemble des sous-groupes de Levi de $G$ inclus dans $H$ et contenant $M$ ; on note $\P^H(M)$ l’ensemble des sous-groupes paraboliques de $G$ inclus dans $H$, dont $M$ est un facteur de Levi ; on note $\F^H(M)$ l’ensemble des sous-groupes paraboliques de $G$ inclus dans $H$ et contenant $M$. Pour $P\in \F^G(M)$, on désignera souvent par $M_P$ l'unique facteur de Levi de $P$ contenant $M$. Soit $P$ un sous-groupe parabolique de $G$, on note $N_P$ son radical unipotent. On note 
\[W^{(G,M)}=\text{Norm}_{G(F)}(M)/M(F)\]
le groupe de Weyl relatif associé à $(G,M)$, avec $\text{Norm}_{G(F)}(M)$ le normalisateur de $M$ dans $G(F)$, ou bien le groupe des $F$-points du normalisateur de $M$ dans $G$.   

Soit $M$ un sous-groupe de Levi de $G$. La description dynamique des sous-groupes de Levi d'un group réductif (\cite[corollaires 6.10 (i)$\Leftrightarrow$(ii), 6.11]{SGA3}) dicte que $A_M$ est la partie déployée du tore $\Cent(M,G)^\circ=Z(M)^\circ$ et $M=\Cent(A_M,G)$. Tout sous-groupe de Levi est le centralisateur dans $G$ d'un sous-tore déployé. Une autre formulation (\cite[propositions 2.2.9, 2.2.11]{CGPbook}) est $M=\{g\in G\mid \lambda(t)g\lambda(t)^{-1}=g\}$ pour un $F$-morphisme $\lambda:\mathbb{G}_m\rightarrow A_{M}$ dont l'image rencontre $A_{M,G-\reg}$. De même si $P\in\P^G(M)$, alors $P= \{g\in G\mid \lim_{t\to 0}\lambda(t)g\lambda(t)^{-1}\,\,\text{existe}\}$ pour un $F$-morphisme $\lambda:\mathbb{G}_m\rightarrow A_{M}$ dont l'image rencontre $A_{M,G-\reg}$. Tout élément de $\P^G(M)$ est représenté ainsi. 

Soit $F$ un corps parfait. On note $\g_\ss$ (resp. $\g_\nilp$) l'ensemble des éléments semi-simples (resp. nilpotents) de $\g$. Pour tout $X \in \g$, on dispose de la décomposition de Jordan $X = X_{\ss} + X_{\nilp}$ où $X_{\ss}$ et $X_{\nilp}$ sont respectivement des éléments de $\g_\ss$ et de $\g_\nilp$, qui commutent. Le groupe $G_X\eqdef\Cent(X,G)$ est le centralisateur de $X$ dans $G$. Un élément $X$ de $\g(F)$ est dit $F$-elliptique s'il est semi-simple et $A_G = A_{G_X}$.

\subsection{Orbites induites}\label{subsec:orbitind} 
Soient $F$ un corps et $G$ un groupe réductif connexe sur $F$. 

La notion de l'orbite induite unipotente pour les groupes réductifs est introduite dans  \cite{LuszSpal79} par Lusztig et Spaltenstein, comme une généralisation d'une orbite de Richardson. Les propriétés fondamentales sont résumées dans la proposition suivante.

\begin{proposition}\label{chap3prop:indprop}
Soient $M$ un sous-groupe de Levi de $G$ et $X\in \m(F)$. On note $\o= (\Ad M)X$.
\begin{enumerate}
    \item Il existe une unique orbite $\Ind_M^G(\o)=\Ind_M^G(X)$ dans $\g$ pour l’action adjointe de $G$ telle que l’intersection
    \begin{equation}\label{chap3eq:indorbdef}
    \Ind_M^G(\o)\cap \left(\o+\mathfrak{n}_{P}\right)    
    \end{equation}
    soit un ouvert Zariski dense dans $\o+\mathfrak{n}_{P}$ pour tout $P$ sous-groupe parabolique de $G$ ayant $M$ comme facteur de Levi ;
    \item si $F$ est un corps parfait, alors l'orbite induite commute à la décomposition de Jordan, i.e.
    \begin{equation}\label{chap3eq:indJordan}
    \Ind_M^G(X)=\Ad(G)(\Ind_{M_{X_{\ss}}^\circ}^{G_{X_{\ss}}^\circ}(X))=\Ad(G)(X_{\ss}+\Ind_{M_{X_{\ss}}^\circ}^{G_{X_{\ss}}^\circ}(X_{\nilp})) ;   
    \end{equation}
    \item $\text{codim}_{\mathfrak{m}}(\o)=\text{codim}_{\g}(\Ind_M^G(\o))$ ;
    \item l'intersection \eqref{chap3eq:indorbdef} est également le lieu régulier de $\o+\mathfrak{n}_{P}$ ;
    \item soit $P$ un sous-groupe parabolique de $G$ ayant $M$ comme facteur de Levi, alors $\Ind_M^G(\o)\cap \left(\o+\mathfrak{n}_{P}\right)$ est une orbite dans $\p$ pour l'action adjointe de $P$ ;
    \item soit $P$ un sous-groupe parabolique de $G$ ayant $M$ comme facteur de Levi, soit $Y\in \Ind_M^G(\o)\cap (\o+\mathfrak{n}_{P})$, alors $G_{Y}^\circ\subseteq P$ ;
    \item soit $L$ un sous-group de Levi de $G$ contenant $M$, alors
    \[\Ind_{L}^G(\Ind_{M}^L(\o))=\Ind_{M}^G(\o).\]
\end{enumerate}
\end{proposition}
\begin{remark}
Des énoncés pour les groupes peuvent être trouvés dans \cite[section 1]{Hoff12}.    
\end{remark}
    
\begin{proof}
Dans la suite, nous supposerons que $F$ est séparablement clos, ce qui implique en particulier que $F$ est parfait. Nous déduirons le cas général à l'aide d'un argument simple de descente galoisienne.
\begin{enumerate}
    \item Le cas où $X$ serait nilpotent est déjà connu (\cite[section 7.1]{CM93}). L'unicité de l'orbite vérifiant que l'intersection \eqref{chap3eq:indorbdef} est un ouvert Zariski dense dans $\o+\n_P$ pour tout $P\in\P^G(M)$ est claire, on traitera uniquement l'existence. Nous allons prouver que la classe de $G$-conjugaison  
    \[\Ind_M^G(X)\eqdef\Ad(G)(\Ind_{M_{X_{\ss}}^\circ}^{G_{X_{\ss}}^\circ}(X))=\Ad(G)(X_{\ss}+\Ind_{M_{X_{\ss}}^\circ}^{G_{X_{\ss}}^\circ}(X_{\nilp}))\]
    vérifie que l'intersection \eqref{chap3eq:indorbdef} est un ouvert Zariski dense dans $\o+\n_P$ pour tout $P\in\P^G(M)$. Soit $P\in\P^G(M)$. On suppose en premier lieu que le sous-groupe dérivé de $G$ est simplement connexe. Remarquons qu'il en est de même pour tous les sous-groupes dérivés des sous-groupes de Levi de $G$. Si $\sigma$ est la partie semi-simple d'un élément d'une orbite $\o$ pour l'action adjointe d'un sous-groupe de Levi $L$ de $G$ dans $\mathfrak{l}$, on note
    \[\o_\sigma=\{\sigma+U\mid U\text{ nilpotent dans }\mathfrak{l}_\sigma, \sigma+U\in \o\}.\]
    Alors $\o_\sigma$ est une orbite dans $\mathfrak{l}_\sigma$ pour l'action adjointe de $L_\sigma$. En effet si $\o=(\Ad L)(\sigma+V)$ avec $V$ nilpotent dans $\mathfrak{l}_\sigma$ (autrement dit $\sigma+V$ est la décomposition de Jordan), on aura $\o_\sigma=(\Ad L_\sigma)(\sigma+V)$.

    Revenons sur la preuve de la proposition. Abrégeons $X_\ss$ en $\sigma$. Alors par la définition
    \[\Ind_M^G(X)_\sigma=\Ind_{M_\sigma}^{G_\sigma}(X')\]
    pour $X'\in \o_\sigma$. Considérons 
    \[I_\sigma=(\o_\sigma+\mathfrak{n}_{P_\sigma})\cap \Ind_M^G(X).\] 
    L'ensemble $I_\sigma$ est non-vide par la définition de l'induite d'une orbite nilpotente, et est un ouvert Zariski dense dans $\o_\sigma+\mathfrak{n}_{P_\sigma}$. \'{E}voquons à présent un lemme connu :
    \begin{lemma}\label{YDLioplem:Y+VtoJordan}
    Soit $Y\in \mathfrak{m}(F)$. Pour tout $U\in \mathfrak{n}_P(F)$, il existe $V\in \mathfrak{n}_{P_{Y_\ss}}(F)$ et $\delta\in N_P(F)$ tel que $(\Ad \delta)(Y+U)=Y+V$.
    \end{lemma}
    \begin{proof}
    Ceci est \cite[lemme 2.3]{Ch02a}. À noter que le lemme y est prouvé pour $F$ de caractéristique 0. Cependant, la démonstration est également valable pour tout corps parfait, car la décomposition de Jordan existe sur un corps parfait, et tout radical unipotent d'un groupe parabolique d'un groupe réductif sur un corps quelconque est lisse et déployé (\cite[propositions 2.1.8 (3), 2.1.10]{CGPbook}).     
    \end{proof}
    On déduit du lemme que $(\Ad N_P)I_\sigma= (\o_\sigma+\mathfrak{n}_{P})\cap \Ind_M^G(X)$, puis
    \[(\Ad P)I_\sigma= (\o+\mathfrak{n}_{P})\cap \Ind_M^G(X).\]
    Ce dernier est un ouvert Zariski dense dans $(\Ad P)(\o_\sigma+\mathfrak{n}_{P_\sigma})=\o+\mathfrak{n}_{P}$, ce qu'il fallait.

    Abandonnons maintenant l'hypothèse que le sous-groupe dérivés de $G$ est simplement connexe. \`{A} la place de quoi, on prend une z-extension de $G$ (\cite[section 1]{Kott82}) : 
    \[1\rightarrow \widetilde{T}\rightarrow\widetilde{G}\xrightarrow{\alpha} G\rightarrow 1\]
    avec $\widetilde{G}$ un groupe réductif dont le sous-groupe dérivé est simplement connexe et $\widetilde{T}$ un sous-tore central induit. La projection $\alpha$ identifie le système de racines de $\widetilde{G}$ à celui de $G$. On note $\widetilde{M}\eqdef \alpha^{-1}(M)$, qui est un sous-groupe de Levi de $\widetilde{G}$. Prenons $\widetilde{X}\in\widetilde{\g}$ tel que $\alpha(\widetilde{X})=X$. On a $\alpha(\widetilde{X}_\ss)=X_\ss$ et $\alpha(\widetilde{X}_\nilp)=X_\nilp$. Il va sans dire que 
    \begin{align*}    \alpha(\Ind_{\widetilde{M}}^{\widetilde{G}}(\widetilde{X}))&=\alpha\left(\Ad(\widetilde{G})(\widetilde{X}_{\ss}+\Ind_{\widetilde{M}_{\widetilde{X}_{\ss}}}^{\widetilde{G}_{\widetilde{X}_{\ss}}}(\widetilde{X}_{\nilp}))\right)\\
    &=\Ad(G)(X_{\ss}+\Ind_{M_{X_{\ss}}^\circ}^{G_{X_{\ss}}^\circ}(X_{\nilp}))    
    \end{align*}
    (la dernière égalité vient de \cite[lemme 3.1 (i)]{Kott82}) est une orbite vérifiant la propriété voulue, la preuve est terminée.

    \item Déjà établi dans la démonstration du point précédent.
    
    \item Prouvons d'abord l'énoncé de la codimension. Posons $\sigma=X_\ss$ et $\nu=X_\nilp$. Soit $\nu'\in \Ind_{M_{\sigma}^\circ}^{G_{\sigma}^\circ}(\nu)$, on est amené à prouver $\dim M_{\sigma+\nu}=\dim G_{\sigma+\nu'}$ par le point 1. Or $M_{\sigma+\nu}=(M_{\sigma})_{\nu}$ et $G_{\sigma+\nu'}=(G_{\sigma})_{\nu'}$, on est amené au cas où $\o$ serait nilpotent. On conclut alors par \cite[théorème 1.3 (a)]{LuszSpal79}.

    \item On sait d'ores et déjà que $\dim \Ind_M^G(\o)=\dim \o+2\dim \n_P$. En parallèle, en notant par $\overline{N_P}$ l'opposé de $N_P$ par rapport à $M$, on voit que si $Y\in \o+\n_P$ alors $\dim (\Ad G)Y=\dim (\Ad \overline{N_P}MN_P)Y\leq \dim (\Ad \overline{N_P})(\o+\n_P)\leq \dim \overline{N_P}+\dim (\o+\n_P)=\dim\o+2\dim \n_P$. Il en ressort $(\o+\n_P)\cap\Ind_M^G(\o)\subseteq (\o+\n_P)_\reg$. Or $\o+\n_P$ est inclus dans l'adhérence de $\Ind_M^G(\o)$, ce dernier est ainsi l'unique orbite de sa dimension qui rencontre $\o+\n_P$.
    
    \item La preuve de Lusztig-Spaltenstein \cite[théorème 1.3 (c)]{LuszSpal79} pour $\o$ nilpotent se généralise directement au cas général : soit $Y\in \Ind_P^G(X)\cap (X+\mathfrak{n}_P)$. On a
    \begin{align*}
    \dim ((\Ad P)Y)&= \dim P -\dim P_{Y}\\
    &
\geq \dim P -\dim G_{Y}=\dim M+\dim N_P-\dim M_{X}  \\
    &=\dim \o+\dim \mathfrak{n}_P =\dim (\o+\mathfrak{n}_P),
    \end{align*}
    ce qui prouve le résultat voulu puisque $Y$ est un élément arbitraire de $\Ind_P^G(X)\cap (X+\mathfrak{n}_P)$.
    \item Partant de la preuve du point précédent on a $\dim P_Y=\dim G_Y$ donc $G_Y^\circ \subseteq P$.
    \item Cela résulte de l'unicité de l'orbite induite.\qedhere
\end{enumerate}
\end{proof}

L'orbite  $\Ind_P^G(X)$ contient un $F$-point car $\mathfrak{n}_P(F)$ est Zariski dense dans $\mathfrak{n}_P$.

\begin{proposition}\label{YDLiopprop:whenXinIndMGX}Soient $M$ un sous-groupe de Levi de $G$ et $X \in \m(F)$. Les conditions suivantes sont équivalentes.
\begin{enumerate}
    \item $X\in \Ind_M^G(X)$.
    \item Pour tous $P\in\P^G(M)$ et $N\in N_P$, on a $X+N\in \Ind_M^G(X)$.
    \item Pour tous $P\in\P^G(M)$ et $N\in N_P$, il existe $n\in N_P$ tel que $(\Ad n)(X+N)=X$.
    \item $M_{X_\ss}=G_{X_\ss}$.
    \item $M_X=G_X$.
    \item Pour tout $P\in\P^G(M)$, le morphisme de schémas 
    \begin{align*}
    \n_P&\to \n_P\\
    N&\mapsto [N,X]
    \end{align*}
    est un isomorphisme, ici $[-,-]$ est le crochet de Lie sur $\g$.
    \item Pour tout $P\in\P^G(M)$, le morphisme de schémas 
    \begin{align*}
    N_P&\to \n_P\\
    n&\mapsto (\Ad n)X-X
    \end{align*}
    est un isomorphisme.
\end{enumerate}    
\end{proposition}
\begin{proof}Remarquons que pour tout groupe réductif $H$ et tout sous-groupe de Levi $L$, l'égalité $L=H$ a lieu si et seulement si $L^\circ=H^\circ$. 

(4) $\Leftrightarrow$ (5) : l'implication directe est triviale, quant à la réciproque on prend $P\in \P^{G_{X_\ss}^\circ}(M_{X_\ss}^\circ)$, l'adjonction par $X_\nilp\in \mathfrak{m}_{X_\ss}^\circ$ sur $\n_P$ admet 0 comme unique valeur propre, on a donc $M_X^\circ\not=G_X^\circ$ si $\n_Q\not=\{0\}$, soit $M_{X_\ss}^\circ\not=G_{X_\ss}^\circ$. 

(1) $\Leftrightarrow$ (5) : supposons que $X\in\Ind_M^G(X)$, alors le point 3 de la proposition \ref{chap3prop:indprop} nous assure que $\dim M_X=\dim G_X$, il s'ensuit $M_X^\circ=G_X^\circ$. Réciproquement, supposons que $M_X^\circ=G_X^\circ$, alors $M_{X_\ss}^\circ=G_{X_\ss}^\circ$ et l'équation \eqref{chap3eq:indJordan} nous assure donc que 
\[X\in (\Ad G)(X_\ss+\Ind_{M_{X_\ss}^\circ}^{G_{X_\ss}^\circ}(X_\nilp))=\Ind_M^G(X).\]

(1) $\Leftrightarrow$ (2) : la réciproque est triviale. Pour l'implication directe : soit $N\in N_P$, on cherche à établir que $X+N$ est conjugué à $X$ par un élément de $G$. Or d'après le lemme \ref{YDLioplem:Y+VtoJordan} on sait que $X+N$ et conjugué à un élément de la forme $X+V$, où $V\in \n_{P_{X_\ss}}$. Mais l'équivalence (1) $\Leftrightarrow$ (4) nous garantit que $M_{X_\ss}=G_{X_\ss}$. Ainsi $\n_{P_{X_\ss}}=0$. Ce qu'il fallait.

(2) $\Leftrightarrow$ (3) : la réciproque vient du fait que $\Ind_M^G(X)=(\Ad G)(\Ind_M^G(X)\cap (X+\n_P))$ d'après le point 5 de la proposition \ref{chap3prop:indprop}. Pour l'implication directe : soit $N\in \N_P$, on cherche à établir que $X+N$ et conjugué à $X$ par un élément de $N_P$. Or d'après le lemme \ref{YDLioplem:Y+VtoJordan} on sait que $X+N$ et conjugué par un élément de $N_P$ à un élément de la forme $X+V$, où $V\in \n_{P_{X_\ss}}$. Mais l'équivalence (1) $\Leftrightarrow$ (4) nous garantit que $M_{X_\ss}=G_{X_\ss}$. Ainsi $\n_{P_{X_\ss}}=0$. Ce qu'il fallait.

(5) $\Leftrightarrow$ (6) : prouvons d'abord la réciproque. Supposons que le morphisme de schémas $\n_P\to \n_P$ en question est un isomorphisme. Il est a fortiori injectif, ainsi $(\n_P)_X=\{0\}$. De même $(\overline{\n_P})_X=\{0\}$ avec $\overline{N_P}$ l'opposé de $N_P$. 
Puisque $X\in \m$, on en déduit $\g_X=(\overline{\n_P})_X\oplus \m_X\oplus (\n_P)_X=\m_X$, ainsi $M_X=G_X$. Prouvons ensuite l'implication directe. D'après l'équivalence (1) $\Leftrightarrow$ (5) on sait que $\g_X\subseteq\m$. L'endomorphisme $N\in\n_P\mapsto [N,X]\in\n_P$ est donc un automorphisme de $\n_P$.

(6) $\Leftrightarrow$ (7) : l'implication directe vient du fait que $N_P$ est unipotent et donc l'exponentielle $\exp :\n_P\to N_P$ est un isomorphisme. La réciproque est triviale.\qedhere 

\end{proof}

En particulier, l'équivalence (1) $\Leftrightarrow$ (4) dit que si $X$ est nilpotent, alors $X\in \Ind_M^G(X)$ si et seulement si $M=G$.

\begin{definition}\label{YDLiopdef:IndPG}
Soient $P$ un sous-groupe de Levi de $G$ et $X \in\mathfrak{p}(F)$. On définit
\[\Ind_P^G(X)\eqdef \Ind_M^G(\pi_{\p,\mathfrak{m}}(X)).\]
Ici $M$ est un facteur de Levi de $P$ et $\pi_{\p,\mathfrak{m}}:\mathfrak{p}\rightarrow \mathfrak{m}$ est la projection.

La définition de $\Ind_P^G(X)$ est indépendante du choix de $M$ car la sous-variété $(\Ad M)\pi_{\p,\mathfrak{m}}(X)+\n_P$ de $\g$ l'est. 
\end{definition}

\subsection{Pseudo-sous-groupe de Levi et son enveloppe de Levi}\label{subsec:pseudo-levi}
Le contenu de ce numéro devrait être connu, cependant, en l'absence d'une référence adéquate, nous présentons les détails. 

Soient $F$ un corps parfait et $G$ un groupe réductif connexe sur $F$. 

\begin{definition}
Nous appellerons pseudo-sous-groupe de Levi de $G$ tout sous-groupe de la forme $G_\sigma^\circ$ pour $\sigma\in\g_\ss(F)$.    
\end{definition}

Tout sous-groupe de Levi d'un pseudo-sous-groupe de Levi de $G$ est un pseudo-sous-groupe de Levi de $G$. En effet, si $\sigma\in\g_\ss(F)$, et $L'$ un sous-groupe de Levi de $G_\sigma^\circ$, alors $L'=(G_{\sigma}^\circ)_A$ pour $A$ un élément $G_\sigma^\circ$-régulier de $\mathfrak{a}_{L'}$, d'où $L'=G_{\sigma+A}^\circ$ un pseudo-sous-groupe de Levi.

On observe que pour tout groupe algébrique $H$, il y a $A_{H}=A_{H^\circ}$.

\begin{lemma}Pour tout pseudo-sous-groupe de Levi $L'$ de $G$, le plus petit élément (pour l'inclusion) dans l'ensemble de Levi de $G$ contenant $L'$ existe. On l'appelle l'enveloppe de Levi de $L'$ dans $G$, et le notera par $\envL(L';G)$. Plus précisément, $L\eqdef\Cent(A_{L'},G)$ est l'enveloppe de Levi de $L'$ dans $G$. On a $A_L=A_{L'}$, et $L$ est le seul sous-groupe de Levi de $G$ vérifiant cette égalité.
\end{lemma}
\begin{proof}
Considérons $L=\Cent(A_{L'},G)$, c'est un sous-groupe de Levi de $G$ puisque $A_{L'}$ est un sous-tore déployé. On a bien sûr $A_L=A_{L'}$. 

\'{E}crivons $L'=G_\sigma^\circ$ pour $\sigma\in\g_\ss(F)$, on voit que $L'=\Cent(A_{L'},G_{\sigma}^\circ)\subseteq L$. Supposons d'autre part que $M$ est un sous-groupe de Levi de $G$ contenant $L'$. Alors de $Z(L')\supseteq Z(M)$ on déduit $A_L=A_{L'}\supseteq A_M$, par conséquent $L=\Cent(A_L,G)\subseteq \Cent(A_M,G)=M$.
\end{proof}

Soit $\sigma\in \g_\ss(F)$. Si $M$ est un sous-groupe de Lie de $G$ dont l'algèbre de Levi contient $\sigma$, alors $M_\sigma^\circ$ est également un pseudo-sous-groupe de Levi de $G$ vu que $M_\sigma^\circ$ est un sous-groupe de Levi du pseudo-sous-groupe de Levi $G_\sigma^\circ$. Soit $L$ l'enveloppe de Levi de $M_\sigma^\circ$ dans $G$. On a par définition $M_\sigma^\circ \subseteq L\subseteq M$. Le lemme ci-dessus nous assure que $L$ est aussi l'enveloppe de Levi de $M_\sigma^\circ$ dans $M$. D'ailleurs les inclusions $M_\sigma^\circ \subseteq L_\sigma^\circ\subseteq L$ nous assure de plus que $L$ est l'enveloppe de Levi de $L_\sigma^\circ$ dans $G$.

Au regard de sa définition, on obtient la caractérisation suivante de l'enveloppe de Levi en termes de l'ellipticité.

\begin{lemma}\label{lem:envLell}Soient $\sigma\in \g_\ss(F)$ et $L'=G_\sigma^\circ$. Alors l'enveloppe de Levi $L$ de $G_\sigma^\circ$ dans $G$ est aussi le plus grand sous-groupe de Levi de $G$ dont l'algèbre de Lie contient $\sigma$ comme élément $F$-elliptique. 
\end{lemma}
\begin{proof}
On a d'abord $\sigma\in \mathfrak{l}'(F)\subseteq \mathfrak{l}(F)$. Puis $L'\subseteq L$ entraîne $G_\sigma^\circ= L_\sigma^\circ$, ainsi
$A_L=A_{L'}=A_{L_{\sigma}}$. En d'autres mots $L$ est un sous-groupe de Levi de $G$ dont l'algèbre de Lie contient $\sigma$ comme élément $F$-elliptique. Réciproquement si $M$ est un sous-groupe de Levi de $G$ dont l'algèbre de Lie contient $\sigma$ comme élément $F$-elliptique, on aura $Z(M_\sigma)\supseteq Z(L')$, d'où $A_{M_\sigma}\supseteq A_{L'}$, il vient de ce fait $M=\Cent(A_M,G)=\Cent(A_{M_\sigma},G)\subseteq\Cent(A_{L'},G)=\Cent(A_L,G)=L$. 
\end{proof}


\begin{lemma}
Soient $M$ un sous-groupe de Levi de $G$ et $\sigma\in\mathfrak{m}_\ss(F)$. 
\begin{enumerate}
    \item L'application naturelle
    \begin{align*}
        \F^G(M)&\longrightarrow \F^{G_\sigma^\circ}(M_\sigma^\circ)
        \\
        P& \longmapsto P_\sigma^\circ
    \end{align*}
    est surjective.
    \item  L'application naturelle
    \begin{align*}
        \{L\in\L^G(M)\mid A_L=A_{L_\sigma}\}&\longrightarrow \L^{G_\sigma^\circ}(M_\sigma^\circ)
        \\
        L& \longmapsto L_\sigma^\circ
    \end{align*}
    est bijective. Sa réciproque est l'application qui à un pseudo-sous-groupe de Levi associe son enveloppe de Levi.
\end{enumerate}
\end{lemma}
\begin{proof}
Supposons pour le moment que $\sigma$ est de plus un élément $F$-elliptique de $\m(F)$. Vis à vis de celle concernant les sous-groupes paraboliques, il suffit d'appliquer la description dynamique des sous-groupes paraboliques, en observant que $A_{M,G-\reg}=A_{M_\sigma,G_\sigma^\circ-\reg}$. Quant à celle concernant les sous-groupes de Levi, nous voyons que l'injectivité est claire. Pour la surjectivité soit $L'\in \L^{G_\sigma^\circ}(M_\sigma^\circ)$. On prend $L$ son enveloppe de Levi dans $G$, son algèbre de Lie contient $\sigma$. Alors $L_\sigma^\circ=\Cent(A_L,G)_\sigma^\circ=\Cent(A_L,G_\sigma^\circ)=\Cent(A_{L'},G_\sigma^\circ)=L'$. Mais on sait déjà que $L$ est l'enveloppe de Levi de $L_\sigma^\circ$ donc $A_{L}=A_{L_\sigma}$. Enfin $L\in \L^G(M)$ car il contient l'enveloppe de Levi de $M_\sigma^\circ$, qui vaut $M$. Le cas général se déduit en observant que si $M_1\subseteq M$ est le plus grand sous-groupe de Levi de $M$ dont l'algèbre de Lie contient $\sigma$ comme élément elliptique alors $M_{1\sigma}^\circ=M_\sigma^\circ$.  
\end{proof}


\begin{lemma}
L'application de l'ensemble des sous-groupes de Levi de $G_\sigma^\circ$ dans l'ensemble des sous-groupes de Levi de $G$ dont l'algèbre de Lie contient $\sigma$ comme élément $F$-elliptique, qui à un pseudo-sous-groupe de Levi associe son enveloppe de Levi, est une bijection. Sa réciproque est l'application qui à un sous-groupe de Levi associe la composante connexe du centralisateur de $\sigma$ dans ce sous-groupe de Levi.    
\end{lemma}
\begin{proof}
Cela dérive du raisonnement similaire à celui du lemme précédent.    
\end{proof}

On voit, compte tenu des discussions, que tout sous-groupe de Levi de $G_\sigma^\circ$ est de la forme $L_\sigma^\circ$ avec $L$ un sous-groupe de Levi de $G$ dont l'algèbre de Lie contient $\sigma$. Il est en de même pour les sous-groupes paraboliques.

\`{A} cet égard, on définit, pour $L$ un sous-groupe de Levi de $G$ dont l'algèbre de Lie contient $\sigma$, le groupe de Weyl relatif associé à $(G,L,\sigma)$ par
\[W^{(G,L,\sigma)}\eqdef W^{(G_\sigma^\circ,L_\sigma^\circ)}.\]

\subsection{Spécificités des groupes généraux linéaires ainsi que leurs formes intérieures}

On voudrait travailler dans la suite exclusivement avec les groupes généraux linéaires ainsi que leurs formes intérieures. 

Pour $S$ un anneau commutatif et $A$ une $S
$-algèbre à gauche. On note $A^{\text{op}}$ l'opposé de $A$, qui est une $S$-algèbre à droite. Il y a une structure de $A^{\text{op}}\otimes_SA$-module à droite sur $A$ donnée par $a\cdot (b\otimes_S b')=bab'$ avec $a,b'\in A$ et $b\in A^{\text{op}}$. Une $S$-algèbre $A$ à gauche est dite séparable si $A$ est un module projectif sur $A^{\text{op}}\otimes_SA$.

\begin{definition}
On dit qu'un groupe réductif $G$ est du type GL (sur $F$) s'il est le groupe des unités d'une algèbre séparable sur $F$.

De façon équivalente, un groupe réductif est du type GL (sur $F$) s'il est le groupe des unités d'une algèbre semi-simple de dimension finie sur $F$, ou encore, par le théorème d'Artin-Wedderburn, s'il est le groupe des unités d'un produit fini de restrictions des scalaires d'algèbres simples centrales sur des extensions finies de $F$.
\end{definition}

Donnons des propriétés classiques des groupes du type GL.

\begin{proposition}
~{}\begin{enumerate}
    \item Un produit fini de groupes du type GL l'est aussi.
    \item Tout sous-groupe de Levi d'un groupe du type GL l'est aussi.
    \item Toute forme intérieure d'un groupe du type GL l'est aussi.
    \item Toute extension des scalaires par un sur-corps d'un groupe du type GL l'est aussi (sur le sur-corps).
    \item Soit $G$ un groupe du type GL et $\g$ son algèbre de Lie. On dispose de $G \hookrightarrow \g$ une inclusion canonique $G$-équivariante ($G$ agit par conjugaison sur les deux espaces).
    \item Soient $G$ un groupe du type GL et $X\in\g_\ss$, alors $G_X(F)$ est un groupe du type GL.
    \item Soient $G$ un groupe du type GL et $X\in\g(F)$, alors $G_X$ est géométriquement connexe.
    \item Dans l'algèbre des $F$-points de l'algèbre de Lie d'un groupe $G$ du type GL, les notions de $G(F)$-conjugaison et de $G(\overline{F})$-conjugaison coïncident.  
\end{enumerate}
\end{proposition}
\begin{proof}
Les points 1,2,4 et 5 ressortent de la définition. Le point 3 vient de la théorie des corps de classes ; le point 6 vient de \cite[section 12.7]{Pi82} ; le point 7 vient du fait que $G_X$ est un produit semi-directe d'un groupe isomorphe en tant que schéma à une puissance du groupe additif, par un groupe du type GL, puis une puissance du groupe additif est géométriquement connexe et un groupe du type GL est aussi géométriquement connexe, car son extension à $\overline{F}$ est un produit fini de groupes généraux linéaires sur $\overline{F}$ ; le point 8 est une conséquence du théorème de Hilbert 90, conjugué avec le lemme de Shapiro. 
\end{proof}

\begin{proposition}[{{\cite[proposition A.13]{YDL23b}}}]\label{prop:condition(H?)}Soit $G$ un groupe du type GL sur un corps parfait $F$.
\begin{enumerate}
    \item Toute orbite est l'induite d'une orbite semi-simple elliptique. C'est-à-dire que pour tout $X\in \g(F)$, il existe $M$ un sous-groupe de Levi de $G$ dont l'algèbre de Lie contient $X_\ss$ comme élément elliptique, tel que $X\in\Ind_P^G(X_\ss)$ avec $P\in \P^G(M)$.
    \item Soient $M_1$ (resp. $M_2$) un sous-groupe de Levi de $G$ et $X_1$ (resp. $X_2$) un élément ellitique dans $\m_{1}(F)$ (resp. $\m_2(F)$), tels que $\Ind_{M_1}^G(X_1)=\Ind_{M_2}^G(X_2)$, alors il existe $g\in G(F)$ tel que $M_1=gM_2g^{-1}$ et $X_1=gX_2g^{-1}$.
\end{enumerate}
\end{proposition}

\begin{remark}
Vis-à-vis de chacune des ces propriétés, on peut trouver des contrexemples avec $G$ un groupe réductif non du type GL. Pour la facilité des lecteurs et lectrices nous donnons des contrexemples ici. Il est connu que, au moins pour les groupes de Lie classiques, les orbites nilpotentes sont classifiées par des partitions d'un entier, assujetties à une certaine condition de parité (cf. \cite[chapitre 5]{CM93}). Puis l'induite d'une orbite nilpotente peut être explicitée en termes de partition (cf. \textit{Ibid.} chapitre 7). Il n'est donc par difficile de fabriquer des contrexemples : (\textit{Ibid.} example 7.3.6.) dans $\mathfrak{sp}_6$ aucune des orbites nilpotentes correspondant aux partitions $[1^6]$ ou $[2,1^4]$ n'est induite de $0$ ; dans $\mathfrak{so}_7$ aucune des orbites nilpotentes correspondant aux partitions $[1^7]$ ou $[2^2,1^3]$ n'est induite de $0$ ; dans $\mathfrak{so}_8$ aucune des orbites nilpotentes correspondant aux partitions $[1^8]$, $[2^2,1^4]$ ou $[3,2^2,1]$ n'est induite de $0$. Finalement dans $\mathfrak{so}_{12}$, on prend $L_1$ un sous-groupe de Levi isomorphe à $\GL_2\times \text{SO}_{
8}$ et $L_2$ un sous-groupe de Levi isomorphe à $\GL_1\times \text{SO}_{
10}$, quoique non conjugués, ils donnent la même orbite induite $\Ind_{L_1}^G(0)=\Ind_{L_2}^G(0)$, à savoir $[3^2,1^6]$. Si l'on veut des points rationnels, il convient de prendre un groupe quasi-déployé du type correspondant.
\end{remark}

Désormais $G$ désigne un groupe du type GL sur un corps parfait $F$.

\subsection{Ensemble de Richardson et de Lustig-Spaltenstein généralisé}
Soit $X\in\g(F)$.

Posons
\begin{equation}\label{eq:cLSGX}
\cLS^G(X)=\{P\subseteq G\text{ sous-}\text{groupe parabolique} \mid X\in \p, X\in\Ind_P^G(X)\},      
\end{equation}
c'est un ensemble non-vide puisque $G\in\cLS^G(X)$. Les éléments de $\cLS^G(X)$ sont appelés les sous-groupes de Lusztig-Spaltenstein généralisés de $X$.  Notons que si $P\in \cLS^G(X)$ alors $X_\nilp\in \p$ et il existe automatiquement $M$ un facteur de Levi de $P$ tel que $X_\ss\in \m$. Puis regardons l'ensemble
\begin{equation}\label{eq:cRG'X}
\cR^G(X)'=\{P\subseteq G\text{ sous-}\text{groupe parabolique} \mid X_\nilp\in \n_P, X\in \p, X\in\Ind_P^G(X_\ss)\},    
\end{equation}
il est pareillement non-vide d'après le point 1 de la proposition \ref{prop:condition(H?)}. Une première description de ces ensembles est donnée par le lemme suivant :
\begin{lemma}\label{lem:LSR'ss}
Soit $P$ un sous-groupe parabolique de $G$, avec $X\in \p$. Alors $P\in \cLS^G(X)$ si et seulement si $P_{X_\ss}\in \cLS^{G_{X_\ss}}(X_\nilp)$ ; $P\in \cR^G(X)'$ si et seulement si $P_{X_\ss}\in \cR^{G_{X_\ss}}(X_\nilp)'$.
\end{lemma}
\begin{proof}
Effectivement, un sous-groupe parabolique $P$ de $G$ vérifiant l'hypothèse de l'énoncé est dans $\cLS^G(X)$ si et seulement si, grâce à l'équation \eqref{chap3eq:indJordan}, il existe $g\in G$ tel que $X=X_\ss+X_\nilp\in (\Ad g)(X_\ss+\Ind_{P_{X_{\ss}}}^{G_{X_{\ss}}}(X_{\nilp}))$. L'unicité de la décomposition de Jordan nous garantit que $g\in G_{X_\ss}$, ainsi $P$ est dans $\cLS^G(X)$ si et seulement si $X\in\p$ et $X_\nilp\in \Ind_{P_{X_{\ss}}}^{G_{X_{\ss}}}(X_{\nilp})$, soit $X_{\nilp}\in \p_{X_\ss}$ et $P_{X_\ss}\in \cLS^{G_{X_\ss}}(X_\nilp)$. L'argument pour $\cR^G(X)'$ est similaire.
\end{proof}

Un corollaire immédiat du lemme est que $\cR^G(X)'$ est un sous-ensemble de $\cLS^G(X)$ : partant du lemme, si $P\in \cR^G(X)'$ on aura $X_\nilp \in \mathfrak{n}_{P_{X_\ss}}$ donc $\pi_{\p_{X_\ss},\mathfrak{m}_{X_\ss}}(X_\nilp)=0$ d'où $P\in\cLS^G(X)$, ce qu'il fallait.

\begin{lemma}\label{lem:ellinRichardson}
Soient $P\in \cR^G(X)'$ et $M$ un facteur de Levi de $P$ dont l'algèbre de Lie contient $X_\ss$. Alors si $L=\envL(M_{X_\ss};G)$, et $Q$ l'unique élément de $\P^G(L)$ qui est contenu dans $P$, on aura également $Q\in \cR^G(X)'$.
\end{lemma}
\begin{proof}
De $M_{X_\ss}\subseteq L\subseteq M$ il découle que $M_{X_\ss}=L_{X_\ss}$. Pour cette raison $X_\nilp\in \Ind_{Q_{X_\ss}}^{G_{X_\ss}}(0)$. On a ensuite $X_\nilp\in \n_P\subseteq \n_Q$, puis $X_\ss\in \mathfrak{m}_{X_\ss}=\mathfrak{l}_{X_\ss}\subseteq \mathfrak{l}$, ce qu'il fallait.
\end{proof}

Posons en l'occurrence
\begin{equation}\label{eq:cRGX}\cR^G(X)=\{P\in \cR^G(X)' \mid P\text{ un élément minimal (pour l'inclusion) dans }\cR^G(X)'\},
\end{equation}
c'est un ensemble non-vide. Les éléments de $\cR^G(X)$ sont appelés les sous-groupes de Richardson généralisés de $X$. 

\begin{lemma}\label{lem:Rlevienvlop}Soit $P\in \cR^G(X)$ et $M$ un facteur de Levi de $P$ dont l'algèbre de Lie contient $X_\ss$. Alors $M=\envL(M_{X_\ss};G)$. Inversement, soit $P\in \cR^G{}'(X)$ et $M$ un facteur de Levi de $P$ dont l'algèbre de Lie contient $X_\ss$, si $M=\envL(M_{X_\ss};G)$, alors $P\in \cR^G(X)$.
\end{lemma}
\begin{proof}
On raisonne de la façon suivante quant à l'implication directe : eu égard au lemme précédent, $P$ est également un sous-groupe parabolique ayant $L\eqdef\envL(M_{X_\ss};G)$, comme facteur de Levi, $L$ et $M$ sont donc conjugués. Or $M_{X_\ss}\subseteq L\subseteq M$, il s'ensuit donc $L=M$.

Inversement, soit $Q\in\cR^G(X)$ avec $Q\subseteq P$, soit $L$ un facteur de Levi de $Q$ dont l'algèbre de Lie contient $X_\ss$. Soit $M'$ le facteur de Levi de $P$ qui contient $L$, son algèbre de Lie contient bien sûr $X_\ss$. Les groupes $M_{X_\ss}'$ et $M_{X_\ss}$ étant deux facteurs de Levi de $P_{X_\ss}$, il existe donc $p\in P_{X_\ss}(F)$ tel que $M_{X_\ss}'=(\Ad p)M_{X_\ss}$. Il s'ensuit $A_{M'}=A_{M_{X_\ss}'}=(\Ad p)A_{M_{X_\ss}}=(\Ad p)A_M$ et $M'=\Cent(A_{M'},G)=\Cent((\Ad p)A_M,G)=(\Ad p)M$, on peut par conséquent supposer que $M'=M$. Du lemme \ref{lem:LSR'ss} et le point 2 de la proposition \ref{prop:condition(H?)}, on sait qu'il existe $g\in G_{X_\ss}(F)$ tel que $L_{X_\ss}=(\Ad g)M_{X_{\ss}}$. En outre $L$ est l'enveloppe de Levi de $L_{X_\ss}$ selon la première partie de la démonstration, il en résulte que $A_L=A_{L_{X_\ss}}=(\Ad g)A_{M_{X_{\ss}}}=(\Ad g)A_M$ puis $L=(\Ad g)M$. Comme $L\subseteq M$ on en déduit que $L=M$ et $P=Q\in \cR^G(X)$.
\end{proof}

\begin{corollary}\label{coro:elementisellinRich}
Soit $P\in \cR^G(X)$ et $M$ un facteur de Levi de $P$ dont l'algèbre de Lie contient $X_\ss$. Alors $X_\ss$ est $F$-elliptique dans $\m(F)$.
\end{corollary}
\begin{proof}
On combine le lemme précédent et le lemme \ref{lem:envLell}.
\end{proof}

Soit $P$ un sous-groupe parabolique de $G$, avec $X_\ss\in \p$. Un autre corollaire immédiat du lemme précédent est que $P\in \cR^G(X)$ implique $P_{X_\ss}\in \cR^{G_{X_\ss}}(X_\nilp)$. La réciproque est fausse.

\begin{corollary}\label{coro:LeviRssconj}Soient $P\in \cR^G(X)$ et $M_1,M_2$ deux facteurs de Levi de $P$, tels que l'algèbre de Lie de chacun contient $X_\ss$. Alors $M_1$ et $M_2$ sont conjugués par un élément de $P_{X_\ss}(F)$.
\end{corollary}
\begin{proof}Puisque $M_{1,X_\ss}$ et $M_{2,X_\ss}$ sont deux facteurs de Levi de $P_{X_\ss}$, ils sont conjugués par un élément de $P_{X_\ss}(F)$, on a l'existence de $p\in P_{X_\ss}(F)$ tel que $M_{1,X_\ss}=(\Ad p)M_{2,X_\ss}$. On tire ensuite, du lemme \ref{lem:Rlevienvlop}, que  $A_{M_1}=A_{M_{1,X_\ss}}=(\Ad p)A_{M_{2,X_\ss}}=(\Ad p)A_{M_2}$. Il vient $M_1=(\Ad p)M_2$.
\end{proof}

Soient $P_1$ et $P_2$ deux sous-groupes paraboliques de $G$. On dit que $P_1$ et $P_2$ sont associés par $g\in G$ s'il existe $M_1$ et $M_2$ des facteurs de Levi respectifs de $P_1$ et $P_2$ tels que $M_1=(\Ad g)M_2$.

\begin{proposition}\label{prop:Rasso}Deux élément de $\cR^G(X)$ sont associés par un élément de $G_{X_\ss}(F)$. Inversement, un sous-groupe parabolique dans $\cR^G(X)'$ associé par un élément de $G_{X_\ss}(F)$ à un sous-groupe parabolique dans $\cR^G(X)$ est lui-même dans $\cR^G(X)$.
\end{proposition}
\begin{proof}Commençons par l'implication directe. Soient $P_1$ (resp. $P_2$) un élément de $\cR^G(X)$, et $M_1$ (resp. $M_2$) un facteur de Levi dont l'algèbre de Lie contient $X_\ss$. Par  le point 2 de la proposition \ref{prop:condition(H?)} on sait qu'il existe $g\in G_{X_\ss}(F)$ tel que $M_{1,X_\ss}=(\Ad g)M_{2,X_\ss}$. On tire ensuite, du lemme \ref{lem:Rlevienvlop}, que  $A_{M_1}=A_{M_{1,X_\ss}}=(\Ad g)A_{M_{2,X_\ss}}=(\Ad g)A_{M_2}$. Il vient $M_1=(\Ad g)M_2$. 

Inversement, soient $P\in R^G(X)'$ associé par un élément de $G_{X_\ss}(F)$ à un sous-groupe parabolique dans $\cR^G(X)$, et $M$ un facteur de Levi de $P$ dont l'algèbre de Lie contient $X_\ss$. Alors en argumentant de la même manière que dans la preuve du lemme précédent on voit que $M$ est l'enveloppe de Levi de $M_{X_\ss}$, par conséquent $P\in \cR^G(X)$. La preuve est accomplie.
\end{proof}

\begin{proposition}\label{prop:LS->R}Un sous-groupe parabolique de $G$ appartient à $\cLS^G(X)$ si et seulement s'il contient un élément de $\cR^G(X)$.
\end{proposition}

\begin{proof}Montrons au premier abord la réciproque. Soient $P\in\cR^G(X)$ et $Q\supseteq P$. On prend $M$ un facteur de Levi de $P$ et $L$ un facteur de Levi de $Q$ tels que $L\supseteq M$. Par le lemme \ref{lem:LSR'ss} on peut supposer que $X$ est nilpotent. Notons $\pi :\n_{P}\rightarrow \n_{P \cap L}^{L}$ la projection, avec $\n_{P \cap L}^{L}\eqdef \n_{P}\cap\mathfrak{l}$. Alors $\pi(\Ind_{P}^{G}(0)\cap \n_{P})$ est un ouvert dense de Zariski dans $\n_{P\cap L}^{L}$, cette image intersecte de ce fait $\Ind_{P \cap L}^{L}(0)\cap \n_{P\cap L}^{L}$ dans un ouvert dense de Zariski. Or $\Ind_{P}^{G}(0)\cap \n_{P}=(\Ad P)X$ par le point 5 de la proposition \ref{chap3prop:indprop}, quitte à remplaçer $X$ par un $P$-conjugué on peut donc supposer que $\pi_{\mathfrak{q},\mathfrak{l}}(X)\in \Ind_{P \cap L}^{L}(0)$. Il résulte maintenant de la transitivité de l’orbite induite que $X\in \Ind_Q^G(X)$, par voie de conséquence $Q\in\cLS^G(X)$.

Montrons dans un deuxième temps le sens direct. Soient $Q\in\cLS^G(X)$ et $L$ un facteur de Levi dont l'algèbre de Lie contient $X_\ss$. Constatons que la projection $\pi_{\mathfrak{q}_{X_\ss},\mathfrak{l}_{X_\ss}}(X)$ dans l'algèbre réductive $\mathfrak{l}_{X_\ss}$ vaut $X_\ss+\pi_{\mathfrak{q}_{X_\ss},\mathfrak{l}_{X_\ss}}(X_\nilp)$, cette expression est également sa décomposition de Jordan, car l'action adjointe par $L_{X_\ss}$ laisse stable la décomposition $\mathfrak{q}_{X_\ss}=\mathfrak{l}_{X_\ss}\oplus\n_{Q_{X_\ss}}$. De là et le lemme \ref{lem:LSR'ss}, on peut supposer que $X$ est nilpotent. Le point 1 de la proposition \ref{prop:condition(H?)} nous garantit l'existence d'un sous-groupe parabolique $P'$ de $L$ tel que $\pi_{\mathfrak{q},\mathfrak{l}}(X)\in \Ind_{P'}^L(0)$. Soit $P$ l'unique sous-groupe parabolique de $G$ qui est contenu dans $Q$ et dont l'intersection avec $L$ vaut $P'$. Il découle à nouveau de la transitivité de l'orbite induite que $X\in \Ind_P^G(0)$, par voie de conséquence $P\in\cR^G(X)'$.
\end{proof}

Les ensembles $\cLS^G(X)$ et $\cR^G(X)$ ont été introduits par Chaudouard dans le cas $X$ nilpotent (\cite{Ch17}). \`{A} noter que son article est rédigé dans le cadre de $G=\GL_{n,F}$, cela dit il n'y a guère d'obstable pour généraliser sa théorie nilpotente à tous les groupes du type GL. Lorsque l'on a besoin des résultats de son article, mais pour un groupe du type GL, on fera habituellement mention de celui-ci de manière concise, sans commentaire, tout en apportant les ajustements nécessaires lorsque la situation les réclame.

\section{Nouvelle définition d'une intégrale orbitale pondérée}\label{sec:2nouvdef}

On donne un bref aperçu du numéro \ref{sec:2nouvdef}. On fixe ici $M_0$ un sous-groupe de Levi minimal de $G$. Soient $L$ un sous-groupe de Levi de $G$ contenant $M_0$, $Q$ un sous-groupe parabolique de $G$ contenant $L$, et $\o$ une $L(F)$-orbite dans l'algèbre de Lie $\mathfrak{l}(F)$. L'objectif de ce numéro est de définir l'intégrale orbitale pondérée $J_L^Q(\o,f)$ pour une fonction test $f$ sur $\g(F)$ comme 
\[
J_L^Q(\o,f)\eqdef |D^{\g}(X)|_{F}^{1/2} \int_{G_{X}(F)\backslash G(F)} f((\Ad g^{-1})X)v_{(\Ad w)L,X}^{(\Ad w)Q}(g)\,dg,
\]
(définition \ref{def:NIOPl}). Ici, $D^{\g}(X)$ est le discriminant de Weyl de $X$, rappelé dans la sous-section \ref{YDLiopsubsec:normalizationmeasure}, et $|-|_F$ désigne la valeur absolue usuelle sur $F$. Ensuite, $X$ est le représentant standard de $\Ind_L^G(\o)(F)$, dont la définition sera expliquée dans la section \ref{subsec:ssssschoix}. Enfin, $v_{(\Ad w)L,X}^{(\Ad w)Q}(g)\in \C$ est le poids dérivé de la théorie des $(G,M)$-familles, rappelée dans la sous-section \ref{YDLiopsubsec:GMfamily}. Nous vérifions dans le lemme \ref{lem:poidsinvnorm} que la fonction $v_{(\Ad w)L,X}^{(\Ad w)Q}$ est invariante à gauche par $G_{X}(F)$, ce qui légitime la définition ci-dessus.

\subsection{Paramétrage de \texorpdfstring{$\cR^G(X)$}{RG(X)}}

Dorénavant, soit $X\in\g(F)$. 

\begin{proposition}Soit $M$ un facteur de Levi d'un élément de $\cR^G(X)$ dont l'algèbre de Lie contient $X_\ss$. Alors
\begin{enumerate}
    \item Pour tout $P \in \P^G(M)$, il existe un unique élément $\widetilde{P} \in \cR^G(X)$ tel que $P$ et $\widetilde{P}$ soient conjugués sous $G_{X_\ss}$. Ils sont en fait conjugués sous $G_{X_\ss}(F)$.
    \item L'application
    \begin{align}\label{R}
    \begin{split}
        \P^G(M) &\longrightarrow \cR^G(X)
        \\
        P& \longmapsto \widetilde{P}
    \end{split}
    \end{align}
    est surjective.
    \item On a $\text{Norm}_{G_{X_\ss}(F)}(M)=\text{Norm}_{G_{X_\ss}(F)}(M_{X_\ss})$
. Les fibres de l’application \eqref{R} sont naturellement des torseurs sous le groupe de Weyl relatif 
    \[\text{Norm}_{G_{X_\ss}(F)}(M_{X_\ss})/M_{X_{\ss}}(F).\]
\end{enumerate}
\end{proposition}

\begin{remark}
Comme $\P^G(M)$ est fini, on obtient en particulier que l'ensemble $\cR^G(X)$, et donc aussi $\cLS^G(X)$, est fini.
\end{remark}

\begin{proof}Commençons par montrer l'unicité dans l'assertion 1 : autrement dit deux éléments de $\cR^G(X)$ qui sont conjugués sous $G_{X_\ss}$ sont égaux. Soient $P_1$ et $P_2$ dans $\cR^G(X)$, et $g\in G_{X_\ss}$ tel que $P_1=(\Ad g)P_2$. Posons $M_{P_1}$ et $M_{P_2}$ des facteurs de Levi de $P_1$ et $P_2$ tels que $M_{P_1}=(\Ad g)M_{P_2}$, on suppose que $\mathfrak{m}_{P_1}$ contient $X_\ss$, alors il en est de même pour $\mathfrak{m}_{P_2}$. On sait par définition et proposition \ref{chap3prop:indprop} que
\[X\in(\Ad P_i)X=(\Ad G) X\cap ((\Ad M_{P_i})X_\ss+\mathfrak{n}_{P_i}),\]
pour $i=1,2$.  En conjuguant par $g$ aux deux côtés de l'égalité précédente pour $i=2$ on trouve 
\[(\Ad g) X\in (\Ad G) X\cap ((\Ad M_{P_1})X_\ss+\mathfrak{n}_{P_1})=(\Ad P_1)X.\]
Il existe alors $p_1\in P_1$ tel que $(\Ad g)X=(\Ad p_1)X$, soit $p_1^{-1}g\in G_X$. Or, toujours par la proposition \ref{chap3prop:indprop}, $G_X\subseteq P_1$, on en déduit $g\in P_1$ d'où 
\[P_2=(\Ad g^{-1})P_1=P_1.\]

Montrons ensuite l'existence dans l'assertion 1, c'est-à-dire pour tout $P \in \P^G(M)$, il existe un élément de $\cR^G(X)$ qui est conjugué par $G_{X_\ss}(F)$ à $P$. Soit $P \in \P^G(M)$. On a $X\in \Ind_P^G(X_\ss)(F)=(\Ad G(F))(X_\ss+\mathfrak{n}_{P_{X_\ss}})_{G_{X_\ss}-\reg}(F)$. L'unicité de la décomposition de Jordan (de $X$) nous garantit donc l'existence de $g\in G_{X_\ss}(F)$ tel que $(\Ad g)X\in (X_\ss +\mathfrak{n}_{P_{X_\ss}})_{G_{X_\ss}-\reg}\subseteq \mathfrak{p}$. On prétend primo que $\widetilde{P}\eqdef(\Ad g^{-1})P\in \cR^G(X)'$. En effet $X_\nilp\in (\Ad g^{-1})\n_{P}=\n_{\widetilde{P}}$, et $\widetilde{M}\eqdef(\Ad g^{-1})M$ est un facteur de Levi de $\widetilde{P}$ dont l'algèbre de Lie contient $(\Ad g^{-1})X_\ss=X_\ss$. Enfin $X\in (\Ad g^{-1})(X_\ss +\mathfrak{n}_{P_{X_\ss}})_{G_{X_\ss}-\reg}=(X_\ss +\mathfrak{n}_{\widetilde{P}_{X_\ss}})_{G_{X_\ss}-\reg}\subseteq\Ind_{\widetilde{P}}^G(X_\ss)$. On prétend secundo que $\widetilde{P}\in \cR^G{}(X)$, en effet, la rationalité de l'élément $g$ ci-dessus et la première partie du lemme \ref{lem:Rlevienvlop} nous assurent que $\widetilde{M}=\envL(\widetilde{M}_{X_\ss};G)$, après la deuxième partie du même lemme nous donne $\widetilde{P}\in \cR^G{}(X)$.

Continuons par la surjectivité de l'assertion 2. Fixons $P\in \cR^G(X)$ ayant $M$ comme facteur de Levi. Soit $Q\in \cR^G(X)$. D'après la proposition \ref{prop:Rasso}, les sous-groupes paraboliques $P$ et $Q$ sont associés par un élément de $G_{X_\ss}(F)$, on en déduit l'existence de $g \in G_{X_\ss}(F)$ tel que $Q \in \P((\Ad g)M)$ c'est-à-dire $(\Ad g^{-1})Q \in \P(M)$ ce qui montre la surjectivité.

Tournons-nous vers l'égalité d'ensembles
\begin{equation}\label{eq:pW=W}
\text{Norm}_{G_{X_\ss}(F)}(M)=\text{Norm}_{G_{X_\ss}(F)}(M_{X_\ss}).    
\end{equation}
Amorçons par l'inclusion $\subseteq$ : soit $g\in G_{X_\ss}(F)$ tel que $(\Ad g)M=M$. Nous avons, en vertu du lemme \ref{lem:Rlevienvlop}, $(\Ad g)A_{M_{X_\ss}}=(\Ad g)A_{M}=A_M=A_{M_{X_\ss}}$, d'où $(\Ad g)M_{X_\ss}=\Cent((\Ad g)A_{M_{X_\ss}},G_{X_\ss})=\Cent(A_{M_{X_\ss}},G_{X_\ss})=M_{X_\ss}$. L'inclusion inverse dérive du raisonnement similaire.

Calculons une fibre de \eqref{R}. Soit $P\in\P^G(M)$. La fibre de l’image de $P$ par
\eqref{R} est formée des éléments de $\P^G(M)$ qui sont $G_{X_\ss}(F)$-conjugués à $P$. Soit $g\in G_{X_\ss}(F)/P_{X_\ss}(F)$ tel que $Q=(\Ad g)P\in\P^G(M)$. Identifions $g$ à un représentant dans $G_{X_\ss}(F)$. Les groupes $(\Ad g^{-1})M$ et $M$ sont deux facteurs de Levi de $P$, tels que l'algèbre de Lie de chacun contient $X_\ss$. Ils sont donc conjugués par un élément de $P_{X_\ss}(F)$ d'après le corollaire \ref{coro:LeviRssconj}, quitte à changer le représentant $g$, on peut supposer que $(\Ad g^{-1})M=M$, c'est-à-dire $g\in \text{Norm}_{G_{X_\ss}(F)}(M)$. Comme $\text{Norm}_{P_{X_\ss}(F)}(M)=\text{Norm}_{P(F)}(M)\cap P_{X_\ss}(F)=M_{X_\ss}(F)$, une fibre de \eqref{R} est bien un torseur sous $\text{Norm}_{G_{X_\ss}(F)}(M)/M_{X_{\ss}}(F)$. 
\end{proof}


\subsection{Paramétrage de \texorpdfstring{$\cLS^G(X)$}{LSG(X)}}

\begin{proposition}Soit $M$ un facteur de Levi d'un élément de $\cR^G(X)$ dont l'algèbre de Lie contient $X_\ss$. Il existe une unique application
\begin{align}\label{LS}
    \begin{split}
        \F^G(M) &\longrightarrow \cLS^G(X)
        \\
        Q& \longmapsto \widetilde{Q}
    \end{split}
    \end{align}
telle que
\begin{enumerate}
    \item elle co\"{i}ncide sur $\P^G(M)$ avec \eqref{R} ;
    \item un élément de $\F^G(M)$ est $G_{X_\ss}(F)$-conjugué à son image ;
    \item elle préserve les inclusions ;
    \item elle est surjective ;
    \item le groupe $\text{Norm}_{G_{X_\ss(F)}}(M_{X_\ss})$ agit transitivement sur les fibres de \eqref{LS}. Le stabilisateur de $Q\in \F^G(M)$ est $\text{Norm}_{Q_{X_\ss(F)}}(M_{X_\ss})$.
\end{enumerate}
\end{proposition}

\begin{proof}
Montrons d'abord l'unicité d'une telle application. Soit $Q \in \F^G(M)$. Soit $P \in \P^G(M)$ tels que $P \subseteq Q$. Par la propriété 1, $\widetilde{P}$ est l'image de $P$ par l'application \eqref{R}. Donc il existe un unique élément $g \in G_{X_\ss}(F)/P_{X_\ss}(F)$ tel que $(\Ad g)P = \widetilde{P}$. Par conséquent, on  $\widetilde{P}\subseteq (\Ad g)Q$. D'autre part, par la propriété 3 on a $\widetilde{P}\subseteq \widetilde{Q}$, et par la propriété 2 le groupe $\widetilde{Q}$ est conjugué à $(\Ad g)Q$. Comme $\widetilde{Q}$ et $(\Ad g)Q$ contiennent tous deux $\widetilde{P}$, on a nécessairement
\begin{equation}\label{appdef:LS}
\widetilde{Q}=(\Ad g) Q.    
\end{equation}

Réciproquement, l'équation (\ref{appdef:LS}) permet de définir l'application \eqref{LS} : déjà $\widetilde{Q}$ est dans $\cLS^G(X)$ car il contient $\widetilde{P}$ un élément de $\cR^G(X)$, en second lieu d'après le lemme \ref{lem:2.10.2} ci-dessous, le membre de droite de (\ref{appdef:LS}) ne dépend pas du choix de $P \in \P(M)$ tel que $P \subseteq Q$. Les propriétés 1 à 3 sont alors évidentes.

Prouvons la surjectivité. Soit $Q \in \cLS^G(X)$. Il existe $P \in \cR^G(X)$ tel que $P \subseteq Q$. D’après la surjectivité de \eqref{R}, il existe $g \in G_{X_\ss}(F)$ tel que $(\Ad g^{-1})P \in\P^G(M)$, donc $(\Ad g^{-1})Q \in \F^G(M)$ et l’image de $(\Ad g^{-1})Q$ par \eqref{LS} est donc $Q$.

Prouvons l’assertion 5. Soient $Q_1$ et $Q_2$ deux éléments de $\F^G(M)$ qui ont la même image, notée $Q$, dans $\cLS^G(X)$ par l’application \eqref{LS}. Pour $i=1,2$, soient $P_i\in\P^G(M)$ et $g_i\in G_{X_\ss}(F)$ tels que $P_i\subseteq Q_i$ et $P_i'\eqdef(\Ad g_i)P_i\in \cR^G(X)$. Soit $M_{Q_i}$ l'unique facteur de Levi de $Q_i$ contenant $M$. Par construction même de l’application \eqref{LS}, on a
\begin{equation}\label{eq:2.10.3}
Q=(\Ad g_1)Q_1=(\Ad g_2)Q_2\in \cLS^G(X).    
\end{equation}
Notons $L=(\Ad g_1)M_{Q_1}=(\Ad g_2)M_{Q_2}$, c'est un facteur de Levi de $Q$ dont l'algèbre de Lie contient $X_\ss$. Notons $\n_i\eqdef \n_{P_{i,X_\ss}'}$. On a $X_\nilp\in \n_1\cap \n_2$ et on écrit
\[X_\nilp=Y+U\]
avec $Y\in \mathfrak{l}_{X_\ss}\cap \n_1\cap \n_2$ et $U\in \n_{Q,X_\ss}$. Comme $X_\nilp\in \Ind_{P_{i,X_\ss}'}^{G_{X_{\ss}}}(0)$ on a également $Y\in \Ind_{P_{i,X_\ss}'\cap L_{X_\ss}}^{L_{X_{\ss}}}(0)$ (cf. la preuve de la proposition \ref{prop:LS->R}). Donc
\[\Ind_{P_{1,X_\ss}'\cap L_{X_\ss}}^{L_{X_{\ss}}}(0)=\Ind_{P_{2,X_\ss}'\cap L_{X_\ss}}^{L_{X_{\ss}}}(0).\]
D'après le point 2 de la proposition \ref{prop:condition(H?)} on a l'existence de $l\in L_{X_{\ss}}(F)$ tel qu'on ait
\[(\Ad l)(\Ad g_1)M_{X_\ss}=(\Ad g_2)M_{X_\ss},\]
c'est-à-dire
\[x=g_2^{-1}mg_1\in\text{Norm}_{G_{X_\ss}(F)}(M_{X_\ss}).\]
On conclut en utilisant (\ref{eq:2.10.3}) d'où l'on tire
\[(\Ad x)Q_1=(\Ad g_2^{-1})(\Ad l)Q=(\Ad g_2^{-1})Q=Q_2.\]
Réciproquement si $Q\in \F^G(M)$ et si $x\in\text{Norm}_{G_{X_\ss}(F)}(M)$ alors pour tout $P\in \P^G(M)$ tel que $P\subseteq Q$, on a $(\Ad x)P\in \P^G(M)$ donc $(\Ad x)Q$ et $Q$ ont même image par \eqref{LS}. Le stabilisateur de $Q$ sous l'action par conjugaison de $G_{X_\ss}(F)$ est évidemment $\text{Norm}_{G_{X_\ss}(F)}(M_{X_\ss})\cap \text{Norm}_{G_{X_\ss}(F)}(Q)=\text{Norm}_{G(F)}(M_{X_\ss})\cap \text{Norm}_{G(F)}(Q)\cap G_{X_\ss}(F)=\text{Norm}_{Q_{X_\ss}(F)}(M_{X_\ss})$.
\end{proof}

\begin{lemma}\label{lem:2.10.2}Soient $M$ un facteur de Levi d'un élément de $\cR^G(X)$ dont l'algèbre de Lie contient $X_\ss$, et $Q\in \F^G(M)$. Pour $i=1,2$ soient $P_i\in\P^G(M)$ et $g_i\in G_{X_\ss}(F)$ tels que 
\[P_i\subseteq Q\,\,\,\,\text{ et }\,\,\,\,(\Ad g_i)P_i\in \cR^G(X).\]
Alors $g_2$ appartient à l'ensemble 
\[G_X(F)g_1M_{Q,X_\ss}(F)P_{2,X_\ss}(F).\]
En particulier, on a
\[(\Ad g_1)Q=(\Ad g_2)Q.\]
\end{lemma}

\begin{proof}
Notons $\n_i\eqdef \n_{P_{i,X_\ss}}$. On écrit
\[(\Ad g_i)X_\nilp=Y_i+U_i\]
avec $Y\in (\mathfrak{m}_{Q,X_\ss}\cap \n_i)(F)$ et $U\in \n_{Q,X_\ss}(F)$. On a 
\[(\Ad P_{i,X_\ss})Y_i=(\Ad G_{X_\ss})X_\nilp\cap \n_i\,\,\,\,\text{ et }\,\,\,\,(\Ad M_{Q,X_\ss})Y_i=\Ind_{P_{1,X_\ss}\cap M_{Q,X_\ss}}^{M_{Q,X_\ss}}(0)\]
par hypothèse et le point 5 de la proposition \ref{chap3prop:indprop}. Les sous-groupes paraboliques $P_{1,X_\ss}\cap M_{Q,X_\ss}$ et $P_{2,X_\ss}\cap M_{Q,X_\ss}$ de $M_{Q,X_\ss}$ ayant même facteur de Levi $M_{X_\ss}$, il en découle 
\[\Ind_{P_{1,X_\ss}\cap M_{Q,X_\ss}}^{M_{Q,X_\ss}}(0)=\Ind_{P_{1,X_\ss}\cap M_{Q,X_\ss}}^{M_{Q,X_\ss}}(0).\]
Il existe donc $m\in M_{Q,X_\ss}(F)$ tel que $(\Ad m)Y_1=Y_2$. Mais alors $(\Ad m)Y_1$ et $Y_2$ appartient à $\n_2\cap (\Ad G_{X_\ss})X_\nilp$. Il existe donc $p_2\in P_{2,X_\ss}(F)$ tel que
\[(\Ad p_2m)Y_1=Y_2.\]
D'où
\[(\Ad p_2mg_1^{-1})X=(\Ad g_2^{-1})X.\]
Ainsi $g_2\in G_X(F)g_1M_{Q,X_\ss}(F)P_{2,X_\ss}(F)$. On a donc
\[(\Ad g_2)Q=(\Ad x)(\Ad g_1)Q\]
pour un certain $x\in G_X(F)$. Or on a $(\Ad g_1)P_1\in \cR^G(X)$ d'où
\[G_X\subseteq (\Ad g_1)P_1\subseteq (\Ad g_1)Q,\]
selon le point 6 de la proposition \ref{chap3prop:indprop}, c'est-à-dire $(\Ad x)(\Ad g_1)Q=(\Ad g_1)Q$.
\end{proof}

La propriété suivante ainsi que sa démonstration joueront un rôle essentiel lorsqu'on discutera la compatibilité des différentes définitions d'une intégrale orbitale pondérée dans la section \ref{sec:4comparedef}.

\begin{proposition}\label{prop:LX=GX}
Soit $L$ un sous-groupe de Levi de $G$ dont l'algèbre de Lie contient $X$, tel que $L_X=G_X$. Alors $L$ est facteur de Levi d'un élément de $\cLS^G(X)$, il contient donc un facteur de Levi d'un élément de $\cR^G(X)$ dont l'algèbre de Lie contient $X_\ss$. L'application \eqref{LS} restreinte sur $\F^G(L)$ devient l'application identité.

La réciproque est vraie. Autrement dit si $L$ est facteur de Levi d'un élément de $\cLS^G(X)$ tel que l'application \eqref{LS} restreinte sur $\F^G(L)$ est l'application identité, alors $L_X=G_X$.
\end{proposition}

\begin{proof}
Pour l'implication directe, on sait que $X\in\Ind_L^G(X)$ selon l'équivalence (1) $\Leftrightarrow$ (5) de la proposition \ref{YDLiopprop:whenXinIndMGX}. Tout élément de $\P^G(L)$ est de ce fait dans $\cLS^G(X)$. Suivant la preuve de la proposition \ref{prop:LS->R} on voit que $L$ contient un facteur de Levi d'un élément de $\cR^G(X)$ dont l'algèbre de Lie contient $X_\ss$. Finalement l'application \eqref{LS} est définie par l'équation (\ref{appdef:LS}), dans laquelle l'élément $g$ est dans $G_{X_\ss}(F)=L_{X_\ss}(F)$, la restriction sur sur $\F^G(L)$ est ainsi l'application identité.

Quant à la réciproque, notons au premier abord que $X\in \cap_{P\in\P^G(L)}\p=\l$, d'où $X\in\Ind_L^G(X)$. Le point 3 de la proposition \ref{chap3prop:indprop} nous assure ainsi $\dim L_X=\dim G_X$, il s'ensuit $L_X=G_X$. 
\end{proof}

\subsection{\'{E}léments standard}\label{subsec:ssssschoix}

Désormais, jusqu'à la fin de l'article, sauf mention explicite du contraire, $F$ désigne un corps local de caractéristique 0.

L'objectif de ce numéro est de désigner, de façon purement combinatoire, un élément spécifique dans chaque classe de $G(F)$-conjugaison de $\g(F)$ (définition \ref{def':repstd}).

Pour la simplicité on prend dans la suite $D$ une algèbre à division dont le centre contient $F$, $V$ un $D$-module à droite libre de rang $n$, muni de $e$ une base ordonnée. Posons $G\eqdef\text{Aut}_{D}(V)\simeq\GL_{n,D}$, où l'isomorphisme est donnée via la base ordonnée $e$. On a $\g=\text{End}_{D}(V)\simeq\gl_{n,D}$. On étend sans difficulté la discussion ci-dessous à tout groupe du type GL en travaillant sur ses facteurs irréductibles. 

Posons $M_0=M_{0}^{e}\eqdef\GL_{1,D}^n$ le sous-groupe diagonal relativement à la base $e$ de $G$, qui est aussi un sous-groupe de Levi minimal. On dit qu'un sous-groupe de Levi (resp. sous-groupe parabolique) de $G$ est semi-standard s'il contient $M_0$. Pour tout $L$ sous-groupe de Levi de $G$ semi-standard, $e$ nous donne une base ordonnée pour $L$. Autrement dit il existe $I$ un ensemble fini, $e_i$ une sous-famille de $e$ ordonnée de la façon respectant la relation d'ordre totale de $e$ pour tout $i\in I$, tels que $L=\prod_{i\in I}\text{Aut}_{D}(V_i)\simeq \prod_{i\in I}\GL_{|e_i|,D}$, où $V_i$ le sous-$D$-module à droite libre engendré par $e_i$. La base $e$ nous fournit également $\uW^{G,0}=\uW^{G,0,e}$, le groupe des matrices de permutation relativement à $e$. C'est un sous-groupe de $G(F)$ isomorphe au groupe de Weyl relatif $W^{(G,M_0)}$.

Quand $F$ est non-archimédien on écrit $\O_D$ pour l'anneau des entiers de $D$ ; quand $F$ est archimédien on écrit $X\mapsto\overline{X}$  pour l'involution principale et $X\mapsto {}^tX$ pour la transposée sur $\gl_m(D)$. On dit que le sous-groupe
\[K=K^e\eqdef\begin{cases*}
 \GL_n(\O_D)  & \text{si $F$ est non-archimédien} \\
 \{g\in \GL_n(D)\mid g\cdot{}^t\overline{g}=\text{Id}\}& \text{si $F$ est archimédien} 
\end{cases*}\] 
de $G(F)$ est en bonne position par rapport à $e$. Il est opportun de remarquer que $\uW^{G,0}\subseteq K$.

Posons $\Irr_F$ l'ensemble des polynômes irréductibles de coefficient dominant 1 de $F[T]$, on munit $\Irr_F$ d'une relation d'ordre totale quelconque $\leq$. Pour $p\in\Irr_F$ on pose $F_p\eqdef F[T]/(p)$. 
Pour tout $p\in \Irr_F$, fixons ${}^p\textbf{X}$ une matrice carrée à coefficient dans $D$ de taille $\underline{n}_p$, de polynôme caractéristique $p^{\underline{a}_p}$, avec 
\[\underline{n}_p\eqdef\frac{\deg p \deg_{F_p}D_p}{\deg_F D}\,\,\,\,\text{ et }\,\,\,\,\underline{a}_p\eqdef\deg_{F_p}D_p,\]
et $D_p$ une $F_p$-algèbre à division telle que $D\otimes_F F_p\simeq \Mat_{c_p}(D_p)$ où $c_p$ est un entier positif.

Pour $V'$ un $D$-module à droite libre, on écrit $\rk(V')=\rk_D(V')$ son rang. Soit $p\in \Irr_F$. Supposons que $V$ est muni d'une décomposition en somme directe de $D$-modules libres
\[V=\bigoplus_{1\leq i\leq j\leq {}^pr}{}^pV_j^i\]
telle que 
\begin{enumerate}
    \item ${}^pr\in\N_{>0}$ ;
    \item $\rk({}^pV_j^i)$ ne dépend que de $j$, et il vaut $\underline{n}_p{}^pd_j$ pour un ${}^pd_j\in \N_{>0}$;
    \item pour tous $1\leq i\leq j\leq {}^pr$ et $1\leq k\leq {}^pd_j$, il existe ${}^pE_{k,j}^i\subseteq {}^pV_j^i$ un sous-$D$-module libre de rang $\underline{n}_p$, muni d'une base ordonnée  $({}^{l,p}e_{k,j}^i)_{1\leq l\leq \underline{n}_p}$ ; 
    \item soit $e'\eqdef\bigsqcup_{i,j,k,l}({}^{l,p}e_{k,j}^i)$ la base ordonnée de la façon suivante :  ${}^{l,p}e_{k,j}^i<{}^{l',p}e_{k',j'}^{i'}$ si l'une des conditions suivantes est satisfaite
    \begin{itemize}
    \item[$-$] $i<i'$ ;
    \item[$-$] $i=i'$ et $j>j'$ (on souligne le sens de l'inégalité ici) ;
    \item[$-$] $i=i'$, $j=j'$ et $k<k'$ ; 
    \item[$-$] $i=i'$, $j=j'$, $k=k'$ et $l< l'$. 
    \end{itemize}
    Alors $e'=e$.
\end{enumerate}

Nous avons 
\[\g=\bigoplus_{i,j,k,i',j',k'}\text{Hom}_{D}({}^pE_{k,j}^{i},{}^pE_{k',j'}^{i'}).\]
Soit ${}^pX\in \g(F)\simeq \gl_n(D)$ défini par
\begin{equation}\label{eq:elementpXellstd}
{}^pX|_{\text{Hom}_{D}({}^{p}E_{k,j}^{i},{}^{p}E_{k',j'}^{i'})}=\begin{cases}
{}^p\textbf{X}& \text{si }i=i',j=j',k=k'\\
\text{Id}_{\underline{n}_p\times \underline{n}_p} & \text{si }i>1,i'=i-1, j=j',k=k'\\
0 & \text{sinon},
\end{cases}   
\end{equation}
pour tous $i,j,k,i',j',k'$, ici le membre de gauche est la représentation matricielle du morphisme en quesiton dans les bases ordonnées concernées. L'élément $X_\ss\in \g_{\ss}(F)$ est elliptique. On dira que $X\in \g(F)$ est standard relativement à $e$ de partie semi-simple elliptique associée à $p\in \Irr_F$.

Matriciellement parlant, dans la base $e$, nous avons
\begin{equation*}
{}^pX=\begin{bmatrix}
{}^p\textbf{X}_{1}  & {}^p\mathbf{J}_{2}  &         &         &  \\
  & {}^p\textbf{X}_2  & {}^p\mathbf{J}_{3}    &        &     \\
       &  &  \ddots &  \ddots   &\\
            &         &  & {}^p\textbf{X}_{{}^pr-1} &  {}^p\mathbf{J}_{{{}^pr}}\\
     &         &         &  &  {}^p\textbf{X}_{{}^pr}
\end{bmatrix}    
\end{equation*}
avec 
\[{}^p\textbf{X}_{i}=\text{la concatenation diagonale $\sum_{l=i}^{{}^pr} {}^pd_l$ fois de ${}^p\textbf{X}$}\in \Mat_{(\sum_{l=i}^{{}^pr} \underline{n}_p{}^pd_l)\times (\sum_{l=i}^{{}^pr} \underline{n}_p{}^pd_l)}(D)\]
et
\[{}^p\mathbf{J}_{i}=
\left[ 
\begin{array}{@{}ccc@{}}
&&\\
 &{\Id_{\Mat_{\sum_{l=i}^{{}^pr}\underline{n}_p{}^pd_l}}}(D) & \\
 && \\
 \hline 
 &0 & \\
\end{array}\right]\in\Mat_{(\sum_{l=i-1}^{{}^pr} \underline{n}_p{}^pd_l)\times (\sum_{l=i}^{{}^pr} \underline{n}_p{}^pd_l)}(D).
\]    

Plus généralement, on dira qu'un élément $X\in \g(F)$ est standard (relativement à $e$ et à la relation d'ordre totale $\leq$ sur $\Irr_F$) s'il existe une décomposition de $V$ en somme directe de $D$-modules libres
\[V=\bigoplus_{p\in \Irr_F}{}^pV\]
telle que 
\begin{enumerate}
    \item soit $p\in \Irr_F$ avec $\text{rk}({}^{p}V)>0$, alors ${}^{p}V$ est muni d'une base ordonnée ${}^pe$ ;
    \item soit $e'\eqdef\bigsqcup_{p : \text{rk}({}^pV)>0}{}^pe$ la base ordonnée de la façon suivante :  
    \begin{itemize}
    \item[$-$] pour tout $p$ avec $\text{rk}({}^{p}V)>0$, la relation d'ordre totale de $e'$ restreinte sur ${}^pe$ est la même que celle de ${}^pe$ ;
    \item[$-$] si $p<p'$ avec $\text{rk}({}^{p}V)>0$ et $\text{rk}({}^{p'}V)>0$ alors tout élément de ${}^{p}e$ est strictement plus petit que tout élément de ${}^{p'}e$. 
    \end{itemize}
    Alors $e'=e$ ;
    \item pour tous $p$ et $p'$ avec $\text{rk}({}^{p}V)>0$ et $\text{rk}({}^{p'}V)>0$, on a $X|_{\text{Hom}_{D}({}^{p}V,{}^{p'}V)}=0$ ;
    \item pour tout $p$ avec $\text{rk}({}^{p}V)>0$, on a que $X|_{\text{End}_{D}({}^{p}V,{}^{p}V)}$ est standard relativement à ${}^pe$ de partie semi-simple elliptique associée à $p$.
\end{enumerate}

Dans la suite, sauf mention explicite du contraire, lorsqu'on parle d'un élément standard de $\mathfrak{g}(F)$, on fait toujours référence à l'élément $X$ défini par le procédé précédent. Autrement dit $V$ est muni d'une décomposition en somme directe de $D$-modules libres $V=\bigoplus_{p\in \Irr_F}{}^pV$ respectant la base ordonnée $e$, et pour tous $p$ et $p'$ avec $\text{rk}({}^{p}V)>0$ et $\text{rk}({}^{p'}V)>0$, on a 
\[X|_{\text{Hom}_{D}({}^{p}V,{}^{p'}V)}=\begin{cases}
\text{l'élément ${}^pX$ défini par l'équation \eqref{eq:elementpXellstd}}& \text{ si }p=p'\\
0 & \text{ sinon}.
\end{cases}   \]
On peut alors parler des objets ${}^pr$, ${}^pe$, etc.


\begin{lemma}
Soit $V'$ un $D$-module à droite libre de rang fini. Soient $f$ et $f'$ deux bases ordonnées de $V'$. Alors tout élément de $\text{End}_{D}(V')$ admet la même écriture matricelle dans les deux bases ordonnées si et seulement s'il existe $d\in D^\times$ tel que $f=df'$ (i.e. $f_s=df_s'$ pour tout $1\leq s\leq \text{rk}(V)$).   
\end{lemma}
\begin{proof}
On regarde d'abord un élément de $\text{End}_{D}(V')$ qui a pour écriture matricelle dans la base $f$ une matrice diagonale de coefficients deux à deux différents. Comme cet élément s'écrit de la même façon dans la base $f'$, on voit qu'il existe $d_s\in D^\times$ tel que $f_s=d_sf_s'$ pour tout $1\leq s\leq \text{rk}(V)$. On regarde ensuite l'élément de $\text{End}_{D}(V')$ qui a pour écriture matricelle dans la base $f$ la matrice qui vaut 1 partout dans la première colonne et qui vaut 0 ailleurs. Comme cet élément s'écrit de la même façon dans la base $f'$, on voit que $d_1=\cdots= d_{\text{rk}(V)}$.
\end{proof}

Soit $X\in\g(F)$ standard. Soit $p\in \Irr_F$ tel que $\rk({}^pV)>0$. Le centralisateur de ${}^p\textbf{X}$ dans $\Mat_{\underline{n}_p\times \underline{n}_p}(D)$ est une algèbre à division, notons-la par $D^{{}^p\textbf{X}}$. On a
\[\g_{X_\ss}=\bigoplus_{p\in \Irr_F,\rk({}^pV)>0}\text{End}_{D^{{}^p\textbf{X}},{}^pe_{X_\ss}}({}^pV_{X_\ss}),\]
avec ${}^pV_{X_\ss}$ un $D^{{}^p\textbf{X}}$-module à droite libre de rang $\sum_{1\leq j\leq {}^pr}{}^pd_j$. Ensuite, on dote ${}^pV_{X_\ss}$ d'une base ordonnée ${}^pe_{X_\ss}$ définie de la manière suivante : d'après le lemme précédent, il existe une base ordonnée unique, à un scalaire près, telle que tout élément de $\text{End}_{D^{{}^p\textbf{X}}}({}^pV_{X_\ss})$ admette la même écriture matricielle dans cette base que dans la base ${}^pe$, lorsque cet élément est considéré comme un élément de $\g$ qui commute avec $X_\ss$. On choisit alors une base ordonnée ${}^pe_{X_\ss}$ quelconque satisfaisant cette condition.

Un élément de $\g(F)$ commutant à $X_\ss$ est standard relativement à $e$ si et seulement s'il est standard relativement à $e_{X_\ss}\eqdef\sqcup_{p:\rk({}^pV)>0}{}^pe_{X_\ss}$ quand il est vu comme élément de $\g_{X_\ss}(F)$. Notons $\uM_{0}^{X}=\uM_{0}^{G,X}$ le sous-groupe diagonal relativement à la base $e_{X_\ss}$ de $G_{X_\ss}$, c'est un sous-groupe de Levi minimal de $G_{X_\ss}$. On note $\uM^{X}=\uM^{G,X}\eqdef\envL(\uM_0^{X};G)$, c'est un sous-groupe de Levi de $G$ semi-standard relativement à $M_0$. On sait, du lemme \ref{lem:envLell}, que l'algèbre de Lie de $\uM^{X}$ contient $X_\ss$ et $(\uM^{X})_{X_\ss}=\uM_{0}^{X}$.

\begin{proposition}\label{def:repstd} Toute classe de conjugaison rationnelle de $G(F)$ contient un unique élément standard. 
\end{proposition} 
\begin{proof}
Ceci est une conséquence de la théorie de la réduction de Jordan \cite[section A.2]{YDL23b}.
\end{proof}

\begin{definition}[Représentant standard]\label{def':repstd} L'élément standard d'une classe de $G(F)$-conjugaison dans $\g(F)$ est dit le représentant standard de cette classe de conjugaison rationnelle. 
\end{definition}

\begin{lemma}\label{lem:LSsstand} Soit $X\in \g(F)$ standard. Alors tout sous-groupe parabolique de Lusztig-Spaltenstein généralisé de $X$ contient le sous-groupe parabolique $\left(\prod_{p:\rk({}^pV)>0}P_0^{{}^pe}\right)\uM^{G,X}$ de $\left(\prod_{p:\rk({}^pV)>0}\Aut_D({}^pV)\right)$, avec $P_0^{{}^pe}$ le sous-groupe des matrices triangulaires supérieures relativement à la base ${}^pe$ de $\Aut_D({}^pV)$. En particulier tout sous-groupe parabolique de Lusztig-Spaltenstein généralisé de $X$ est semi-standard.
\end{lemma}
\begin{proof}
On montre d'abord que tout sous-groupe parabolique de Lusztig-Spaltenstein généralisé de $X$ est semi-standard : nous avons $\cLS^{G_{X_\ss}}(X_\nilp)\subseteq \F^{G_{X_\ss}}(\uM_0^{G,X})$ selon \cite[proposition 3.4.1]{Ch17} et la proposition \ref{prop:LS->R}. Ainsi $\cLS^G(X)\subseteq \F^G(\uM^{G,X})$ grâce au lemme \ref{lem:LSR'ss}. En particulier tout sous-groupe parabolique de Lusztig-Spaltenstein généralisé de $X$ est semi-standard. Soit maintenant $P\in \cLS^G(X)$, on a $P_{X_\ss}\in \cLS^{G_{X_\ss}}(X_\nilp)$ grâce encore une fois au lemme \ref{lem:LSR'ss}, et  \cite[proposition 3.4.1]{Ch17} nous assure aussi que $P_{X_\ss}$ contient $\left(\prod_{p:\rk({}^pV)>0}P_0^{{}^p{e_{X_\ss}}}\right)$ avec $P_0^{{}^pe_{X_\ss}}$ le sous-groupe des matrices triangulaires supérieures relativement à la base ${}^p{e_{X_\ss}}$ de ${}^pV_{X_\ss}$. On en déduit que $P$ contient $\left(\prod_{p:\rk({}^pV)>0}P_0^{{}^pe}\right)\uM^{G,X}$.
\end{proof}

\begin{lemma}\label{lem:uMmin}
Soit $X\in \g
_\ss(F)$ standard. Alors pour tout $M\in\L^G(M_0)$, $X\in\m(F)$ si et seulement si $\uM^{X}\subseteq M$.
\end{lemma}
\begin{proof}
Regardons l'ensemble
\[\{M\in\L^G(M_0)\mid X\in \m(F)\}\not=\emptyset.\]
Le plus petit élément de cet ensemble existe : c'est l'intersection de tous les éléments. On note $M_{\text{min}}\subseteq L$ le plus petit élément. Il s'agit de prouver que $M_{\text{min}}=\uM^{X}$. On sait que $\uM^{X}$ appartient à cet ensemble, donc $M_{\text{min}}\subseteq\uM^{X}$ et $(M_{\text{min}})_X\subseteq(\uM^{X})_X=\uM_0^{X}$. Comme $\uM_0^{X}$ est un sous-groupe de Levi minimal de $G_X$, il vient $(M_{\text{min}})_X=\uM_0^{X}$. Il s'ensuit que $\uM^{X}=\envL((M_{\text{min}})_X;G)$, il est alors inclus dans $M_{\text{min}}$, par voie de conséquence $M_{\text{min}}=\uM^{X}$.
\end{proof}

\begin{corollary}\label{YDLiop:RG'(X)=RG(X)ell}
Soit $X\in \g(F)$ tel que $X_\ss$ est $F$-elliptique, alors $\cR^G(X)'=\cR^G(X)$.
\end{corollary}
\begin{proof}
Il suffit de prouver ce corollaire pour $X$ standard. Soit $P\in \cR^G(X)'$. On note $M$ son facteur de Levi semi-standard. On a donc $M\in\L^G(\uM^{X_\ss})$. Or, par des calculs directs il est facile à voir que $X_\ss$ est $F$-elliptique dans tout élément de $\L^G(\uM^{X_\ss})$, le lemme \ref{lem:Rlevienvlop} entraîne alors $\cR^{G}(X)'=\cR^{G}(X)$.    
\end{proof}

\subsection{\'{E}lément \texorpdfstring{$w_P$}{wP}}\label{YDLiopsubsec:elementwP}
Soit $X\in \g(F)$ standard défini par le procédé de la sous-section précédente.

Pour $A$ un ensemble on note $\mathfrak{S}(A)$ son groupe symétrique. Soit $\varepsilon$ un élément de $\prod_{p\in \Irr_F : \rk({}^pV)>0}\mathfrak{S}(\{(i,j,k)\mid 1\leq i\leq j\leq {}^pr,1\leq k\leq {}^pd_j\})$. On pose
\[\varepsilon_X\in G(F)\subseteq\bigoplus_{i,j,k,i',j',k'}\text{Hom}_{D}(E_{k,j}^{i},E_{k',j'}^{i'}).\]
l'élément défini par
\begin{equation}\label{elementepsilonX}
\varepsilon_X|_{\text{Hom}_{D}({}^{p}E_{k,j}^{i},{}^{p'}E_{k',j'}^{i'})}=\begin{cases}
\text{Id}_{\underline{n}_p\times \underline{n}_p} & \text{si }p=p'\text{ et }(i',j',k')=\varepsilon(i,j,k)\\
0 & \text{sinon},
\end{cases}   
\end{equation}
pour tous $p,i,j,k,p',i',j',k'$. Posons
\[\uW^{G,X}=\{\varepsilon_X\mid \varepsilon\in \prod_{p\in \Irr_F : \rk({}^pV)>0}\mathfrak{S}(\{(i,j,k)\mid 1\leq i\leq j\leq {}^pr,1\leq k\leq {}^pd_j\})\}.\]
On vérifie aisément que $\uW^{G,X}\subseteq {\text{Norm}_{G_{X_\ss}(F)}(\uM_{0}^{G,X})}$ et il est isomorphe à $W^{(G,\uM^{G,X},X_\ss)}$. Il est en effet le sous-groupe de $G_{X_\ss}(F)$ des matrices de permutation relativement à la base $e_{X_\ss}$. En outre, pour tout $P$ sous-groupe parabolique de $G$ dont l'algèbre de Lie contient $X_\ss$ et $M_P$ son facteur de Levi semi-standard (donc son algèbre de Lie contient $X_\ss$), nous avons
\begin{equation}\label{eq:weylindP}
P(F)\cap \uW^{G,X}=\text{Norm}_{P_{X_\ss}(F)}(\uM_{0}^{X})\cap \uW^{G,X} =M_P(F)\cap \uW^{G,X}, 
\end{equation}
tous vus comme intersection dans $G(F)$.  Pour tout $M$ sous-groupe de Levi de $G$ dont l'algèbre de Lie contient $X_\ss$, on note $\uW^{G,X,M}\eqdef \uW^{G,X}\cap M(F)$. 

En parallèle, l'inclusion naturelle $\text{Norm}_{G_{X_\ss}(F)}(\uM_{0}^{G,X})=\text{Norm}_{G_{X_\ss}(F)}(\uM^{G,X})\hookrightarrow \text{Norm}_{G(F)}(\uM^{G,X})$ (cf. la preuve de l'équation \eqref{eq:pW=W} quant à la première égalité) induit une inclusion naturelle
\begin{align*}
W^{(G,\uM^{G,X},X_\ss)}=\text{Norm}_{G_{X_\ss}(F)}(\uM_{0}^{G,X})/\uM_{0}^{G,X}(F)&=\text{Norm}_{G_{X_\ss}(F)}(\uM_{0}^{G,X})/(\text{Norm}_{G_{X_\ss}(F)}(\uM_{0}^{G,X})\cap \uM^{G,X}(F))\\
&\hookrightarrow\text{Norm}_{G(F)}(\uM^{G,X})/\uM^{G,X}(F)=W^{(G,\uM^{G,X})}    
\end{align*}
du groupe de Weyl relatif de $(G_{X_\ss},\uM_0^{G,X})$ dans le groupe de Weyl de $(G,\uM^{G,X})$. 

Constatons que la formation des groupes $\uW$ commute à la décomposition de Jordan, i.e. $\uW^{G,X}=\uW^{G_{X_\ss},X_\nilp}$ et $\uW^{G,X,M}=\uW^{G_{X_\ss},X_\nilp,M_{X_\ss}}$ pour tout $M$ sous-groupe de Levi de $G$ dont l'algèbre de Lie contient $X_\ss$. 
 
L'élément nul $0$ est standard. Il y a des inclusions naturelles $\uW^{G,X}\hookrightarrow \uW^{G,0}$, et $\uW^{L,0}\hookrightarrow \uW^{G,0}$ pour tout $L\in\L^G(M_0)$.

\begin{lemma}\label{lem:wP}
Soit $M$ un facteur de Levi semi-standard d'un élément de $\cR^G(X)$. Pour tout $P\in \F^G(M)$ il existe un élément $w_{P,X}=w_P \in \uW^{G,0}$, unique à translation à gauche près par $\uW^{G,0,M_P}$ tel que
\[(\Ad w_P^{-1})P\in \cLS^G(X)\]
soit l’image de $P$ par l’application \eqref{LS}. Si $P\in\P^G(M)$, la condition $(\Ad w_P^{-1})P\in \cR^G(X)$ suffit pour que $(\Ad w_P^{-1})P$ soit l’image de $P$ par l’application \eqref{LS}. Cet élément est compatible à la décomposition de Jordan, i.e. on peut prendre $w_{P,X}=w_{P_{X_\ss},X_\nilp}\in \uW^{G,X}$. Si $w_{P,X}$ est choisi dans $\uW^{G,X}$, il sera unique à translation à gauche près par $\uW^{G,X,M_P}$.
\end{lemma}

\begin{proof}
Montrons d'abord le lemme lorsque $P \in \P^G(M)$. Il  existe un unique élément $\widetilde{P} \in R^G(X)$ tel que $\widetilde{P}$ soit $G_{X_\ss}(F)$-conjugué à $P$. Donc il existe un unique élément $g\in P_{X_\ss}(F)\backslash G_{X_\ss}(F)$ tel que $(\Ad g^{-1})P\in \cR^G(X)$. On identifie $g$ à un représentant dans $G_{X_\ss}(F)$. On a $(\Ad g)\uM_0^X\subseteq P$ par le lemme \ref{lem:LSsstand}, donc $(\Ad g)\uM_0^X\subseteq P_{X_\ss}$. D'un autre côté $P$ contient $M$, donc contient $\uM^{X}$ selon le lemme \ref{lem:uMmin}, il vient $\uM_0^X\subseteq P_{X_{\ss}}$. Ces deux sous-groupes de Levi minimaux sont donc $P_{X_{\ss}}(F)$-conjugués : il existe $p\in P_{X_{\ss}}(F)$ tel que $pg\in \text{Norm}_{G_{X_\ss}(F)}(\uM_0^X)$. Quitte
à translater $p$ à gauche par un élément de $\uM_0^X$, on peut supposer que $pg\in \uW^{G,X}\subseteq \uW^{G,0}$. D'où l'existence de $w_P$. L'unicité résulte des égalités $\text{Norm}_{\uW^{G,0}}(P)=\uW^{G,0,M_P}$ et $\text{Norm}_{\uW^{G,X}}(P)=\uW^{G,X,M_P}$, conséquences de l'équation \eqref{eq:weylindP}.

Soit $Q\in \F^G(M)$. Il existe $P\in \P^G(M)$ tel que $P\subseteq Q$. Soit $w_P\in \uW^{G,0}$ tel que $(\Ad w_P^{-1})P\in \cR^G(X)$, on a $(\Ad w_P^{-1})Q\in \cLS^G(X)$. Cela donne l'existence. L'unicité résulte de $\text{Norm}_{\uW^{G,0}}(Q)=\uW^{G,0,M_Q}$ et $\text{Norm}_{\uW^{G,X}}(Q)=\uW^{G,X,M_Q}$.

Enfin on peut prendre $w_{P,X}=w_{P_{X_\ss},X_\nilp}$ grâce au lemme \ref{lem:LSR'ss}, et s'il est choisi dans $\uW^{G,X}$, il sera unique à translation à gauche près par $\uW^{G,X,M_P}$ toujours grâce à l'équation \eqref{eq:weylindP}.
\end{proof}

\begin{remark}\label{rem:w_Psimple}
On voit aisément, grâce à \cite[proposition 3.4.1]{Ch17}, que lorsque $X$ est un nilpotent standard, $w_{P,X}$ est également l'unique élément de $\uW^{G,0}$ modulo $\uW^{G,0,M_P}$ à gauche, tel que $(\Ad w_{P,X}^{-1})P\supseteq P_0^{e}$ avec $P_0^{e}$ le sous-groupe des matrices triangulaires supérieures relativement à la base $e$. Ainsi pour tout $X\in \g(F)$  standard, $w_{P,X}$ est également l'unique élément de $\uW^{G,X}$ modulo $\uW^{G,X,M_P}$ à gauche, tel que $(\Ad w_{P,X}^{-1})P_{X_\ss}\supseteq P_0^{e_{X_\ss}}$ avec $P_0^{e_{X_\ss}}$ le sous-groupe (de $G_{X_\ss}$) des matrices triangulaires supérieures relativement à la base $e_{X_\ss}$.
\end{remark}

\subsection{\texorpdfstring{$(G,M)$}{(G,M)}-famille}\label{YDLiopsubsec:GMfamily}
Pour rappel un sous-groupe de Levi (resp. sous-groupe parabolique) de $G$ est appelé semi-standard s'il contient $M_0$. Pour $H$ un groupe algébrique défini sur $F$, on note $X^\ast(H)$ le groupe des caractères de $G$ définis sur $F$ et $a_H$ (resp. $a_H^\ast$) l’espace vectoriel reél $\Hom(X^\ast(H),\R)$ (resp.$X^\ast(H)\otimes_\Z\R$). Les espaces $a_H$ et $a_H^\ast$ sont duaux l'un de l'autre. Leur dimension est égale à celle de $A_H$ sur $F$. 
Soit $M$ un sous-groupe de Levi semi-standard. Soit $P\in\F^G(M)$. Notons $M_P$ l'unique facteur de Levi de $P$ contenant $M$. On a $A_P= A_{M_P}$. Par restriction il y des isomorphismes canoniques de groupes $X^\ast(P)\simeq X^\ast(M_P)$, $a_P\simeq a_{M_P}$, et $a_{P}^\ast\simeq a_{M_P}^\ast$. ,

Soit $H$ un sous-groupe de $G$ stable par conjugaison par $A_M$. On note $\Sigma(\mathfrak{h}; A_M)$ l’ensemble des racines de $A_M$ sur $\mathfrak{h}$, on a alors
\[\mathfrak{h}=\mathfrak{h}^{A_M}\oplus\bigoplus_{\alpha\in\Sigma(\mathfrak{h};A_M)}\mathfrak{h}_\alpha\]
avec $\mathfrak{h}^{A_M}$ le sous-espace invariant et $\mathfrak{h}_\alpha$ le sous-espace propre relativement à $\alpha$ de $\mathfrak{h}$ sous l'action adjointe par $A_M$ :
\[\mathfrak{h}_\alpha=\{X\in\mathfrak{h}\mid \Ad(a)X=\alpha(a)X\,\,\,\,\forall a\in A_M\}.\]

Soit de plus $Q\in\F^G(M)$ tel que $P\subseteq Q$. On note $\Delta_P^Q$ l’ensemble des racines simples dans $\Sigma (M_Q\cap N_P; A_{M_P})$, et $\rho_P^Q\eqdef \frac{1}{2}\sum_{\alpha\in\Sigma (M_Q\cap N_P; A_{M_P})}\left(\dim (\mathfrak{m}_Q\cap \mathfrak{n}_P)_\alpha\right)\alpha$. Lorsque $Q = G$, on note simplement $\Delta_P = \Delta_P^G$ et $\rho_P=\rho_P^G$.  On a $\Delta_P^Q\subseteq \Delta_P$.

Les ensembles $\Sigma (G; A_{M_P})$ et $\Sigma (P; A_{M_P})$ sont canoniquement inclus dans $a_{M_P}^\ast$. \`{A} chaque racine $\beta\in \Sigma (G; A_{M})$ est associée une coracine $\beta^\vee\in a_{M}$ (Il existe $R\in \P^G(M)$ tel que $\beta\in\Delta_R$ et $P_0$ un sous-groupe parabolique minimal contenu dans $R$. Dans ce cas, $\beta$ est la restriction à $a_M$ d’une unique racine $\beta_0\in\Delta_{P_0}$. On définit alors la coracine $\beta^\vee$ comme la projection de la coracine $\beta_0^\vee$ sur $a_M$. On vérifie que $\beta^\vee$ ne dépend que de $\beta$ et non pas des choix de $P$ ou $P_0$.) On définit alors dans $a_{M_P}$ le sous-ensemble $(\Delta_P^Q)^\vee$ des coracines des éléments de $\Delta_P^Q$. Pareillement,  \`{a} chaque racine $\beta\in a_{M}^\ast$ est associée un copoids $\varpi_{\beta}^\vee\in a_{M}$ ; \`{a} chaque coracine $\beta^\vee\in a_{M}$ est associée un poids $\varpi_{\beta}\in a_{M}^\ast$. On définit dans $a_{M_P}$ le sous-ensemble $(\widehat{\Delta}_P^Q)^\vee$ des copoids des éléments de $\Delta_P^Q$, puis dans $a_{M_P}^\ast$ le sous-ensemble $\widehat{\Delta}_P^Q$ des poids des éléments de $(\Delta_P^Q)^\vee$.

Soit $a_P^Q$ (resp. $(a_p^Q)^\ast$) le sous-espace vectoriel engendré par $(\Delta_P^Q)^\vee$ (resp. $\Delta_P^Q$). Notons que $a_{Q}$ (resp. $a_{Q}^\ast$) s’identifie canoniquement à un sous-espace de $a_{P}$ (resp. $a_{P}^\ast$).
On sait que 
\begin{equation}\label{eq:a0decomporthoW}
a_P=a_P^Q\oplus a_Q\,\,\,\,\text{et}\,\,\,\,a_P^\ast=(a_P^Q)^\ast\oplus a_Q^\ast.    
\end{equation}
La réalisation $\uW^{G,0}$ du groupe de Weyl $W^{(G,M_0)}$ de $(G,M_0)$ agit naturellement sur $a_{M_0}$ et son dual. On peut et on va munir chacun de $a_{M_0}$ et son dual d’un produit scalaire invariant sous $\uW^{G,0}$, et note $\|\cdot\|$ la norme euclidienne qui s’en déduit. Les décompositions \eqref{eq:a0decomporthoW} sont orthogonales. 

Pour tout espace vectoriel $V$ sur $\R$, on note $V_{\C}\eqdef V\otimes_{\R}\C$ sa complexification, et $iV\subseteq V_{\C}$ le sous-$\R$-espace évident.

Une $(G, M)$-famille est un ensemble d'applications $\{c_P(\lambda)\mid P\in\P^G(M)\}$ de $\lambda\in ia_M^\ast$ vérifiant la propriété que si $P$ et $Q$ sont adjacents (i.e. $\Sigma(\mathfrak{p}; A_{M})\cap (-\Sigma(\mathfrak{q}; A_{M}))$ est un singleton) et si $\lambda$ appartient à l'hyperplan engendré par le mur en commun des chambres de $P$ et $Q$ dans $ia_M^\ast$ alors $c_P(\lambda) = c_{Q}(\lambda)$. 

Pour tout $P\in \P^G(M)$ on définit les fonctions polynomiales homogènes
\[\theta_P(\lambda)=\vol\left(a_M^G/\Z(\Delta_P^\vee)\right)^{-1}\cdot\prod_{\alpha\in \Delta_P}\lambda(\alpha^\vee),\,\,\,\,\lambda\in ia_M^\ast.\]
et
\[\widehat{\theta}_P(\lambda)=\vol\left(a_M^G/\Z(\widehat{\Delta}_P^\vee)\right)^{-1}\cdot\prod_{\varpi\in \widehat{\Delta}_P}\lambda(\varpi^\vee),\,\,\,\,\lambda\in ia_M^\ast.\]
Pour toute $(G,M)$-famille $(c_P)_{P\in\P^G(M)}$, la fonction
\[c_M(\lambda)=\sum_{P\in \P^G(M)}c_P(\lambda)\theta_P(\lambda)^{-1} \]
s'étend en une fonction sur $ia_M^\ast$. On note $c_M$ la valeur $c_M(0)$.

Soit $(c_P)_{P\in\P^G(M)}$ une $(G,M)$-famille. D'une part pour tout $L \in \L^G(M)$, on peut définir une $(G, L)$-famille $(c_R)_{R\in \P^G(L)}$ par
\begin{equation}\label{eq:GMfamilletoGL}
c_R(\lambda)=c_P(\lambda),\,\,\,\,\lambda\in ia_L^\ast    
\end{equation}
où $P$ est un élément de $\P^G(M)$ qui vérifie $P\subseteq R$. Si l’on fixe d'autre part un élément $Q\in \F^G(M)$, on définit une $(M_Q,M)$-famille $(c_R^Q)_{R\in\P^{M_Q}(M)}$ par
\begin{equation}\label{eq:GMfamilletoMQM}
c_R^Q(\lambda)=c_{Q(R)}(\lambda),\,\,\,\,\lambda\in ia_M^\ast
\end{equation}
où $Q(R)$ est l’unique élément de $\P^G(M)$ qui est contenu dans $Q$, à savoir $RN_Q$.

Soit $(d_P)_{P\in\P^G(M)}$ une  $(G,M)$-famille. Pour tout élément $Q\in \F^G(M)$, il existe un
nombre complexe $d_Q'$ de sorte que pour toute $(G,M)$-famille $(c_P)_{P\in\P^G(M)}$, le produit s’écrive comme
\[(cd)_M=\sum_{Q\in\F^G(M)}c_M^Qd_Q'.\]
S'il existe de plus pour tout $L\in \L^G(M)$ un nombre complexe $c_M^L$ tel que $c_M^Q= c_M^L$ pour tout $Q \in \P^G(L)$, alors on a
\[(cd)_M=\sum_{L\in\L^G(M)}c_M^Ld_L.\]

Soit $\lambda\in a_{M,\C}^\ast$ un vecteur en position générale. D'une part pour calculer $c_M^G$ on possède la formule, indépendante de $\lambda$, 
\begin{equation}\label{eq:GMcalcul}
c_M^G=\frac{1}{p!}\sum_{R\in\P^{G}(M)}\left(\lim_{t\to 0}\left(\frac{d}{dt}\right)^p c_R^G(t\lambda)\right)\theta_R^{G}(\lambda)^{-1},
\end{equation} 
où $t\in\R$ et $p=\dim(a_{M}^{G})$ (cf. \cite[l'équation (6.5)]{Art81}). D'autre part pour calculer $c_Q'$ on possède la formule, indépendante de $\lambda$, 
\begin{equation}\label{eq:GM'calcul}
c_Q^G{}'=\frac{1}{q!}\sum_{\{R : Q\subseteq R\}}(-1)^{\dim(a_Q^R)}\widehat{\theta}_Q^R(\lambda)^{-1}\left(\lim_{t\to 0}\left(\frac{d}{dt}\right)^q c_R^G(t\lambda)\right)\theta_R^{G}(\lambda)^{-1},
\end{equation}
où $t\in\R$ et $q=\dim(a_{Q}^{G})$ (\textit{Ibid. }la ligne après l'équation (6.5)). 

Moyennant certains choix (\cite[section 7]{Art88I}), on peux montrer qu’il existe des applications
\begin{align*}
    d_M^G :\L^G(M)\times \L^G(M)&\rightarrow [0,+\infty[ \\
    s:\L^G(M)\times \L^G(M)&\rightarrow \F^G(M)\times \F^G(M)
\end{align*}
de sorte que, pour tout $(L_1,L_2)\in \L^G(M)\times \L^G(M)$, on ait
\begin{enumerate}
    \item si $s(L_1,L_2)=(Q_1,Q_2)$ alors $(Q_1,Q_2)\in \P^G(L_1)\times\P^G(L_2)$ ;
    \item pour tout $w\in \uW^{G,0}$ on a $s((\Ad w )L_1,(\Ad w)L_2)=(\Ad w)s(L_1,L_2)$ ;
    \item $d_M^G(L_1,L_2)\not =0$ si et seulement si l'une des flèches naturelles
    \[a_M^{L_1}\oplus a_M^{L_2}\longrightarrow a_M^G\]
    et
    \[a_{L_1}^G\oplus a_{L_2}^G\longrightarrow a_M^G\]
    est un isomorphisme, auquel cas les deux sont isomorphismes et $d_M^G(L_1,L_2)$ est le volume dans $a_M^G$ du parallélotope formé par les bases orthonormées de $a_M^{L_1}$ et de $a_M^{L_2}$ ;
    \item si $(c_P)_{P\in\P^G(M)}$ et $(d_P)_{P\in\P^G(M)}$ sont des $(G,M)$-familles, on a l'égalité
    \begin{equation}\label{eq:GMfamilleproductformula}
    (cd)_M=\sum_{(L_1,L_2)\in \L^G(M)\times \L^G(M)}d_M^G(L_1,L_2)c_M^{Q_1}d_M^{Q_2}    
    \end{equation}
    \item (Formule de descente) si $(c_P)_{P\in\P^G(M)}$ est une $(G,M)$-famille et $L\in \L^G(M)$, alors
    \begin{equation}\label{eq:GMfamilledescentformula}
    c_L=\sum_{L'\in \L^G(M)}d_M^G(L,L')c_M^{Q'}
    \end{equation}
    où l'on note $Q'$ la deuxieme composante de $s(L,L')$.
\end{enumerate}

On dit qu'une famille de points $(Y_P)_{P\in\P^G(M)}$ de $a_M$ est orthogonale si
elle vérifie la condition suivante : si $P, Q\in P^G(M)$ sont adjacents, le vecteur $Y_P-Y_Q$ appartient à la droite engendrée par la coracine associée à l’unique élément de $\Sigma(\mathfrak{p}; A_{M})\cap (-\Sigma(\mathfrak{q}; A_{M}))$. Lorsqu'on dispose d'une famille orthogonale $(Y_P)_{P\in\P^G(M)}$, on définit pour tout
$Q \in \F^G(M)$ un point $Y_Q \in a_{M_Q}$ de la manière suivante : on choisit $P \in \P^G(M)$ tel
que $P\subseteq Q$ et, on définit $Y_Q$ comme l'image de $Y_P$ par la projection orthogonale $a_M\rightarrow a_{M_Q}$. Les propriétés d'orthogonalité font que le point $Y_Q$ ne dépend pas du choix de $P \subseteq Q$, et si $L \in \L^G(M)$, la famille ainsi obtenue $(Y_Q)_{Q\in\P^G(L)}$ est encore orthogonale. 

Notons $|-|=|-|_F$ la valeur absolue normalisée sur $F$.

Nous avons fixé $K$ un sous-groupe compact maximal de $G(F)$ en bonne position par rapport à $e$. On défini l'application de Harish-Chandra comme suit : soit $M$ un sous-groupe de Levi semi-standard, on définit alors $H_M : M(F)\rightarrow a_{M}$ par $e^{\langle H_M(m),\chi\rangle}=|\chi(m)|_{F}$. Puis si $P\in\F^G(M_0)$, on définit $H_P : G(F)\rightarrow a_{P}$ par $H_P(mnk)\eqdef H_{M_P}(m)$, suivant la décomposition d'Iwasawa $G(F)=M_P(F)N_P(F)K$.

La famille $(-H_{P}(g))_{P\in \P^G(M)}$ est orthogonale pour tout $g \in G(F)$, et si $P\subseteq Q$, le point $H_Q(g)$ est bien la projection orthogonale de $H_P(g)$. On note $(v_{P,G-\reg}(\lambda,g))_{P\in \P^G(M)}$ la $(G,M)$-famille $v_{P,G-\reg}(\lambda,g)\eqdef e^{-\langle H_{P}(g),\lambda\rangle}$ pour $g\in G(F)$.

\subsection{Normalisations de mesures}\label{YDLiopsubsec:normalizationmeasure}

On fixe à présent les normalisations de mesures. 

Rappelons que l'on a fixé un produit scalaire sur $a_{M_0}$ et son dual invariant par $\uW^{G,0}$. La mesure sur les sous-espaces de $a_{M_0}$ et son dual (en particulier les $a_M$ et $a_M^\ast$ pour $M\in \L^G(M_0)$) est héréditée de celle sur $a_{M_0}$ et son dual.

Pour tout $X\in \g$ on pose $D^{\g}(X)=\det(\text{ad}(X_{\ss}); \g/\g_{X_{\ss}})$, avec $\text{ad}$ l'action adjointe de $\g$ sur lui-même. 

Notons $\tau_\g$ la trace réduite de l'algèbre séparable sous-jacente de $G$, elle envoie alors un élément de $\g(F)$ sur un élément de $F$. On fixe $\langle X,Y\rangle =\tau_\g(XY)$ comme forme bilinéaire sur $\g(F)$ non-dégénérée et invariante par adjonction. Nous appellerons $\langle-,-\rangle$ la forme bilinéaire canonique de $\g$ dans cet article. Remarquons que pour $X\in\g_\ss(F)$, la forme bilinéaire canonique de $\g$ restreint à $\g_X$ est la forme bilinéaire canonique de $\g_X$. 

On se donne un caractère non-trivial $\psi$ de $F$. Pour $\mathfrak{h}$ une sous-algèbre séparable de $\g$ de la forme $G_X$ avec $X\in \g_{\ss}(F)$, la forme bilinéaire $\langle-,-\rangle$ est non-dégénéré sur $\mathfrak{h}(F)$, et on prend la mesure de Haar auto-duale sur $\mathfrak{h}(F)$ relativement à sa transformée de Fourier. Puis on note $H$ le groupe des unités de $\mathfrak{h}$, et on prend sur $H(F)$ la mesure de Haar 
\begin{equation}\label{eq:HaarmeasureonGXss}
dh=\frac{dH}{|\text{Norm}_{\mathfrak{h}}(h)|},    
\end{equation}
où $dH$ est la mesure sur $\mathfrak{h}(F)$ et $\text{Norm}_{\mathfrak{h}}$ est la norme (et non la norme réduite) sur $F$ de l'algèbre séparable $\mathfrak{h}$. On remarque cette mesure est également celle induite via l’exponentielle de la mesure sur un voisinage de $0$ de $\mathfrak{h}(F)$, i.e.
\[\int_{V} f(\exp H)|\det(d\exp)_H|\,dH=\int_{\exp V}f(h)\,dh\]
pour tous $V$ un voisinage de $0$ de $\mathfrak{h}(F)$ tel que $\exp :V \to \exp(V)$ soit bijectif et $f\in L^1(\exp V)$. Tout sous-groupe de Levi, tout sous-tore maximal d'un sous-group de Levi de $G$, et tout sous-groupe de $G$ de la forme $M_Y$ avec $M$ un sous-groupe de Levi et $Y\in\m_\ss(F)$, est de la forme $G_X$ avec $X\in \g_\ss(F)$. Remarquons que pour tout $H$ sous-groupe de $G$ comme plus haut et $g\in G(F)$, la mesure sur $(\Ad g )H(F)$ correspond à la mesure sur $H(F)$ par conjugaison par $g$.

La mesure sur un espace quotient est la mesure quotient. On obtient alors une mesure sur $G_X(F)\backslash G(F)$ pour tout $X\in \g_\ss(F)$.

Soit $M$ un sous-groupe de Levi. Soit $\o$ une $M$-orbite dans $\m$ pour l'action adjointe, qui contient un $F$-point. On pose 
\begin{equation}\label{YDLiopeq:defaMoG-reg}
\a_{M,\o,G-\reg}\eqdef \{A\in\a_M=\text{Lie}(A_M)\mid A+Y\in \Ind_M^G(A+Y),\forall Y\in \o\}.    
\end{equation}

\begin{proposition}\label{YDLiopprop:whatisaMoG-reg}~{}
\begin{enumerate}
    \item $\a_{M,\o,G-\reg}$ est un ouvert dense de $\a_{M}$ défini sur $F$.
    \item L'intersection $\a_{M,\o,G-\reg}\cap \a_{M,G-\reg}$ contient un voisinage de $0$ dans $\a_{M,G-\reg}$, avec $\a_{M,G-\reg}$ le lieu régulier de $\a_{M}$ pour la $G$-conjugaison (équation \eqref{eq:deflieuG-reg}).
    \item $\a_{M,\o,G-\reg}=\a_{M,\o_\ss,G-\reg}$, où $\o_\ss\eqdef\{Y_\ss\mid Y\in \o\}$ est une $M$-orbite semi-simple dans $\m$.
    \item Pour tout $Y\in \o$, on a 
    \[\a_{M,\o,G-\reg}\subseteq \a_{M_{Y_\ss},G_{Y_\ss}-\reg}.\]
    \item Si $\o$ est nilpotent, on a 
     \[\a_{M,\o,G-\reg}=\a_{M,G-\reg}.\]
\end{enumerate}    
\end{proposition}
\begin{proof}On a $\a_{M,G-\reg}=\{A\in \a_M\mid \prod_{\alpha\in\Sigma(\g;A_M)}\alpha(A)\not=0\}$.  Soit dans la suite $P\in\P^G(M)$.
\begin{enumerate}
    \item Soit $A\in \a_M$. Considérons l'endomorphisme $F$-linéaire $[-,A+Y]:\n_P\to\n_P$, avec $[-,-]$ le crochet de Lie sur $\g$. En utilisant la décomposition $\n_P=\bigoplus_{\alpha\in \Sigma(\g;A_M)}\n_{P,\alpha}$ où $\n_{P,\alpha}$ est le sous-espace propre  relativement à $\alpha$, on voit que le déterminant $\det([-,A+Y];\n_P)$ est une fonction polynômiale à $|\Sigma(\g;A_M)|$ variables $\{\alpha(A)\mid \alpha\in\Sigma(\g;A_M) \}$ de terme dominant $\prod_{\alpha\in\Sigma(\g;A_M)}\alpha(A)^{\dim \n_{P,\alpha}}$. Selon l'équivalence (1) $\Leftrightarrow$ (6) de la proposition \ref{YDLiopprop:whenXinIndMGX} on voit que $A\in \a_{M,\o,G-\reg}$ si et seulement si $\det([-,A+Y];\n_P)\not =0$. On en déduit que $\a_{M,\o,G-\reg}$ est un ouvert dense de $\a_{M}$ défini sur $F$.
    \item Si $M=G$ il n'y a rien à faire. Supposons que $M\not =G$. Bien sûr $0\not\in \a_{M,G-\reg}$. On conclut ainsi par la description du déterminant $\det([-,A+Y];\n_P)$ du point précédent.
    \item Cela découle de l'équivalence (1) $\Leftrightarrow$ (4) de la proposition \ref{YDLiopprop:whenXinIndMGX}.
    \item Soient $A\in \a_{M,\o,G-\reg}$ et $Y\in \o$. On veut prouver que $A\in  \a_{M_{Y_\ss},\o,G_{Y_\ss}-\reg}$.
    D'abord on a $A\in \a_M\subseteq a_{M_{Y_\ss}}$. Il reste, grâce à l'équivalence (1) $\Leftrightarrow$ (4) de la proposition \ref{YDLiopprop:whenXinIndMGX}, à prouver que $(G_{Y_\ss})_A=(M_{Y_{\ss}})_A$. Soit $V\in (G_{Y_\ss})_A$. Alors $V\in \g$ et $[V,Y_\ss]=[V,A]=0$. On a donc $V\in G_{A+Y_\ss}=M_{A+Y_\ss}\subseteq M$. Ainsi $V\in (M_{Y_\ss})_A$, ce qu'il fallait.
    \item Si $\o$ est nilpotent, alors le déterminant $\det([-,A+Y];\n_P)$ vaut $\prod_{\alpha\in\Sigma(\g;A_M)}\alpha(A)^{\dim \n_{P,\alpha}}$ pour tout $Y\in\o$. Puisque $\dim \n_{P,\alpha}>0$ pour tout $\alpha$, on a $\a_{M,\o,G-\reg}=\a_{M,G-\reg}$. \qedhere  
\end{enumerate}    
\end{proof}


Une description plus précise de $\det([-,A+Y];\n_P)$ est fournie ici. On suppose pour l'instant que $G=\GL_{n,D}$ et $M=\prod_{i=1}^r\GL_{n_i,D}$ avec $D$ une algèbre simple dont le centre contient $F$, $r$ un entier strictement positif, et $(n_1,\dots,n_r)$ une partition de $n$. On a $A_M=\prod_{i=1}^r\GL_{1,F}$. On note $M_i=\GL_{n_i,D}$ pour $i=1,\dots,r$. L'ensemble des racines $\Sigma(\g;A_M)$ est en bijection avec l'ensemble des paires $(i,j)$ d'entiers différents entre $1$ et $r$ : soit $(a_1\Id_{n_1\times n_1},\dots,a_r\Id_{n_r\times n_r})\in A_M$ avec $a_i\in \GL_{1,F}$, alors la racine qui correspond à $(i,j)$ est la racine $\alpha_{(i,j)}(a_1\Id_{n_1\times n_1},\dots,a_r\Id_{n_r\times n_r})=a_ia_j^{-1}$. La projection de $\o$ sur chaque bloc $\m_i$ de $\m$ est une une $M_i$-orbite dans $\m_i$. On écrit $p_{i,\o}$ pour le polynôme caractéristique de la projection de $\o$ sur le bloc $\m_i$. Enfin pour $(i,j)$ une paire d'entiers différents entre $1$ et $r$, on note $p_{\alpha_{(i,j)},\o}\in F[T]$ le résultant des $p_{i,\o}(X)$ et $p_{j,\o}(X-T)$, vus comme polynômes en $X$. On a $p_{\alpha_{(j,i)},\o}(T)=p_{\alpha_{(i,j)},\o}(-T)$. Pour $G$ un groupe du type GL général, on définit les polynômes $p_{\alpha,\o}\in F[T]$ pour $\alpha\in \Sigma(\g;A_M)$ en travaillant sur les facteurs de $G$.

\begin{proposition}
On a 
    \[\a_{M,\o,G-\reg}=\{A\in \a_M\mid \prod_{\alpha\in \Sigma(\g;A_M)}p_{\alpha,\o}(\alpha(A))\not=0\}.\]    
\end{proposition}
\begin{proof}
Soient $A\in\a_M$ et $Y\in \o$. Alors $A+Y\in \Ind_M^G(A+Y)$ si et seulement si  $M_{A+Y_\ss}=G_{A+Y_\ss}$, soit $G_{A+Y_\ss}\subseteq M$. \`{A} partir de cette dernière condition on voit aisément que $\a_{M,\o,G-\reg}=\{A\in \a_M\mid \prod_{\alpha\in \Sigma(\g;A_M)}p_{\alpha,\o_{\ss}}(\alpha(A))\not=0\}$.    
\end{proof}
Pour $Y\in \m(F)$, on pose
\begin{equation}\label{YDLiopeq:defaMYG-reg}
a_{M,Y,G-\reg}\eqdef a_{M,(\Ad M)Y,G-\reg}.    
\end{equation}

Soit maintenant $X\in \g(F)$ quelconque, on prend d'abord $M$ un sous-groupe de Levi tel que $X_\ss \in\m(F)$ elliptique et $X\in \Ind_M^G(X_\ss)$ (proposition \ref{prop:condition(H?)}). Puis on prend sur $G_{X}(F)$ l'unique mesure de Haar telle que 
\begin{equation}\label{eq:defmeasureonanyorb}
\begin{split}
|D^\g(X)|_F^{1/2}&\int_{G_{X}(F)\backslash G(F)}f(g^{-1}Xg)\,dg\\
&=\lim_{\substack{A\to 0\\A\in\a_{M,X_\ss,G-\reg}(F)}}|D^\g(A+X)|_F^{1/2}\int_{M_{X_\ss}(F)\backslash G(F)}f(g^{-1}(A+X_\ss)g)\,dg 
\end{split}
\end{equation}
pour tout $f\in C_c^\infty(G(F))$. On doit vérifier que ce procédé donne une mesure bien définie, c'est-à-dire :
\begin{proposition}\label{YDLiopprop:compatibilityHaarmeasures}On a les compatibilités suivantes :    
\begin{enumerate}
    \item la limite du membre de droite de l'égalité \eqref{eq:defmeasureonanyorb} existe, et elle ne dépend pas du choix de $M$ ;
    \item si $X$ est semi-simple, alors la mesure sur $G_X(F)$ définie par l'équation \eqref{eq:defmeasureonanyorb} coïncide avec la mesure définie par l'équation \eqref{eq:HaarmeasureonGXss} ;
    \item si $L$ est un sous-groupe de Levi de $G$ dont l'algèbre de Lie contient $X$ tel que $L_X=G_X$, alors la mesure sur $G_X(F)$ définie par l'équation \eqref{eq:defmeasureonanyorb} en voyant $X$ comme élément de $\g(F)$ coïncide avec la mesure sur $L_X$ définie par la même équation en voyant $X$ comme élément de $\mathfrak{l}(F)$ ;
    \item la mesure sur $G_X(F)$ définie par l'équation \eqref{eq:defmeasureonanyorb} en voyant $X$ comme élément de $\g(F)$ coïncide avec la mesure sur $(G_{X_\ss})_{X_\nilp}(F)$ définie par la même équation en voyant $X_\nilp$ comme élément de $\g_{X_\ss}(F)$ ;
    \item la mesure sur $G_{X+Z}(F)$ coïncide avec la mesure sur $G_X(F)$ pour tous $X\in \g(F)$ et $Z\in\mathfrak{z}(F)$.
\end{enumerate}
\end{proposition}
\begin{proof}~{}
	\begin{enumerate}
		\item On peut réduire le domaine d'intégration à une partie compacte de $G(F)$ indépandante de $A$ selon un résultat connu de Harish-Chandra (cf. \cite[lemme 14.1]{Kottbook}, le groupe $H$ mentionné dans son article est notre groupe $M_{X_\ss}$ ici). L'existence de la limite résulte alors du théorème de convergence dominée. Soit ensuite $M'$ un autre sous-groupe de Levi tel que $X_\ss \in\m'(F)$ elliptique et $X\in \Ind_{M'}^G(X_\ss)$. On prend $h\in G_{X_\ss}(F)$ tel que $M'=h^{-1}Mh$ (proposition \ref{prop:condition(H?)}). Alors 
		\begin{align*}
			\int_{M_{X_\ss}(F)\backslash G(F)}f(g^{-1}(A+X_\ss)g)\,dg =\int_{M_{X_\ss}'(F)\backslash G(F)}f(g^{-1}h^{-1}(A+X_\ss)hg)\,dg \\
			=\int_{M_{X_\ss}'(F)\backslash G(F)}f(g^{-1}(h^{-1}Ah+X_\ss)g)\,dg. 
		\end{align*} 
		Puis la congugaison par $h^{-1}$ induit un isomorphisme de $\a_{M,X_\ss,G-\reg}$ à $\a_{M',X_\ss,G-\reg}$. La limite du membre de droite de l'égalité \eqref{eq:defmeasureonanyorb} ne dépend en conséquence pas du choix de $M$. 
		\item Cela résulte du fait que si $X=X_\ss$ alors $X_\ss\in\Ind_M^G(X_\ss)$ et donc
		$M_{X_{\ss}}=G_{X_{\ss}}=G_X$ (proposition \ref{YDLiopprop:whenXinIndMGX}). Ainsi, si l'on munit
		$G_X(F)$ dans le domaine d'intégration du membre de gauche de l'équation \eqref{eq:defmeasureonanyorb} de la mesure de Haar définie par l'équation \eqref{eq:HaarmeasureonGXss}, alors l'équation \eqref{eq:defmeasureonanyorb} sera une conséquence du fait que l'on peut réduire le domaine d'intégration en une partie compacte indépandante de $A$ selon le résultat de compacité de Harish-Chandra précédemment mentionné.
		\item On prend $M$ un sous-groupe de Levi de $L$ tel que $X_\ss \in\m(F)$ elliptique et $X\in \Ind_M^L(X_\ss)$, alors la condition $L_X=G_X$ implique que $X\in \Ind_L^G(X)=\Ind_L^G(\Ind_M^L(X_\ss))=\Ind_M^G(X_\ss)$ (proposition \ref{YDLiopprop:whenXinIndMGX}). Ainsi, si l'on munit $G_X(F)$ dans le domaine d'intégration du membre de gauche de l'équation \eqref{eq:defmeasureonanyorb} de la même mesure de Haar que $L_X(F)$, alors 
		\begin{align*}
			\int_{G_{X}(F)\backslash G(F)}&f(g^{-1}Xg)\,dg\\
            &= \int_{L(F)\backslash G(F)}\int_{L_{X}(F)\backslash L(F)}f(g^{-1}l^{-1}Xlg)\,dl\,dg   \\
			&=\int_{L(F)\backslash G(F)}\lim_{\substack{A\to 0\\A\in\a_{M,X_\ss,G-\reg}(F)}}\frac{|D^\l(A+X)|_F^{1/2}}{|D^{\l}(X)|_F^{1/2}}\int_{M_{X_\ss}(F)\backslash L(F)}f(g^{-1}l^{-1}(A+X_\ss)lg)\,dl\,dg\\
			&=\lim_{\substack{A\to 0\\A\in\a_{M,X_\ss,G-\reg}(F)}}\frac{|D^\l(A+X)|_F^{1/2}}{|D^{\l}(X)|_F^{1/2}}\int_{L(F)\backslash G(F)}\int_{M_{X_\ss}(F)\backslash L(F)}f(g^{-1}l^{-1}(A+X_\ss)lg)\,dl\,dg \\
			&=\lim_{\substack{A\to 0\\A\in\a_{M,X_\ss,G-\reg}(F)}}\frac{|D^\l(A+X)|_F^{1/2}}{|D^{\l}(X)|_F^{1/2}}\int_{M_{X_\ss}(F)\backslash G(F)}f(g^{-1}(A+X_\ss)g)\,dg\\
            &=\lim_{\substack{A\to 0\\A\in\a_{M,X_\ss,G-\reg}(F)}}\frac{|D^\g(A+X)|_F^{1/2}}{|D^{\g}(X)|_F^{1/2}}\int_{M_{X_\ss}(F)\backslash G(F)}f(g^{-1}(A+X_\ss)g)\,dg,
		\end{align*}
		ici toute interversion d'opérations est justifié par le résultat de compacité de Harish-Chandra précédemment mentionné. 
		\item Cela découle du point 2 de la proposition \ref{chap3prop:indprop} et d'un raisonnement similaire à celui du point précédent. 
        \item Pour $X$ semi-simple cela découle de la définition. Pour $X$ général cela découle du cas semi-simple et l'équation \eqref{eq:defmeasureonanyorb} 
        \qedhere
	\end{enumerate}    
\end{proof}

Remarquons que pour tous $X\in \g(F)$ et $g\in G(F)$, la mesure sur $G_{(\Ad g)X}(F)$ correspond à la mesure sur $G_X(F)$ par conjugaison par $g$.

Soit $P$ un sous-groupe parabolique de $G$. On définit la fonction module pour tout $p \in P(F)$ par
\[\delta_P (p)=\delta_P^G(p)\eqdef e^{2\rho_P(H_{M_0}(p))} = |\det(\Ad(p);\mathfrak{n}_P (F))|.\]

Pour tout $L\in \L^G(M_0)$ on munit $L(F)\cap K$ d'une mesure de Haar. On munit $N_P(F)$, pour tout $P\in\F^L(M_0)$, d'une mesure de Haar. On écrit $\gamma^L(P)>0$ la constante telle que
\begin{align*}
\int_{L(F)} f(x)\,dx &=\gamma^L(P)\int_{M_P(F)}\int_{N_P(F)}\int_{L(F)\cap K} f(mnk)\,dk\,dn\,dm\\
&=\gamma^L(P)\int_{N_P(F)}\int_{M_P(F)}\int_{L(F)\cap K} f(nmk)e^{-2\rho_P^L(H_{M}(m))}\,dk\,dm\,dn
\end{align*}
pour toute fonction $f\in L^1(L(F))$. 

Finalement, on munit $\n_P(F)$, pour tout $P\in\F^L(M_0)$, de la mesure de Haar qui correspond à celle de $N_P(F)$ au sens suivant : prenant une filtration par des sous-groupes normaux $N_P=N_0\supseteq N_1\supseteq \cdots \supseteq N_r=\{1\}$ stable par l'action par $A_P$ telle que $N_k/N_{k+1}$ est abélien et $[N_0,N_k]\subseteq N_{k+1}$ pour tout $k$ (une telle filtration existe toujours : supposons que $N_{k+1}$ est défini, on pose $N_k$ le sous-groupe unipotent de l'algèbre de Lie l'espace associé au poids maximal de $A_P$ dans $\n_P/\n_{k+1}$), alors $N_{K}/N_{k+1}$ s'identifie canoniquement avec $\n_k/\n_{k+1}$ et la mesure de $N_{K}/N_{k+1}$ coïncide avec celle de $\n_k/\n_{k+1}$. Bien sûr cette compatibilité ne dépend pas de la filtration choisie.

Soient $M\in \L^G(M_0)$ et $Y\in \m(F)$. Soient $P\in \P^G(M)$ et $X\in \Ind_M^G(Y)(F)\cap ((\Ad M(F))Y+\n_P(F))$. On a $P_X=G_X$ (point 6 de la proposition \ref{chap3prop:indprop}). En particulier $P_X(F)$ est muni d'une mesure de Haar. On munit $P(F)$ de la mesure produit de celle de $M(F)$ avec celle de $N_P(F)$ (la mesure sur $P_X(F)\backslash P(F)$ n'est pas $P(F)$-invariante). Alors on dispose de la formule d'intégration suivante :
\begin{align*}
\int_{G_X(F)\backslash G(F)}&f\left((\Ad g^{-1})X\right)\,dg \\
&=\gamma^G(P)\int_{P_X(F)\backslash
 P(F)}\int_K f\left((\Ad k^{-1})(\Ad p^{-1})X\right)\,dp\,dk \\
&=\gamma^G(P)|D^{\mathfrak{g}}(X)|^{-1/2}|D^{\mathfrak{m}}(Y)|^{1/2}\\
&\,\,\,\,\,\,\,\,\,\,\,\,\times\int_{M_{Y}(F)\backslash M(F)}\int_{\n_P(F)}\int_{K}f\left((\Ad k^{-1})((\Ad m^{-1})Y+U)\right)\,dk\,dU\,dm,      
\end{align*}
pour tout $f\in L^1((\Ad G(F))X)$. En effet, pour la première égalité on applique simplement la décomposition d'Iwasawa. Puis la deuxième égalité  vient du changement de variables $p^{-1}Xp=m^{-1}Ym+U$ avec $(m,U)$ défini sur $((M_{Y}\backslash M) \times \n_P)_{G-\reg}(F)$ où $((M_{Y}\backslash M) \times \n_P)_{G-\reg}\eqdef \{(m,U)\in (M_{Y}\backslash M) \times \n_P : m^{-1}Ym+U\in \Ind_M^G(Y)\}$ est un ouvert de Zariski dense de $(M_{Y}\backslash M) \times
 \n_P$, le Jacobien de ce morphisme vaut $|D^{\g}(X)|^{-1/2}|D^{\m}(Y)|^{1/2}$.

\subsection{Fonction \texorpdfstring{$R_P$}{RP}}
Soit $X\in \g(F)$ standard. Prenons $\MR^{X}=\MR^{G,X}$ le facteur de Levi semi-standard d'un élément de $\cR^G(X)$. On souligne que dans son article \cite{Ch17}, Chaudouard fixe un $\MR^{X}$ spécifique. Cependant, nous vérifierons que l'abandon de ce choix n'altère pas la véracité des résultats dans le présent texte.

Pour tous $P\in \P^G(\MR^{X})$ et $g\in G(F)$ on pose
\[R_{P}(g)=R_{P,X}(g)\eqdef H_P(w_Pg)\in a_{\MR^{X}}\]
où $w_P\in\uW^{G,0}$ vérifie $(\Ad w_P^{-1})P\in\cR^G(X)$. D'après le lemme \ref{lem:wP}, un tel élément non seulement existe mais est unique à translation à gauche près par un élément $w\in\uW^{G,0,\MR^{X}}$. Soit $
w_Pg=mnk$ la décomposition d'Iwasawa suivant $G(F)=\MR^{X}(F)N_P(F)K$. Alors on a la décomposition d'Iwasawa $ww_Pg=(wm)nk$. En conséquence
\begin{align*}
\langle \chi, H_P(ww_Pg)\rangle =\log |\chi(wm)|=\log|\chi(w)|+\log|\chi(m)|=\langle \chi, H_P(w_Pg)\rangle. 
\end{align*}
Ainsi $R_P$ est indépendante du choix de $w_P$.

Semblable à $(-H_P(g))_{P\in \P^G(\MR^{X})}$, on montre que cette nouvelle famille est aussi orthogonale.

\begin{lemma}
La famille $(-R_P(g))_{P\in \P^G(\MR^{X})}$ est orthogonale au sens d'Arthur.
\end{lemma}

\begin{proof}
Abrégeons $\MR^{X}$ en $M$. Il s'agit de vérifier que pour $P$ et $P'$ adjacents, le vecteur
\[-R_P(g)+R_{P'}(g)\]
appartient à la droite engendrée par l’unique élément de $\Sigma(\mathfrak{p};A_M)^\vee\cap (-\Sigma(\mathfrak{p}';A_M)^\vee)$. Cette
droite est la droite $a_M^{M_Q}$ pour l'unique sous-groupe parabolique $Q$ de $G$ contenant $P$ et $P'$. Si l'on revient à la définition de $R_P$ et $R_{P'}$, on voit qu'il s'agit de prouver que l'on a
\[H_Q(w_Pg)=H_Q(w_{P'}g).\]
En raison du lemme \ref{lem:2.10.2}, on a
\[w_{P'}^{-1}\in G_X(F) w_P^{-1}M_{Q,X_\ss}(F)P_{X_\ss}'(F).\]
Puisque $(\Ad w_P^{-1})P\in \cR^G(X)$, on a $G_X\in (\Ad w_P^{-1})P$. Par suite
\[w_{P'}^{-1}\in w_P^{-1} Q(F),\]
de sorte que $w_Pw_{P'}^{-1}\in Q(F) \cap \uW^{G,0}=\uW^{G,0,M_Q}$. Ce qui nous permet de conclure que $H_Q(w_Pw_{P'}^{-1})=0$ et 
\[H_Q(w_Pg)=H_Q(w_Pw_{P'}^{-1})+H_Q(w_{P'}g)=H_Q(w_{P'}g),\]
ce qu'il fallait.
\end{proof}

Soit $g \in G(F)$. Attendu que la famille $(-R_P(g))_{P\in\P^G(\MR^{X})}$ est orthogonale, on dispose pour tout $Q \in \F^G(\MR^{X})$ du point $-R_Q(g) \in a_{Q}$ par projection orthogonale évidente. On vérifie qu'on a
\[ R_Q(g) = H_Q(w_Qg)\]
où $w_Q \in \uW^{G,0}$ vérifie $(\Ad w_Q^{-1})Q\in\cLS^G(X)$ (cf. lemme \ref{lem:wP}). Comme l'expression $H_Q(w_Qg)$ est indépendante de toute translation à gauche de $w_Q$ par un élément de $\uW^{G,0,(\MR^{X})_Q}$, elle ne dépend pas du choix de $w_Q$. On pourrait penser à la famille $(-R_P(g))_P$ comme étant la version tordue de la famille $(-H_P(g))_P$.

\subsection{Action du normalisateur de \texorpdfstring{$\MR^{X}$}{MRG,X} dans \texorpdfstring{$\uW^{G,0}$}{WG,0}}
Pour tous $w\in \text{Norm}_{\uW^{G,0}}(\MR^{X})$ et $P\in\P^G(\MR^{X})$, le groupe $(\Ad w^{-1})P$ appartient encore à $\P^G(\MR^{X})$. Nous avons la formule suivante :
\[R_P(g)=w\cdot R_{(\Ad w^{-1})P}(g),\,\,\,\,\forall g\in G(F),\]
où l'on note $w\cdot$ l'action naturelle de $\text{Norm}_{\uW^{G,0}}(\MR^{X})$ sur l'espace $a_{\MR^{X}}$ et $Y_{\MR^{X}}$ la projection orthogonale d'un vecteur $Y\in a_{M_0}$ sur $a_{\MR^{X}}$. En effet, soit $w_Pg=mnk$ la décomposition d'Iwasawa suivant $G(F)=\MR^{X}(F)N_P(F)K$. Alors
\begin{align*}
R_{(\Ad w^{-1})P}(g)&=H_{(\Ad w^{-1})P}(w_{(\Ad w^{-1})P}g)=H_{(\Ad w^{-1})P}(w^{-1}w_Pg)\\
&=H_{(\Ad w^{-1})P}((w^{-1}mw)(w^{-1}nw)(w^{-1}k))   \\
&=w^{-1}\cdot H_P(m)=w^{-1}\cdot R_P(g),
\end{align*}
l'avant dernière égalité résulte du fait $\uW^{G,0}\subseteq K$. 
 
Plus généralement si $L\in \L^G(\MR^{X})$ et $Q\in\F^G(L)$ nous avons la suivante formule dans $a_{L}$ :
\begin{equation}\label{eq:5.5.2}
R_Q(g)=w\cdot R_{(\Ad w^{-1})Q}(g),\,\,\,\,\forall g\in G(F).    
\end{equation}

\subsection{Action du centralisateur}

\begin{lemma}\label{lem:actcent}
Soit $h\in G_X(F)$. Soit $L$ un sous-groupe de Levi de $G$ contenant $\MR^{X}$. La famille de vecteurs $(-R_P(h))_{P\in\P^G(L)}$ est une famille constante. De surcroît, pour tout $g\in G(F)$, la famille orthogonale $(-R_P(hg))_{R\in\P^G(L)}$ s'obtient de la famille $(-R_P(g))_{R\in\P^G(L)}$ par translation par le vecteur constant $-R_P(h)$.
\end{lemma}

\begin{proof}
Il suffit de prouver le lemme pour $L=\MR^{X}$. Abrégeons $\MR^{X}$ en $M$. Soit $P\in \P^G(M)$, et $w_P\in \uW^{G,0}$ est tel que $(\Ad w_P^{-1})P$ vaut $\widetilde{P}$, l'image de $P$ par l'application \eqref{R}. Le facteur de Levi semi-standard de $\widetilde{P}$ est $M_{\widetilde{P}}\eqdef (\Ad w_P^{-1})M$. L'action de conjugaison par $w_P$ induit un isomorphisme d'espaces vectoriels $a_{M_{\widetilde{P}}}\xrightarrow{\sim}a_M$ indépendant du choix de $w_P$ et pour lequel on a $R_P(g)=w_P\cdot H_{\widetilde{P}}(g)$. On a donc pour tous $h\in G_X(F)$ et $g\in G(F)$
\begin{equation}\label{eq:5.6.1}
R_P(hg)=w_P\cdot H_{\widetilde{P}}(hg)=w_P\cdot H_{\widetilde{P}}(h)+w_P\cdot H_{\widetilde{P}}(g)=R_P(h)+R_P(g).  
\end{equation}
En particulier, la seconde partie de l'énoncé se déduit de la première partie, puisque $G_X\subseteq P'$ pour tout $P'\in\cR^G(X)'$ selon le point 6 de la proposition \ref{chap3prop:indprop}.

Soit $h\in G_X(F)=(G_{X_\ss})_{X_\nilp}(F)$. On cherche ensuite à justifier que $(R_P(h))_{R\in\P^G(M)}$ est une famille constante. Pour y parvenir, constatons de prime abord, grâce au lemme \ref{lem:envLell}, que si l'on pose $H\eqdef \text{envL}(G_{X_\ss};G)$ alors l'élément $X_\ss$ est $F$-elliptique dans $\mathfrak{h}(F)$ et $M$ est un sous-groupe de Levi semi-standard de $H$, et que l'application 
\begin{align*}
P\in \cR^G(X)\longmapsto P\cap H \in\cR^H(X)  
\end{align*}
est surjective, on en déduit de ce fait $R_{P}^G(h)=R_{P\cap H}^H(h)$ pour tout $h\in H(F)$, ici les exposants ont pour objectif de préciser les groupes dans lesquels chaque application est réalisée. On reprend les notations de la sous-section \ref{subsec:ssssschoix}. L'élément $X$ est dans $\mathfrak{h}(F)$ et est standard (relativement à la base ordonnée pour $H$ qui se déduit de $e$ et à la même relation d’ordre totale sur $\Irr_F$). Quitte à remplacer $G$ par $H$ on peut alors supposer que $X_\ss$ est $F$-elliptique dans $G$. Ainsi $Q\in\cR^G(X)\mapsto Q_{X_\ss} \in\cR^{G_{X_\ss}}(X_\nilp)$ est une bijection. En travaillant avec des matrices par blocs,
la démonstration dans le cas spécial où $X$ est nilpotent (\cite[lemme 5.6.1]{Ch17} : comme indiqué précédemment la théorie y est seulement développée pour $G=\GL_{n,F}$, mais il suffit de remplacer le mot « déterminant » partout par « norme réduite » pour étendre la démonstration du lemme au $G$ du type GL quelconque) peut être aisément étendue au cas général.
\end{proof}

\subsection{Poids}
Considérons la $(G,\MR^{X})$-famille 
\begin{equation}\label{YDLiopeq:defnewGMfamilyinIOP}
v_{P,X}(\lambda,g)\eqdef \exp(\langle\lambda,-R_P(g)\rangle),\,\,\,\,\lambda\in ia_{\MR^{X}}^\ast,    
\end{equation}
avec $P\in \P^G(\MR^{X})$. On note $v_{L,X}^Q(g)$ la valeur associée en 0 à cette famille et au couple $(L,Q)$ selon le procédé expliqué dans les équations \eqref{eq:GMfamilletoGL} et \eqref{eq:GMfamilletoMQM}, pour $L\in\L^G(\MR^{X})$ et $Q\in\F^G(L)$. Lorsque $Q \in \P^G(L)$, cette fonction vaut
identiquement 1.

\begin{lemma}\label{lem:poidsinvnorm}
La fonction
\[g\in G(F)\longmapsto v_{L,X}^Q(g)\]
est invariante à droite par $K$ et à gauche par le centralisateur $G_X(F)$ de $X$ dans $G(F)$. De plus, pour tout $w\in\text{Norm}_{\uW^{G,0}}(\MR^{X})$ et tout $g \in G(F)$, on a
\[v_{L,X}^Q(g)=v_{(\Ad w)L,X}^{(\Ad w)Q}(g).\]
\end{lemma}
\begin{proof}
L'invariance à droite par $K$ est évidente. Quant à l'invariance par le centralisateur, il s'agit de montrer $v_{L,X}^Q(g)=v_{L,X}^Q(hg)$ pour $h\in G_X(F)$. En tenant compte du lemme \ref{lem:actcent} on a pour $\lambda\in ia_L^\ast$
\begin{align*}
\sum_{P\in \P^Q(L)}\exp(\langle-\lambda,R_P(hg)\rangle)\theta_P^Q(\lambda)^{-1}= \exp(\langle-\lambda,R_P(h)\rangle) \sum_{P\in \P^Q(L)}\exp(\langle-\lambda,R_P(g)\rangle)\theta_P^Q(\lambda)^{-1}.
\end{align*} 
L'opération de passage à la limite $\lambda\to 0$ donne l'égalité voulue. Enfin l'invariance du poids par le normalisateur $\text{Norm}_{\uW^{G,0}}(\MR^{X})$ découle de la définition et la formule (\ref{eq:5.5.2}).
\end{proof}

On a introduit la $(G,M)$-famille usuelle chez Arthur $v_{P}(\lambda,g)\eqdef \exp(\langle\lambda,-H_P(g)\rangle)$. Clairement $v_P(\lambda,g)=v_{P,A}(\lambda,g)$ pour tout $A\in \m(F)$ avec $G_A=M_A$, selon la proposition \ref{prop:LX=GX}.

\begin{proposition}\label{prop:paradesccomp}
Soient $L\in \L^G(M_0)$ et $Q\in \P^G(L)$. Soient $Y\in\l(F)$ standard pour $L$ et $X\in\Ind_L^G(Y)(F)$ standard pour $G$. Alors pour tout $P^Y\in \cR^{L}(Y)$ fixé, il existe $w\in \uW^{G,0}$, $P^Y\in \cR^{L}(Y)$ et  $P^X\in\cR^{G}(X)$ tels que, en notant $\MR^{L,Y}$ et $\MR^{G,X}$ respectivement les facteurs de Levi semi-standard de $P^Y$ et de $P^X$, on ait
\[\begin{cases}
\MR^{G,X} = (\Ad w)\MR^{L,Y},\\
P^X =(\Ad w)(P^YN_Q), \\
X\in(\Ad w)(Y+\n_Q(F)).
\end{cases} \]
\end{proposition}


\begin{proof}
On fixe $M=\MR^{L,Y}$ le facteur de Levi semi-standard d'un sous-groupe de Richardson généralisé de $Y$ dans $L$, il existe $w_1\in \uW^{G,0}$ tel que $\sigma\eqdef X_\ss=(\Ad w_1)Y_\ss$, par hypothèse $X$ est le représentant standard de $\Ind_{(\Ad w_1)M}^G(\sigma)(F)$ et $Y$ le représentant standard de $\Ind_{M}^L((\Ad w_1^{-1})\sigma)(F)$. 
Constatons que $(\Ad w_1)M\subseteq H\eqdef \envL(G_{\sigma};G)$. Quitte à travailler sur les composantes irréductibles de $H$ on peut supposer que $X_\ss$ est $F$-elliptique dans $G$. En particulier $Q\in\cR^G(X)\mapsto Q_{X_\ss} \in\cR^{G_{X_\ss}}(X_\nilp)$ est une bijection. On va traiter uniquement le cas où $\sigma=0$ afin de ne pas trop alourdir les notations, puisque en travaillant sur les matrices par blocs on peut étendre la preuve ci-dessous au cas général. On reprend les notations de la sous-section \ref{subsec:ssssschoix} en dessous, en particulier $G=\Aut_{D}(V)$ avec $V$ un $D$-module à droite libre de rang $n$ muni d'une base ordonnée $e$, et $M_0$ le sous-groupe diagonal. Sans perte de généralité on suppose que $L=\Aut_{D}({}^1V)\times \Aut_{D}({}^2V)$ avec ${}^lV$ ($l=1,2$) le sous-$D$-module à droite libre de rang ${}^ln$ engendré par la base ${}^le$, où ${}^1e$ regroupe les plus petits ${}^1n$ éléments de $e$ et ${}^2e$ regroupe les plus grands ${}^2n$ éléments de $e$ (le cas où $L$ est un sous-groupe de Levi semi-standard quelconque s'obtient par la transitivité de l'orbite induite et une $\uW^{G,0}$-conjugaison). Supposons que ${}^lV$ est muni d'une décomposition en somme directe de $D$-modules libres
\[{}^lV=\bigoplus_{1\leq i\leq j\leq {}^lr}{}^lV_j^i\]
telle que le rang ${}^ld_j\geq 0$ de ${}^lV_j^i$ ne dépende que de $j$. Pour tous $1\leq i\leq j\leq {}^lr$ et $1\leq k\leq {}^ld_j$ soit ${}^le_{k,j}^i\in {}^lV_j^i$ tel que la famille $({}^le_{k,j}^i)_{1\leq k\leq {}^ld_j}$ soit une base du $D$-module ${}^lV_j^i$, ordonnée de la façon suivante :  ${}^le_{k,j}^i<{}^le_{k',j'}^{i'}$ si l'une des conditions suivantes est satisfaite
\begin{itemize}
    \item[$-$] $i<i'$ ;
    \item[$-$] $i=i'$ et $j>j'$ ;
    \item[$-$] $i=i'$, $j=j'$ et $k<k'$, 
\end{itemize}
tous ceci tels que $Y\in \l(F)$ soit le nilpotent standard défini par, pour $l=1,2$,
\[Y{}^le_{k,j}^i=\begin{cases}
{}^le_{k,j}^{i-1} & \text{si }i>1,\\
0 & \text{si }i=1,
\end{cases}\]
pour tous $1\leq i\leq j\leq {}^lr$ et $1\leq k\leq {}^ld_j$. 

Alors en utilisant $V={}^1V\oplus{}^2V$ et en tenant compte que $e$ est la base ordonnée $({}^1e_{k,j}^i)_{1\leq j\leq {}^1d_{{}^1r},1\leq k\leq {}^1d_j}\sqcup ({}^2e_{k,j}^i)_{1\leq j\leq {}^2d_{{}^2r},1\leq k\leq {}^2d_j}$ avec ${}^{l}e_{k,j}^i<{}^{l'}e_{k',j'}^{i'}$ si l'une des conditions suivantes est satisfaite
\begin{itemize}
    \item[$-$] $l<l'$ ;
    \item[$-$] $l=l'$ et ${}^{l}e_{k,j}^i<{}^{l'}e_{k',j'}^{i'}$,
\end{itemize}
on écrit $Y_{\text{ind}}\in \n_Q(F)$, lorsque $Q\eqdef L\cdot P_0^e\in \P^G(L)$ ($P_0^e$ est le sous-groupe de $G$ des matrices triangulaires supérieures relativement à la base $e$), l'élément  
\[Y_{\text{ind}}{}^2e_{s-({}^2d_{t_s+1}+\cdots +{}^2d_{{}^2r}),t_s}^1=\begin{cases}
{}^1e_{s-({}^2d_{t_s+1}+\cdots +{}^2d_{{}^2r}),t_s}^{t_s} & \text{si }t_s\leq {}^1r, s-({}^2d_{t_s+1}+\cdots +{}^2d_{{}^2r})\leq {}^1d_{t_s}, \\
0 & \text{sinon}, 
\end{cases}\]
pour tous $1\leq s\leq {}^2d_1+\cdots +{}^2d_{{}^2r}$ et $1\leq t_s\leq {}^2r$ l'unique entier tel que ${}^2d_{t+1}+\cdots +{}^2d_{{}^2r}< s\leq {}^2d_{t}+\cdots +{}^2d_{{}^2r}$ ; et lorsque $Q\eqdef L\cdot \overline{P_0^e}\in \P^G(L)$ ($\overline{P_0^e}$ est l'opposé de $P_0^e$), \[Y_{\text{ind}}{}^1e_{s-({}^1d_{t_s+1}+\cdots +{}^1d_{{}^1r}),t_s}^2=\begin{cases}
{}^2e_{s-({}^1d_{t_s+1}+\cdots +{}^1d_{{}^1r}),t_s}^{t_s} & \text{si }t_s\leq {}^2r, s-({}^1d_{t_s+1}+\cdots +{}^1d_{{}^1r})\leq {}^2d_{t_s}, \\
0 & \text{sinon}, 
\end{cases}\]
pour tous $1\leq s\leq {}^1d_1+\cdots+{}^1d_{{}^1r}$ et $1\leq t_s\leq {}^1r$ l'unique entier tel que ${}^1d_{t+1}+\cdots +{}^1d_{{}^1r}< s\leq {}^1d_{t}+\cdots +{}^1d_{{}^1r}$. Le représentant standard $X\in \g(F)$ de $\Ind_L^G(Y)(F)$ sera $\uW^{G,0}$-conjugué à $Y+Y_{\text{ind}}$, pour le voir on peut passer par la description de l'orbite induite en termes des diviseurs élémentaires, cf. \cite[théorème A.12]{YDL23b}.


Soit $P^Y\in \cR^L(Y)$. Soit $w\in \uW^{G,0}$ tel que $X=(\Ad w)(Y+Y_{\text{ind}})$. On a clairement $P^X\eqdef (\Ad w)(P^YN_Q)\in \cR^G(X)$. Enfin $\MR^{G,X} = (\Ad w)\MR^{L,Y}$ en est bien sûr une conséquence immédiate.
\end{proof}

\begin{corollary}\label{coro:choixtow}
Soient $L\in \L^G(M_0)$ et $\o\subseteq \mathfrak{l}(F)$ une $L(F)$-orbite. Soit $X$ le représentant standard de $\Ind_L^G(\o)(F)$. Alors il existe $w\in \uW^{G,0}$ tel que 
\begin{enumerate}
    \item $(\Ad w^{-1})\MR^{X}\subseteq L$ ;
    \item $\o=\Ind_{(\Ad w^{-1})\MR^X}^L((\Ad w^{-1})X_\ss)$.
\end{enumerate} 
\end{corollary}
\begin{proof}
Cela est une conséquence immédiate de la proposition précédente. \qedhere

\end{proof}

\subsection{Intégrale orbitale pondérée locale}

On note $\S(V(F))$ l'espace de Schwartz-Bruhat d'un espace vectoriel $V$ sur $F$, muni de sa topologie usuelle, cf. \cite[page 61]{Bruhat1961}. Plus précisément, si $F$ est non-archimédien, $\S(V(F))$ est l'espace des fonctions localement constantes à support compact sur $V(F)$, muni de la topologie discrète ; et si $F$ est archimédien, $\S(V(F))$ est l'espace des fonctions de classe Schwartz usuelle sur $V(F)$, muni de la topologie engendrée par les semi-normes $\|-\|_{a,b}$ pour $a$ et $b$ entiers positifs, avec $\|-\|_{a,b}=\sup_{|\alpha|\leq a,|\beta|\leq b} \|x^\alpha \frac{\partial}{\partial x^\beta}-\|_{L^\infty(\g(F))}$. Une fonctionnelle sur $\g(F)$ est dite tempérée si elle est une fonctionnelle continue sur $\S(\g(F))$.

\begin{definition}[Intégrale orbitale pondérée locale]\label{def:NIOPl}
Soient $f\in\S(\g(F))$, $L\in \L^G(M_0)$, $Q\in\F^G(L)$.
\begin{enumerate}
    \item Soit $\o\subseteq \mathfrak{l}(F)$ une $L(F)$-orbite. Soit $X\in \g(F)$ le représentant standard de $\Ind_L^G(\o)(F)$ (cf. définition \ref{def:repstd}). Il existe $w\in \uW^{G,0}$ tel que $(\Ad w^{-1})\MR^{X}\subseteq L$ et $\o=\Ind_{(\Ad w^{-1})\MR^X}^L((\Ad w^{-1})X_\ss)$ (cf. corollaire \ref{coro:choixtow}). On pose,
    \begin{equation}\label{eqdef:NIOPl}
    J_L^Q(\o,f)\eqdef |D^{\g}(X)|_{F}^{1/2} \int_{G_{X}(F)\backslash G(F)} f((\Ad g^{-1})X)v_{(\Ad w)L,X}^{(\Ad w)Q}(g)\,dg.
    \end{equation}
    \item Soit $Y\in \mathfrak{l}(F)$. On pose
    \begin{equation}
    J_L^Q(Y,f)\eqdef J_L^Q((\Ad L(F))Y,f).
    \end{equation}
\end{enumerate}
\end{definition}

La définition dépend, à priori, du choix des matrices carrées ${}^p\textbf{X}$ pour $p\in \Irr_F$, du choix de la relation d'ordre totale sur $\Irr_F$, aussi du choix de la base ordonnée $e$ pour $G$, puis du choix de $\MR^{X}$, enfin du choix de $w\in \uW^{G,0}$. Nous devrions donc adopter plutôt la notation $J_L^Q(\o,f)=J_L^Q(\o,f)_{({}^p\textbf{X})_p,\leq,e,\MR^X,w}$ pour indiquer ces dépendances. \'{E}tudions la dépendance de l'intégrale posée à l'égard de ces choix : 

\begin{proposition}
L'intégrale \eqref{eqdef:NIOPl} dépend uniqument de $M_0^e$, elle ne dépend ni de $w$, ni de $\MR^X$, ni de $e$, ni de $\leq,$ et ni de $({}^p\textbf{X})_p$.    
\end{proposition}

\begin{proof}
On se donne un sous-groupe de Levi minimal $M_0$ et un sous-groupe compact maximal $K$ de $G$ sont donnés. On suppose que $M_0^e=M_0$. On \'{e}tudiera dans l'ordre la dépendance de l'intégrale de $w, \MR^X, e, \leq,$ enfin $({}^p\textbf{X})_p$.   

Primo, on prend $w_1\in \uW^{G,0}$ et $w_2\in \uW^{G,0}$ tels que $(\Ad w_i^{-1})\MR^{X}\subseteq L$ et $\o=\Ind_{(\Ad w_i^{-1})\MR^X}^L((\Ad w_i^{-1})X_\ss)$ pour $i=1,2$. Le point 2 de la proposition \ref{prop:condition(H?)} et le corollaire \ref{coro:elementisellinRich} nous assurent qu'il existe $w_l\in \uW^{L,0}$ tel que $(\Ad w_l^{-1}w_1^{-1})\MR^{X}=(\Ad w_2^{-1})\MR^{X}$. Ainsi $w_1w_lw_2^{-1}\in \text{Norm}_{\uW^{G,0}}(\MR^{X})$. Or $w_1w_lw_2^{-1}$ conjugue $(\Ad w_2)L$ à $(\Ad w_1)L$, on obtient de ce fait $v_{(\Ad w_1)L,X}^{(\Ad w_1)Q}(g)=v_{(\Ad w_2)L,X}^{(\Ad w_2)Q}(g)$  grâce au lemme \ref{lem:poidsinvnorm}.

Secundo, des choix différents de $\MR^{X}$ sont conjugués sous $\uW^{G,0
}$, on prend alors $\MR_1$ et $\MR_2$ deux sous-groupes de Levi semi-standards de deux éléments de $\cR^G(X)$ et $u\in \uW^{G,0}$ avec $\MR_1=u^{-1}\MR_2u$. Soit $w\in \uW^{G,0}$ tel que $\MR_1\subseteq wLw^{-1}$, on a $\MR_2\subseteq (u w)L(u w)^{-1}$, il nous faut comparer $v_{(\Ad w)L,X}^{(\Ad w)Q}(g)$ (défini à partir d'une $(G,\MR_1)$-famille) à $v_{(\Ad u w)L,X}^{(\Ad u w)Q}(g)$ (défini à partir d'une $(G,\MR_2)$-famille). Soit $P\in \P^G((\Ad w)L)$, alors $w_{(\Ad u
)P,X}=u w_{P,X}$ (encore une fois le membre de gauche est défini à partir de $\MR_1$ tandis que celui situé à droite est défini à partir de $\MR_2$). On a $\exp(-\langle u\cdot\lambda, H_{(\Ad u P)}(w_{
(\Ad u)P,X}g)\rangle)=\exp(-\langle u\cdot\lambda, H_{(\Ad u P)}(u w_{P,X}g)\rangle)=\exp(-\langle u\cdot\lambda,u\cdot H_{P}(w_{P,X}g)\rangle)$ pour tous $g\in G(F)$ et $\lambda\in a_{\mathcal{M}_1}$, comme $\uW^{G,0}$ fixe le produit scalaire sur $a_0$ on en déduit que $v_{(\Ad w)L,X}^{(\Ad w)Q}(g)=v_{(\Ad u w)L,X}^{(\Ad u w)Q}(g)$.

Tertio, on prend $e_1$ et $e_2$ deux bases pour $G$, il existe bien sûr un unique $g\in G(F)$ tel que $g\cdot e_1=e_2$. On a $g^{-1}\uW^{G,0,e_1}g= \uW^{G,0,e_2}$, $g^{-1}M_0^{e_1}g= M_0^{e_2}$ et $g^{-1}K^{e_1}g= K^{e_2}$. La relation $M_0^{e_2}= M_0= M_0^{e_1}$ nous dit que $g=wm_0$ avec $m_0\in M_0(F)$ et $w\in \uW^{G,0,e_1} $, et si $X_1$ est le représentant standard de sa classe de $G(F)$-conjugaison relativement à $e_1$ alors $g^{-1}X_1g=X_2$ est le représentant standard de la même classe de $G(F)$-conjugaison relativement à $e_2$. On peut prendre $w^{-1}\MR^{X_1,e_1}w=\MR^{X_2,e_2}$. On vérifie alors que 
\begin{align*}
J_L^Q(\o,f)_{({}^p\textbf{X})_p,\leq,e_1}&=|D^{\g}(X_1)|^{1/2} \int_{G_{X_1}(F)\backslash G(F)} f((\Ad h^{-1})X_1)v_{(\Ad w_1)L,X_1}^{(\Ad w_1)Q}(h)\,dh \\
&=|D^{\g}(X_2)|^{1/2} \int_{G_{X_2}(F)\backslash G(F)} f((\Ad h^{-1})X_2)v_{(\Ad w^{-1}w_1w)L,X_2}^{(\Ad w^{-1}w_1w)Q}(gh)\,dh\\
&=|D^{\g}(X_2)|^{1/2} \int_{G_{X_2}(F)\backslash G(F)} f((\Ad h^{-1})X_2)v_{(\Ad w^{-1}w_1w)L,X_2}^{(\Ad w^{-1}w_1w)Q}(wh)\,dh=J_L^Q(\o,f)_{({}^p\textbf{X})_p,\leq,e_2}. 
\end{align*}

Quarto, des choix différents de la relation d'ordre totale sur $\Irr_F$ donne des éléments standard $X$ qui sont conjugués sous $\uW^{G,0,e}$, on retombe dans la situation précédente mais avec $g=w\in \uW^{G,0}$, ainsi $J_L^Q(\o,f)_{\leq}$ ne dépend pas de cette relation d'ordre. 

En dernier lieu, la dépendance du choix des  ${}^p\textbf{X}$ est plus complexe, mais nous démontrerons à travers des calculs explicits que la définition ne dépend pas de ces choix (cf. théorème \ref{thm:IOPcomparaisonI-II} ou point 2 du lemme \ref{lem:diffadjacentR_P}). 

La définition \ref{def:NIOPl} ne dépend, en récapitulation, ni de $w$, ni de $\MR^X$, ni de $e$, ni de $\leq$ et ni de $({}^p\textbf{X})_p$.
\end{proof}

\begin{remark}\label{rem:whychoixrepstd}~{}
\begin{enumerate}
    \item Fixons $M_0$ un sous-groupe de Levi minimal de $G$. Alors 
    $\{K^e\mid \text{ $e$ tel que $M_0^e=M_0$}\}$ est exactement l'ensemble des fixateurs dans $G$ des sommets spéciaux dans l'appartement associé à $A_{M_0}$ de l'immeuble de Bruhat-Tits de $G$ (\cite[(I.3.7), définition (2.2.2)]{BruTit72}).
    \item Au contraire, notre définition dépend de façon essentielle des coefficients « nilpotents » dans la définition d'un élément standard. \`{A} titre d'exemple soient $G=\GL_{2,D}$, $M=\GL_{1,D}^2$ et $\o$ est l'orbite nulle. Posons pour $t\in D^\times$ 
    \[X_t=\begin{bmatrix}0 & t\\ 0&0\end{bmatrix}\in\Ind_M^G(\o).\]
    Alors $X_1$ est le représentant standard de $\Ind_M^G(\o)$. Il est clair qu'aucun des $\cLS^G(X_t)$ et $\cR^G(X_t)$ ne dépendent de $t$, il est donc possible de définir $J_M^G(\o,f)$ en changeant le représentant standard de $\Ind_M^G(\o)$ en $X_t$, on notera $J_M^G(\o,f)_t$ la définition obtenue. Des calculs directs montrent
    \begin{align*}
    J_M^G(\o,f)_t&=|D^{\g}(X_t)|^{1/2} \int_{G_{X_t}(F)\backslash G(F)} f((\Ad g^{-1})X_t)v_{(\Ad w)L,X_t}^{(\Ad w)Q}(g)\,dg\\
    &=|t|^{-1}|D^{\g}(X_1)|^{1/2} 
    \int_{G_{X_1}(F)\backslash G(F)} f((\Ad h^{-1})X_1)v_{(\Ad w)L,X_1}^{(\Ad w)Q}\left(\begin{bmatrix}t & 0\\ 0&1\end{bmatrix}h\right)\,dh\\
    &=|t|^{-1}J_M^G(\o,f)_1+|t|^{-1}(\log|t|)J_G^G(\Ind_M^G(\o),f).    
    \end{align*}
\end{enumerate}

\end{remark}

\begin{theorem}\label{thm:locIOPtemperee}
Pour toute fonction $f\in \S(\g(F))$, l'intégrale \eqref{def:NIOPl} converge absolument. De surcroît, lorsque $F$ est archimédien, l'espace de Schwartz-Bruhat $\S(\g(F))$ est muni de semi-normes $\|-\|_{a,b}=\sup_{|\alpha|\leq a,|\beta|\leq b} \|x^\alpha \frac{\partial}{\partial x^\beta}-\|_{L^\infty(\g(F))}$, alors il existe des certaines constantes $c_{\o}>0$ et $a_{\o}\in \N$  tels que $|J_L^Q(\o,-)|$ soit majoré par $c_{\o}\|-\|_{a_\o,0}$. Une intégrale orbitale pondérée locale est donc une distribution tempérée. 
\end{theorem}

\begin{proof}
La preuve est reportée à la section \ref{sec:3cv}.    
\end{proof}

\section{Intégrale orbitale pondérée en tant que distribution tempérée}\label{sec:3cv}
Cette partie est dédiée à la preuve du théorème \ref{thm:locIOPtemperee}.

\subsection{Situation elliptique}\label{subsec:situationelliptique}

Fixons désormais $p\in \Irr_F$ un polynôme irréductible de coefficient dominant 1 de $F[T]$. On va considérer l'élément ${}^pX$ donné par l'équation \eqref{eq:elementpXellstd}. On reprend les notations ci-dessous, en supprimant les $p$ dans l'exposant, avec pour finalité d'alléger l'écriture.

Supposons que $V$ est muni d'une décomposition en somme directe de $D$-modules libres
\[V=\bigoplus_{1\leq i\leq j\leq r}V_j^i\]
telle que 
\begin{enumerate}
    \item le rang de $V_j^i$ ne dépende que de $j$ est et divisible par $\underline{n}_p$, il vaut $\underline{n}_pd_j$;
    \item $r>1$ ; 
    \item pour tous $1\leq i\leq j\leq r$ et $1\leq k\leq d_j$, il existe $E_{k,j}^i\subseteq V_j^i$ un sous-$D$-module libre de rang $\underline{n}_p$, muni d'une base ordonnée  $({}^le_{k,j}^i)_{1\leq l\leq \underline{n}_p}$, et $e'\eqdef\bigsqcup_{i,j,k,l}({}^le_{k,j}^i)=e$. Ici la base $e'$ de $V$ est ordonnée de la façon suivante :  ${}^{l}e_{k,j}^i<{}^{l'}e_{k',j'}^{i'}$ si l'une des conditions suivantes est satisfaite
    \begin{itemize}
    \item[$-$] $i<i'$ ;
    \item[$-$] $i=i'$ et $j>j'$ ;
    \item[$-$] $i=i'$, $j=j'$ et $k<k'$ ; 
    \item[$-$] $i=i'$, $j=j'$, $k=k'$ et $l< l'$. 
    \end{itemize}
\end{enumerate}

Notons $G=\Aut_{D}(V)$ et $\g=\End_{D}(V)$. Nous avons 
\[\g=\bigoplus_{i,j,k,i',j',k'}\text{End}_{D}(E_{k,j}^{i},E_{k',j'}^{i'}).\]
Soit $X\in \g(F)$ défini par
\begin{equation}\label{elementXell}
X|_{\text{Hom}_{D}(E_{k,j}^{i},E_{k',j'}^{i'})}=\begin{cases}
{}^p\textbf{X}& \text{si }i=i',j=j',k=k',\\
\text{Id}_{\underline{n}_p\times \underline{n}_p} & \text{si }i>1,i'=i-1, j=j',k=k'\\
0 & \text{sinon},
\end{cases}   
\end{equation}
pour tous $i,j,k,i',j',k'$, ici le membre de gauche est la représentation matricielle de l'endomorphisme en quesiton dans les bases ordonnées concernées. L'élément $X$ est standard. Nous travaillerons avec cet élément $X$ dans la suite.

\subsubsection{Description de \texorpdfstring{$\cR^G(X)$}{RG(X)} en terme d'algèbre linéaire.}

Soient
\[J=\{j\in \N_{\geq 1}\mid d_j\not =0\}\,\,\,\,\text{et}\,\,\,\,r=\max(J).\]
L'entier $r$ est l'indice de nilpotence de $X_{\nilp}$ et $d_j$ est la multiplicité du bloc $\text{J}(p,j)$ dans la décomposition de Jordan de $X$ (\cite[section A.2]{YDL23b}). Ici
\[\text{J}(p,j)= \begin{pmatrix} 
{}^h\textbf{X} & \textbf{1} &  &  &  &  \\
     & {}^p\textbf{X} & \textbf{1} &  &  &  \\
     &  & \ddots & \ddots &  &  \\
     &  &  & \ddots & \ddots &  \\
     &  &  &  & {}^p\textbf{X} & \textbf{1} \\
     &  &  &  &  & {}^p\textbf{X} \\
\end{pmatrix}\in \Mat_{\underline{n}_pj}(D)\]
avec $\textbf{1}$ la matrice identité de $\Mat_{\underline{n}_p}(D)$.

Soit $\mathcal{E}^G(X)$ l'ensemble (fini) des applications 
\begin{equation}\label{YDLiopeq:defEGX}
 \varepsilon :\begin{tabular}{ccc}
 $ \{k\in \N\mid 1\leq k\leq r\}\times J$ &$\longrightarrow$  &$\{0,1\}$\\ 
 $(k,j)$&$\longmapsto$ &$\varepsilon_{k,j}$
\end{tabular}   
\end{equation}
qui vérifient pour tout $j\in J$ :
\begin{enumerate}
    \item $\sum_{k=1}^r \varepsilon_{k,j}=j$ ;
    \item pour tout $1\leq k \leq r$, l'application $j\in J\longmapsto \varepsilon_{k,j}\in \{0,1\}$ est croissante.
\end{enumerate}

\begin{lemma}\label{YDLioplem:4.1}
Pour tout $\varepsilon\in \mathcal{E}^G(X)$, soit $V_{\bullet}(\varepsilon)$ le drapeau de $E$ défini par $V_{0}(\varepsilon)=(0)$, et pour $1\leq k\leq r$
\[V_{k}(\varepsilon)=\bigoplus_{j\in J}\bigoplus_{i=1}^{\sum_{l=1}^k\varepsilon_{l,j}}V_j^i.\]
L'application 
\[\varepsilon\longmapsto P_{\varepsilon}\eqdef\text{stab}_G(V_{\bullet}(\varepsilon))\]
est une bijection de $\mathcal{E}^G(X)$ sur $\cR^G(X)$.
\end{lemma}

\begin{proof}
Remarquons que $X_{\ss}$ est elliptique dans $\g(F)$ donc $\cR^{G}(X)'=\cR^{G}(X)$ (corollaire \ref{YDLiop:RG'(X)=RG(X)ell}). De ce fait un sous-groupe parabolique $P$ de $G$ dont l'algèbre de Lie contient $X$ est dans $\cR^{G}(X)$ si et seulement si $P_{X_\ss}$ est dans $\cR^{G_{X_\ss}}(X_\nilp)$ (lemme \ref{lem:LSR'ss}). Ce lemme \ref{YDLioplem:4.1} étant déjà prouvé dans le cas nilpotent (\cite[proposition 3.4.1]{Ch17}), on en déduit le cas général. 
\end{proof}

\subsubsection{Calculs pour des sous-groupes paraboliques adjacents}\label{subsubsec:calculspourdes sous-groupesparaboliquesadjacents}

Fixons $M=\MR^{X}$ le facteur de Levi d'un élément de $\cR^G(X)$ semi-standard. Soient deux sous-groupes paraboliques $P_1$ et $P_2$ de $\P^G(M)$ adjacents, $\widetilde{P_1}$ et $\widetilde{P_2}$ les images respectives sous l'application (\ref{R}), enfin $Q$ le plus petit sous-groupe parabolique qui contient $P_1$ et $P_2$. 

Pour tout $\sigma$ dans le groupe symétrique $\mathfrak{S}_r$ et tout $\varepsilon\in\mathcal{E}^G(X)$, l'application
\[\sigma(\varepsilon) : (k, j) \longmapsto \varepsilon_{\sigma^{-1}(k),j}\]
appartient à $\mathcal{E}^G(X)$. Cela définit une action transitive de $\mathfrak{S}_r$ sur $\mathcal{E}^G(X)$ (\cite[Lemme 3.5.1]{Ch17}).

\begin{lemma}
Il existe $\varepsilon\in \mathcal{E}^G(X)$ et $\tau\in\mathfrak{S}_r$ une transposition ou l'identité tels que $\widetilde{P_1}$ et $\widetilde{P_2}$ soient les stabilisateurs respectifs des drapeaux $V_{\bullet}(\varepsilon)$ et $V_{\bullet}(\tau(\varepsilon))$.
\end{lemma}
\begin{proof}
On voit que la preuve est la même qu'en \cite[Lemme 5.7.1]{Ch17} en s'apercevant que les objets rencontrés au fil des calculs sont compatibles à la décomposition de Jordan (lemme \ref{lem:wP}). Il est opportun de remarquer que l'application \eqref{R} dépend de $M$, et dans son article Chaudouard fixe un $M$ spécifique, mais on vérifie que l'abandon de ce choix n'altère pas la véracité du lemme.
\end{proof}

Le sous-groupe parabolique $\widetilde{Q}$ est minimal parmi les sous-groupes paraboliques contenant strictement $\widetilde{P_1}$, par cette minimalité on voit qu'il existe un unique $1\leq k < r$ tel que le sous-groupe parabolique $\widetilde{Q}$ est le stabilisateur du drapeau 
\[(0)=V_0(\varepsilon)\subsetneq V_1(\varepsilon)\subsetneq \cdots \subsetneq V_{k-1}(\varepsilon)\subsetneq V_{k+1}(\varepsilon)\subsetneq \cdots \subsetneq V_r(\varepsilon)=V\]
qui se déduit de $V_{\bullet}(\varepsilon)$ par suppression du sous-espace $V_{k}(\varepsilon)$. Quitte à échanger les rôles de $\varepsilon$ et $\tau(\varepsilon)$, on peut et on va supposer qu'on a $\varepsilon_{k,j}\leq\varepsilon_{k+1,j}$ pour tout $j \in J$ (\cite[page 208]{Ch17}). Soient
\[J_1=\{j\in J\mid \varepsilon_{k,j}=\varepsilon_{k+1,j}=1\}\,\,\,\,\text{et}\,\,\,\,J_2=\{j\in J\mid \varepsilon_{k,j}=0\,\,\text{et}\,\,\varepsilon_{k+1,j}=1\}.\]
On a toujours $r \in J_1$ et l'ensemble $J_2$ est vide si et seulement si $\widetilde{P_1}=\widetilde{P_2}$. On a alors $V_k(\varepsilon) \subseteq V_k(\tau(\varepsilon))$ et les formules suivantes pour les dimensions
\[\underline{n}_pr_1\eqdef\text{rang}(V_k(\varepsilon)/V_{k-1}(\varepsilon))=\text{rang}(V_{k+1}(\varepsilon)/V_{k}(\tau(\varepsilon)))=\underline{n}_p\cdot\left(\sum_{j\in J_1}d_j\right)>0\]
et
\[\underline{n}_pr_2\eqdef\text{rang}(V_{k}(\tau(\varepsilon))/V_{k}(\varepsilon))=\underline{n}_p\cdot\left(\sum_{j\in J_2}d_j\right)\geq 0.\]

Le raffinement du drapeau $V_{\bullet}(\varepsilon)$
\[(0)=V_0(\varepsilon)\subsetneq V_1(\varepsilon)\subsetneq \cdots \subsetneq V_{k}(\varepsilon)\subseteq V_{k}(\tau(\varepsilon))\subsetneq V_{k+1}(\varepsilon)\subsetneq \cdots \subsetneq V_r(\varepsilon)=V\]
est de stabilisateur $P^{-}\eqdef\widetilde{P_1}\cap\widetilde{P_2}$ ($P^{-}$ n'est en général pas un sous-groupe parabolique de Lusztig-Spaltenstein généralisé pour $X$). Par construction des modules $V_k$ comme somme de $V_j^i$, le drapeau
\[V_{k-1}(\varepsilon)\subsetneq V_{k}(\varepsilon)\subseteq V_{k}(\tau(\varepsilon)) \subsetneq V_{k+1}(\varepsilon)\]
vient avec un scindage $V_{k}(\varepsilon)=V_{k-1}(\varepsilon)\oplus W_1$, $V_{k}(\tau(\varepsilon))=V_{k}(\varepsilon)\oplus W_2$ et $V_{k+1}(\varepsilon)=V_{k}(\tau(\varepsilon))\oplus W_3$ où les $W_1, W_2,W_3$ sont eux aussi des sommes de
$V_j^i$. En outre, $X_\nilp$ induit un isomorphisme de $W_3$ sur $W_1$. On a $\text{rang}(W_1) =
\text{rang}(W_3) = \underline{n}_pr_1$ et $\text{rang}(W_2) = \underline{n}_pr_2$. Le sous-groupe parabolique $P^{-}$ contient $X_\ss$ dans son algèbre de Lie, son facteur de Levi semi-standard $M_{P^{-}}$ est muni d'une projection sur
\[\text{Aut}_D(W_1)\times \text{Aut}_D(W_2)\times \text{Aut}_D(W_3)\]
que l'on note par $m\mapsto(m_1,m_2,m_3)$. 

Pour $A$ une algèbre séparable sur $F$, on notera $\Nrd=\Nrd_{A/F} :A \rightarrow F$ la norme réduite. Si $A$ est de plus une algèbre à division, on notera $\deg A=\deg_FA$ son degré.

\begin{lemma}
Soit $g=mnk$ avec $m\in M_{P^{-}}(F)$, $n\in N_{P^{-}}(F)$ et $k\in K$. Avec les
notations ci-dessus,
\[-R_{P_1,X}(g)+R_{P_2,X}(g)=\frac{1}{\deg D}\log|\Nrd(m_1)^{-1}\Nrd(m_3)|\alpha^\vee\]
où $\alpha^\vee$ est l'unique élément de $\Delta_{P_1}^{Q,\vee}\cap (-\Delta_{P_2}^{Q,\vee})$.
\end{lemma}
\begin{proof}
La preuve est la même qu'en \cite[Lemme 5.7.2]{Ch17} : on sait que l'égalité est vraie à un scalaire près, le point est d'évaluer ce scalaire. La fonction $g\mapsto R_{P_i} (g)$ est invariante à gauche par les éléments $h \in \widetilde{P_i}(F)$ tels que $|\chi(h)| = 1$ pour tout $\chi \in X^\ast(\widetilde{P_i})$. Elle est donc invariante à gauche par $N_{P^{-}}$. Elle est par ailleurs invariante à droite par $K$. On est ramené à prouver l'égalité pour $n$ et $k$ triviaux. On a également $R_{P_i} = w_i \cdot H_{\widetilde{P_i}}$ avec $w_i=w_{P_i,X}$. On est donc ramener à prouver pour tout $m \in M_{P^{-}} (F)$ l'égalité
\begin{equation}\label{eq:(5.7.7)}
-H_{\widetilde{P_1}}(m)+(w_1^{-1}w_2)\cdot H_{\widetilde{P_2}}(m)=\frac{1}{\deg D}\log|\Nrd(m_1)^{-1}\Nrd(m_3)|\beta^\vee 
\end{equation}
où $\beta^\vee$ est l'unique élément de $\Delta_{\widetilde{P_1}}^{\widetilde{Q},\vee}$. Le groupe $M_{\widetilde{P_1}}$ est muni naturellement d'une projection sur $\Aut_D(W_1)$ et on va composer cette projection avec la norme réduite. Cela donne un caractère $\chi$ de $X^\ast(\widetilde{P_1})$ pour lequel $\langle \chi, \beta^\vee\rangle = \deg D$ et $\langle \chi,H_{\widetilde{P_1}}(m)\rangle = \log|\Nrd(m_1)|$. Comme (\ref{eq:(5.7.7)}) est au moins vraie à une constante près, il suffit de la vérifier sur le caractère $\chi$.

Pour continuer, il nous faut comprendre l'action de $w_1^{-1}w_2$. On sait que (cf. \cite[preuve du lemme 5.7.1]{Ch17}) $w_1^{-1}P_2w_1$
est le stabilisateur du drapeau
\[(0)=V_0(\varepsilon)\subsetneq V_1(\varepsilon)\subsetneq \cdots\subsetneq V_{k-1}(\varepsilon)\subsetneq V_{k}(\varepsilon)\oplus W_2\oplus W_3\subsetneq V_{k+1}(\varepsilon)\subsetneq \cdots \subsetneq V_r(\varepsilon)=V.\]
Le groupe $\widetilde{P_2}$ est le stabilisateur du drapeau 
\[(0)=V_0(\varepsilon)\subsetneq V_1(\varepsilon)\subsetneq \cdots\subsetneq V_{k-1}(\varepsilon)\subsetneq V_{k}(\varepsilon)\oplus W_1\oplus W_2\subsetneq V_{k+1}(\varepsilon)\subsetneq \cdots \subsetneq V_r(\varepsilon)=V.\]
L'élément $w \in \uW^{M_{\widetilde{Q}},0}$ d'ordre 2 qui échange $W_1$ et $W_3$ (en induisant $X_{\nilp}$ sur $W_3$) conjugue $\widetilde{P_2}$ en $w_1^{-1}P_2w_1$ tout comme d'ailleurs $w_1^{-1}w_2$. Ces éléments sont donc égaux à un élément de $\uW^{M_{\widetilde{P_2}},0}$ près qui, de toute façon, agit trivialement sur $a_{\widetilde{P_2}}$. Le caractère $\chi\circ (w_1^{-1}w_2)$ est alors celui obtenu par composition de la norme réduite avec la projection $M_{\widetilde{P_1}} \rightarrow \Aut_D(W_3)$. Le lemme est prouvé.
\end{proof}

La décomposition de $V$ en somme des $V_j^i$  induit une projection $\g\rightarrow \Hom_D(W_3,W_1)$ dont on note $U\mapsto U_{1,3}$ la restriction à $\mathfrak{p}^{-}=\m_{P^{-}}\oplus\n_{P^{-}}$ . Soit $\Nrd(U_{1,3})=\Nrd_{\Hom_D(W_3,W_1)}(U_{1,3})$ la norme réduite de la matrice de $U_{1,3}$ dans les bases ordonnées de $W_1$ et $W_3$ extraites de la base ordonnée $e$ de $V$. Observons qu'on a
\[\n_{\widetilde{P_1}}\cap \n_{\widetilde{P_2}}=\Hom_D(W_3,W_1)\oplus\n_{\widetilde{Q}}.\]
\begin{lemma}\label{lem:diffadjacentR_P}
Soit $Y \in \g(F)$ dans l'orbite de $X$ sous $G(F)$.
\begin{enumerate}
    \item Il existe $k \in K$ tel que
    \[U\eqdef kYk^{-1}\in (\Ad M_{P^-})X_\ss \oplus (\n_{\widetilde{P_1}}\cap \n_{\widetilde{P_2}}).\]
    Ici on note $\oplus$ au lieu de $+$ par abus de notation, comme $(\Ad M_{P^-})X_\ss$ vit dans $\m_{P^{-}}$, un sous-espace vectoriel en somme directe avec $\n_{\widetilde{P_1}}\cap \n_{\widetilde{P_2}}$. Le nombre $|\Nrd(U_{1,3})| \in \R$ ne dépend pas du choix de $k$.
    \item Pour tout $g \in G(F)$ tel que $Y = g^{-1}Xg$, on a
    \[-R_{P_1,X}(g)+R_{P_2,X}(g)=\frac{1}{\deg D}\log|\Nrd(U_{1,3})|\alpha^\vee\]
    où $\alpha^\vee$ est l'unique élément de $\Delta_{P_1}^{Q,\vee}\cap (-\Delta_{P_2}^{Q,\vee})$.
\end{enumerate}
\end{lemma}
\begin{proof}
On a $Y = g^{-1}Xg$ pour un certain $g \in G(F)$. Par décomposition d'Iwasawa, on a $g = pk$ avec $p \in P^{-}(F)$. Il s'ensuit que $U = kY k^{-1}$ appartient à la $P^{-}(F)$-orbite de $X$. Comme $G_X \subseteq P^{-}=\widetilde{P}_1\cap \widetilde{P}_2$, la $P^{-}$-orbite de $X$ est l'intersection des orbites de $X$ sous $\widetilde{P}_1$ et $\widetilde{P}_2$, c'est aussi un ouvert dense de $(\Ad M_{P^{-}})X_\ss \oplus (\n_{\widetilde{P_1}}\cap \n_{\widetilde{P_2}})$. D'autre part, si $k\in K$ vérifie $kYk^{-1}\in (\Ad M_{P^{-}})X_\ss \oplus (\n_{\widetilde{P_1}}\cap \n_{\widetilde{P_2}})$ alors il existe $p\in P^{-}(F)$ tel que $kg^{-1}Xgk^{-1}=p^{-1}Xp$, donc $g\in G_X(F)pk\subseteq P^{-}(F)k$. L'élément $k$ est nécessairement la composante sur $K$ de la décomposition d'Iwasawa de $g$ relativement au sous-groupe parabolique $P^{-}$. Il s'ensuit que $k$ est bien défini à une translation à gauche près par un élément de $K \cap P^{-}(F)$. Pour tout élément $p \in P^{-}$ de projection $(m_1,m_3)$ sur $\text{Aut}_D (W_1) \times \text{Aut}_D (W_3)$ et tout $U \in (\Ad M_{P^{-}})X_\ss \oplus (\n_{\widetilde{P_1}}\cap \n_{\widetilde{P_2}})$, on a $(p^{-1}Up)_{1,3}=m_1^{-1}U_{1,3}m_3$, et donc \begin{equation}\label{eq:(5.7.9)}
 \Nrd ((p^{-1}Up)_{1,3})=\Nrd(m_1^{-1})\Nrd(U_{1,3})\Nrd(m_3).   
\end{equation} 
Si, de plus, $p \in K \cap P^{-}(F)$, on aura $|\Nrd(m_1)| = |\Nrd(m_3)| = 1$, car 
\[K=\begin{cases*}
 \GL_n(\O_D)  & \text{si $F$ est non-archimédien} \\
 \{g\in \GL_n(D):g\cdot{}^t\overline{g}=\text{Id}\}& \text{si $F$ est archimédien,} 
\end{cases*}\]
d'où l'indépendance vis-à-vis du choix de $k$ dans l'assertion 1. Prouvons l'assertion 2. Soit $g = pk$ la décomposition d'Iwasawa de $g$ selon
$G(F) = P^{-}(F)K$. Dans ce cas, on a $U = p^{-1}Xp$. Il vient d'après (\ref{eq:(5.7.9)}) et l'égalité $\Nrd(X_{1,3}) = 1$ que $\Nrd(U_{1,3})=\Nrd(m_1^{-1})\Nrd(X_{1,3})\Nrd(m_3)=\Nrd(m_1^{-1})\Nrd(m_3)$, on tombe sur le lemme précédent.
\end{proof}

Soit $I_{1,3}\eqdef \text{Iso}_D(W_3,W_1)\subsetneq \Hom_D(W_3,W_1)$ la sous-variété ouverte sur $F$ des
$D$-isomorphismes de $W_3$ sur $W_1$. Le groupe des $F$-points $I_{1,3}$ est muni de la mesure de Haar normalisée selon l'équation \eqref{eq:HaarmeasureonGXss}.

\begin{lemma}\label{lem:centralizerOIformula}
Il existe une constante $c_{P_1,P_2}(X) > 0$ telle que pour toute fonction $f \in L^1((\Ad G(F))X)$ on ait    
\begin{align*}
\int_{G_X(F)\backslash G(F)}&f(g^{-1}Xg)\,dg\\
&=c_{P_1,P_2}(X)\cdot\int_{(M_{P^{-}})_{X_\ss}(F)\backslash M_{P^{-}}(F)}\int_{I_{1,3}(F)}\int_{\n_{\widetilde{Q}}(F)}\int_K \\
&\hspace{3cm}f(k^{-1}(m^{-1}X_\ss m+y+V)k)|\Nrd(y)|^{\underline{n}_p(r_1+r_2)\deg D}\,dk\,dV\,dy\,dm   
\end{align*}
\end{lemma}
\begin{proof}
Soit $H$ la $P^{-}(F)$-orbite de $X$ : c'est un ouvert dense de $(\Ad M_{P^{-}})X_\ss \oplus (\n_{\widetilde{P_1}}\cap \n_{\widetilde{P_2}})$. On sait qu’il existe une constante $c_1> 0$ telle que pour toute fonction complexe intégrable $\phi$ sur $H$, on a  \begin{align*}  
\int_{G_X(F)\backslash P^{-}(F)}\phi(g^{-1}Xg)\,dg=c_{1}\cdot\int_{H}
\phi(U)|\Nrd(U_{1,3})|^{\underline{n}_pr_2\deg D}\,dU  
\end{align*}
car les deux mesures en question sont définies sur le même espace homogène (de $P^{-}$) relativement invariantes pour la même fonction module $\delta_{P^-}$. De même il existe une constante $c_2> 0$
\begin{align*}  
\int_{H}&\phi(U)|\Nrd(U_{1,3})|^{\underline{n}_pr_2\deg D}\,dU  \\
&=c_2\cdot\int_{(M_{P^{-}})_{X_\ss}(F)\backslash M_{P^{-}}(F)}\int_{I_{1,3}(F)}\int_{\n_{\widetilde{Q}}(F)}\phi(m^{-1}X_\ss m+y+V)|\Nrd(y)|^{\underline{n}_p(r_1+r_2)\deg D}\,dk\,dV\,dy\,dm.
\end{align*}
Cela implique l'existence de $c_{P_1,P_2}(X)$.
\end{proof}

On s'acquiert ainsi une description raisonnable du quotient $G_X(F)\backslash G(F)$ dans la situation elliptique. Pour $Y\in \g(F)$ un élément quelconque, représentant standard de sa $G(F)$-orbite, on peut prendre $H=\text{envL}(G_{Y_\ss};G)$ et $P\in\P^G(H)$, alors $G_Y(F)\backslash G(F)=H_Y(F)\backslash H(F)N_P(F)K$, ce qui nous gratifie d'une description raisonnable du quotient $G_Y(F)\backslash G(F)$ dans la situation générale.

\subsection{Norme}\label{subsec:norme}

On définit une norme abstraite en se fondant sur l'approche de \cite[section 18]{Kottbook}. Par une norme abstraite sur un ensemble $A$ on entend une fonction $\|\cdot\|:A\rightarrow \R_{\geq 1}$. Si $\|\cdot\|_1$ et $\|\cdot\|_2$ sont deux normes abstraites, on dit que $\|\cdot\|_2$ domine $\|\cdot\|_1$ et note $\|\cdot\|_1\prec\|\cdot\|_2$ s'il existe $c>0$ et $e>0$ tels que $\|a\|_1< c\|a\|_2^e$ pour tout $a\in A$. Deux normes abstraites sont dites équivalentes si l'une domine l'autre et vice versa.

Les propriétés mentionnées ci-après se trouvent dans les notes citées de Kottwitz. Soit $A$ un schéma affine de type fini sur $F$, et $\O_A$ l'anneau des fonctions régulières. Fixons $\{f_1,\dots,f_m\}$ un ensemble de générateurs de $F$-algèbre $\O_A$. On pose une norme abstraite sur $A(F)$ par $\|a\|=\max\{1,|f_1(a)|,\dots,|f_m(a)|\}$ pour $a\in A(F)$. Sa classe d'équivalence ne dépend pas de l'ensemble de générateurs choisi, et on appellera une norme sur $A(F)$ toute norme abstraite obtenue ainsi. Pour tous morphisme $\phi:A\rightarrow B$ de schémas affines de type fini, normes $\|\cdot\|_A$ sur $A(F)$ et $\|\cdot\|_B$ sur $B(F)$, le tiré en arrière $\phi^\ast\|\cdot\|_B$ est dominé par $\|\cdot\|_A$. Il s'agit d'une équivalence de normes si le morphisme est fini.

Fixons pour la suite $\|\cdot\|_G$ une norme sur $G(F)$. On a une liste de propriétés formelles :
\begin{enumerate}
    \item $\|(\cdot)^{-1}\|_G$ est une norme sur $G(F)$, équivalente à $\|\cdot\|_G$ ;
    \item soit $m:G\times G\rightarrow G$ la multiplication, alors le tiré en arrière $m^\ast\|\cdot\|_G$ est dominé par la norme sur $G(F)\times G(F)$ ;
    \item $\|\cdot\|_G$ est bornée sur tout sous-groupe compact de $G(F)$ ;
    \item soit $P\in \F^G(M_0)$, on écrit $x=m_P(x)n_P(x)k(x)$ avec $m_P(x)\in M_P(F)$, $n_P(x)\in N_P(F)$, $k(x)\in K(F)$, une décomposition d'Iwasawa d'un élément général $x\in G(F)$, alors $\|m_P(\cdot)\|_G+\|n_P(\cdot)\|_G$ est une norme sur sur $G(F)$, équivalente à $\|\cdot\|_G$ ;
\end{enumerate}
partant des deux dernières propriétés il n'est pas difficile de prouver la majoration suivante : pour tous $M\in \L^G(M_0)$ et $Q\in \F^G(M)$, $\|x\|_{M\backslash G}=\min\{\|mx\|_G:m\in M(F)\}$ est une norme sur $(M\backslash G)(F)$ et il existe $c>0$ et $e>0$ tels que 
\[|v_M^Q(x)|< c(1+\log\|x\|_{M\backslash G})^e,\,\,\,\,\forall x\in (M\backslash G)(F)=M(F)\backslash G(F).\]

Munissons également $\g(F)$ d'une norme $\|\cdot\|_{\g}$ pour la suite, au sens des normes vectorielles et non des normes abstraites. Cela procure à tout sous-espace vectoriel de $\g(F)$ une norme vectorielle. On voit facilement, par définition, que 
\begin{enumerate}
    \item[(5)] pour tout $A>1$ la norme abstraite $A+\|\cdot\|_\g$ est dominée par $\|\cdot\|_G$ sur $G(F)$ ;
    \item[(6)] pour tout intervalle compact $I$ de $\R_{>0}$ et tout réel $A>1$ la norme abstraite $A+\|\cdot\|_\g$ est  équivalente à $\|\cdot\|_G$ sur $\{x\in G(F):|\nu(x)|\in I\}$.
\end{enumerate} 
Ces propriétés sont spécifiques aux groupes du type GL.

\subsection{Intégrale orbitale pondérée}

\begin{lemma}\label{YDLioplem:OInormCVarchi}
Soit $F$ un corps local archimédien. Pour tout $Y\in \g_{\ss}(F)$ il existe un entier $r_Y>0$ tel que
\[\int_{G_Y(F)\backslash G(F)}(1+\|g^{-1}Yg\|_{\g}^2)^{-r_Y}\,dg<+\infty.\]
\end{lemma}
\begin{proof}
Les normes sur un espace vectoriel topologique localement compact étant équivalentes, la convergence est indépendante de la norme choisie, et bien sûr indépendante des normalisations des mesures. Un résultat de Deligne-Rao affirme que toute intégrale orbitale sur $\g(F)$ est une distribution $G(F)$-invariante à support compact modulo la $G(F)$-conjugaison (\cite[théorème 2]{Rao72}). Puis un résultat de Harish-Chandra implique que toute telle distribution est tempérée (\cite[théorème I.9.11]{Vara77}). On conclut alors par \cite[lemme I.3.8]{Vara77}.
\end{proof}

\begin{theorem}\label{thm:IOPfullstatement}
Pour toute fonction $f\in \S(\g(F))$, l'intégrale \eqref{def:NIOPl} converge absolument. De surcroît, lorsque $F$ est archimédien, l'espace de Schwartz-Bruhat $\S(\g(F))$ est muni de semi-normes $\|-\|_{a,b}=\sup_{|\alpha|\leq a,|\beta|\leq b} \|x^\alpha \frac{\partial}{\partial x^\beta}-\|_{L^\infty(\g(F))}$, alors il existe des certaines constantes $c_{\o}>0$ et $a_{\o}\in \N$  tels que $|J_L^Q(\o,-)|$ soit majoré par $c_{\o}\|-\|_{a_\o,0}$. Une intégrale orbitale pondérée locale est donc une distribution tempérée. 
\end{theorem}

\begin{proof}
Soit $X$ le représentant standard de $\Ind_L^G(\o)(F)$. Soit $H=\text{envL}(G_{X_\ss};G)$, on a $X\in\mathfrak{h}$. Prenons $P\in \P^G(H)$ quelconque. Nous avons
\begin{align*}
\int_{G_{X}(F)\backslash G(F)} &\left|f((\Ad g^{-1})X)v_{{}^wL,X}^{{}^wQ}(g)\right|\,dg\\
&=\gamma^G(P)\int_{G_X(F)\backslash H(F)}\int_{N_P(F)}\int_K \left|f(k^{-1}n^{-1}h^{-1}Xhnk)v_{{}^wL,X}^{{}^wQ}(hn)\right|\,dk\,dn\,dh  \\
&=\gamma^G(P)\cdot\text{Jac}\cdot\int_{G_X(F)\backslash H(F)}\int_{\n_P(F)}\int_K \left|f\left(k^{-1}(h^{-1}Xh+U)k\right)v_{{}^wL,X}^{{}^wQ}(hn(h,U))\right|\,dk\,dn\,dh\\
&\leq \gamma^G(P)\cdot\text{Jac}\cdot\sum_{(L_1,L_2)\in\L^{L_Q}(L)^2}d_L^{L_Q}(L_1,L_2)\int_{G_X(F)\backslash H(F)}\int_{\n_P(F)}\int_K \\
&\hspace{2.5cm}\left|f\left(k^{-1}(h^{-1}Xh+U)k\right)v_{{}^wL,X}^{{}^wQ_1}(h)v_{{}^wL}^{{}^wQ_2}(k(w_{{}^wQ_2,X}h)n(h,U))\right|\,dk\,dn\,dh,
\end{align*}
où $\text{Jac}=|\det(\ad(h^{-1}Xh);\n_P(F))|^{-1}$, $n(h,U)\in N_P(F)$ est défini sur un ouvert dense de $G_X(F)\backslash H(F)\times \n_P(F)$ par $(\Ad n(h,U)^{-1})(h^{-1}Xh)=(h^{-1}Xh)+U$, et $k(w_{{}^wQ_2,X}h)$ est un élément de $K$ tel que  $(w_{{}^wQ_2,X}h)k(w_{{}^wQ_2,X}h)^{-1}\in P(F)$. La dernière inégalité vient de la formule \eqref{eq:GMfamilleproductformula}. Constatons que l'élément $V(h,U)=n(h,U)-\Id\in \n_P(F)$  dépend polynômialement de $h^{-1}Xh$ et $U$ (ici $\det(\ad(h^{-1}Xh);\n_P(F))$ est une constante). On en déduit donc l'existence de $c>0$, $e\in \N_{>0}$, $B$ un polynôme à une variable et à coefficients positifs, tels que $B(0)>3$ et $\left|v_{{}^wL}^{{}^wQ_2}(k(w_{{}^wQ_2,X}h)n(h,U))\right|<c\left(\log B(\|h^{-1}Xh\|_{\g}) +\log B(\|U\|_{\g})\right)^e$. Par une réduction immédiate, on est conduit au problème de convergence de 
\begin{align*}
\int_{H_X(F)\backslash H(F)}\left|f\left(h^{-1}Xh\right)v_{{}^wL,X}^{{}^wQ_1}(h)\right|(\log B(\|h^{-1}Xh\|_{\mathfrak{h}}))^e\,dh    
\end{align*}
Les calculs dans $G$ sont réduits presque à ceux dans $H$. Il ne reste qu'à dépaqueter le poids $v_{{}^wL,X}^{{}^wQ_1}(h)$, qui est avant tout défini sur $G$. Abrégeons $M=\MR^{X}$ le facteur de Levi d'un élément de $\cR^G(X)$ semi-standard. Prenons un élément de $\P^G(M)$, noté $P_3$. Soit $g \in G(F)$. La $(G,M)$-famille définie pour $P \in\P^G(M)$ par $v_{P,\text{mod}}(\lambda,g)\eqdef\exp(\langle\lambda,R_{P_3}(g)-R_P(g)\rangle)$ donne les mêmes poids que la $(G,M)$-famille $(v_{P,X}(\lambda,g))_P$. On a alors
\[v_{{}^wL,X}^{{}^wQ_1}(g)=\frac{1}{m!}\sum_{P\in \P^{{}^wQ_1}({}^wL)}(\langle\lambda,R_{P_3}(g)-R_P(g)\rangle)^m\theta_P^Q(\lambda)^{-1},\]
où $\lambda\in a_L^\ast$ est assez général et $m = \dim(a_L^{Q_1})$. On en déduit la majoration suivante : il existe $c> 0$ tel pour tout $g \in G(F)$ on a 
\[|v_{{}^wL,X}^{{}^wQ_1}(g)|\leq c\cdot \sum_{(P_1,P_2)\in \P^G(M)^{\text{adj}}}\|R_{P_1}(g)-R_{P_2}(g)\|^m\]
où $\P^G(M)^{\text{adj}} \subseteq \P^G(M)^2$ est l'ensemble formé de couples de paraboliques $(P_1, P_2)$ qui sont adjacents. La norme $\|-\|$ est la norme euclidienne sur $a_M$. L'élément $X_\ss$ étant standard pour $H$, on voit que $M$ est un sous-groupe de Levi semi-standard de $H$, et l'application 
\begin{align*}
P\in \cR^G(X)\longmapsto P\cap H \in\cR^H(X)  
\end{align*}
est surjective. On en déduit de ce fait $R_{P}^G(h)=R_{P\cap H}^H(h)$ pour tout $h\in H(F)$, ici les
exposants ont pour objectif de préciser les groupes dans lesquels chaque application est réalisée. On est confronté au problème de convergence de 
\begin{equation}\label{eq:CVintermediateWOPell}
\int_{H_X(F)\backslash H(F)}\left|f\left(h^{-1}Xh\right)\right| \cdot \|R_{P_1}^H(h)-R_{P_2}^H(h)\|^m\cdot (\log B(\|h^{-1}Xh\|_{\mathfrak{h}}))^e\,dh    
\end{equation}
pour $P_1$ et $P_2$ deux sous-groupes paraboliques adjacents dans $H$. Quitte à œuvrer dans chaque composante irréductible de $H$ et changer le pôlynome $B$, on peut suppose que $H$ est le groupe des automorphismes d'un $D$-module à droite libre avec $D$ une algèbre à division sur $F$. Afin de garder la cohérence avec les sous-sections précédentes, on va ultérieurement réécrire $H$ en $G$ et $h\in H(F)$ en $g\in G(F)$, et supposer que $X$ est l'élément défini par \eqref{elementXell}.

On reprend l'équation \eqref{eq:CVintermediateWOPell}. Soit $(P_1, P_2)\in \P^G(M)^{\text{adj}}$. On reprend sans plus de commentaire les notations des sous-sections précédentes. On peut clairement supposer que $r>1$, sinon la convergence de \eqref{def:NIOPl} vient directement du lemme précédent. Pour tout $k\in K$ et $U \in (\Ad M_{P^{-}})X_\ss \oplus (\n_{\widetilde{P_1}}\cap \n_{\widetilde{P_2}})$ tel que $kUk^{-1}=g^{-1}Xg$, on a
\[\|R_{P_1}(g)-R_{P_2}(g)\|=\frac{1}{\deg D}|\log|\Nrd(U_{1,3})||\|\alpha^\vee\|.\]
D'après le lemme \ref{lem:centralizerOIformula}, on a, pour une certaine constante $c > 0$ qu'il est inutile d'expliciter ici,
\begin{equation}\label{YDLiopeq:IOPCVproofdecompquotcent}
\begin{split}
&\int_{G_X(F)\backslash G(F)}|f(g^{-1}Xg)|\cdot\|R_{P_1}(g)-R_{P_2}(g)\|^m\cdot(\log B(\|g^{-1}Xg\|_{\g}))^e\,dg\\
&\hspace{0.5cm}=c\cdot\int_{(M_{P^{-}})_{X_\ss}(F)\backslash M_{P^{-}}(F)}\int_{I_{1,3}(F)}\int_{\n_{\widetilde{Q}}(F)}\int_K |f(k^{-1}(m_{P^{-}}^{-1}X_\ss m_{P^{-}}+y+V)k)|\\
&\hspace{1cm}\cdot|\log|\Nrd(y)||^m|\Nrd(y)|^{\underline{n}_p(r_1+r_2)\deg D}\cdot(\log B(\|k^{-1}(m_{P^{-}}^{-1}X_\ss m_{P^{-}}+y+V)k\|_{\g}))^e\,dk\,dV\,dy\,dm_{P^{-}}    
\end{split}
\end{equation}
\`{A} peine de changer la norme vectorielle $\|\cdot\|_\g$ on peux supposer que $K$ agit par isométrie à gauche et à droite, puis $\m_{P^{-}}(F)\oplus^{\perp} (\Hom_D(W_3,W_1)(F)\oplus \n_{\widetilde{Q}}(F))$ est orthogonale. 

Supposons pour le moment que $F$ est archimédien. Par le lemme \ref{YDLioplem:OInormCVarchi} et une réduction immédiate, il suffit d'obtenir la convergence de
\[\int_{\GL_{\underline{n}_pr_1,D}(F)}|f(y)||\log|\Nrd(y)||^m|\Nrd(y)|^{\underline{n}_p(r_1+r_2)\deg D}\log(B(\|y\|_{\gl_{\underline{n}_pr_1,D}}))^e\,dy\]
pour tout $f\in \S(\gl_{\underline{n}_pr_1,D}(F))$. 

En utilisant la décomposition d'Iwasawa, et le fait que $N_P=\Id\oplus \n_P$ dans l'algèbre de Lie de $P
$, on voit qu'il suffit de démontrer que l'intégrale suivante converge 
\begin{align*}
&\int_{T_{\underline{n}_pr_1,D}(F)}|f(t)||\log|\Nrd(t)||^m|\Nrd(t)|^{\underline{n}_p(r_1+r_2)\deg D}\\
&\hspace{0.5cm}|\Nrd(t_1)|^{(-\underline{n}_pr_1+1)\deg D}|\Nrd(t_2)|^{(-\underline{n}_pr_1+2)\deg D}\cdots |\Nrd(t_{\underline{n}_pr_1-1})|^{\deg D}\log(B(\|t\|_{\mathfrak{t}_{\underline{n}_pr_1,D}}))^e\,dt    
\end{align*}
où $T_{\underline{n}_pr_1,D}$ est le sous-groupe diagonal de $\GL_{\underline{n}_pr_1,D}$ et $f \in \S(\mathfrak{t}_{\underline{n}_pr_1,D}(F))$. Quitte à majorer $f$ par un produit de fonctions des coordonnées $t_i$, on est ramené à étudier l'intégrale suivante en dimension 1 (sur $D$) :
\begin{equation}\label{eq:dim1intetude}
 \int_{D\setminus\{0\}\subset D}|f(t)||\log|\Nrd(t)||^m|\Nrd(t)|^{k\deg D}\log(B(\|t\|_{\gl_{1,D}}))^e\,dt
 \end{equation}
avec $m\geq 0$ et $\underline{n}_p(r_1+r_2)-1\geq k\geq \underline{n}_pr_2\geq 0$, on a égalment plongé au passage $D^\times$ dans son algèbre de Lie. Identifions $D$ en un espace vectoriel sur $F$ via une base. Remarquons que la norme réduite est une application polynomiale homogène dans cette base. On découpe l'intégrale \eqref{eq:dim1intetude} en deux parties : une boule centrée en 0 de rayon (pour la norme vectorielle) assez petit et son complémentaire dans $D$. L'intégrale sur cette deuxième partie converge comme $f$ est de classe Schwartz-Bruhat. Enfin la convergence de l'intégrale sur la première partie résulte facilement de la convergence absolue de la dérivée d'ordre $m$ sous l'intégrale de la fonction zêta associée à la fonction $f$ évaluée en $(k+1)\deg D$ pour l'algèbre à division $D$, notons que l'on a implicitement majoré $\log(B(\|t\|_{\gl_{1,D}}))^e$ par une constante. Sinon une autre façon plus élaborée de faire est d'invoquer la résolution des singularités de Hironaka : on majore $|f(t)|\log(B(\|t\|_{\gl_{1,D}}))^e$ toujours par une constante puis remplace $\{0\}$ par un diviseur à croisements normaux, $\Nrd$ par une fonction monomiale, avec un facteur Jacobien monomial.

On voit en outre que $J_L^Q(\o,-)$ est une distribution tempérée par les majorations obtenues au fil des calculs, satisfaisant l'estimation voulue.

Considérons ensuite le cas où $F$ est non-archimédien. Dans le raisonnement ci-dessus, nous remplaçons l'endroit où nous utilisons le lemme \ref{YDLioplem:OInormCVarchi} par un lemme de compacité de Harish-Chandra, qui affirme que nous pouvons restreindre le domaine d'intégration de $(M_{P^-})_{X_\ss}(F)\backslash M_{P^-}(F)$ dans l'équation \eqref{YDLiopeq:IOPCVproofdecompquotcent} à une partie compacte de $M_{P^-}(F)$ (cf. \cite[lemme 14.1]{Kottbook}). Le reste de la démonstration fonctionne de la même manière.
\end{proof}

\section{Comparaison des différentes définitions d'une intégrale orbitale pondérée}\label{sec:4comparedef}

L'objectif principal de cette section est de présenter d'autres définitions d'une intégrale orbitale pondérée (théorèmes \ref{YDLiopthm:ArtdefdirectIOP}, \ref{YDLiopthm:ArtdefIOPCVCinftyc}), qui s’inscrivent dans le courant d’idées d'Arthur, et les comparer (théorèmes \ref{thm:IOPcomparaisonII-III}, \ref{thm:IOPcomparaisonI-II}).

\subsection{Remarque sur la définition d'Arthur}

Avant d'aller plus loin, nous aimerions adresser une remarque à l'attention des initiés du domaine. On définira comme Arthur, pour tous $M\in \L^G(M_0)$, $\alpha\in \Sigma(\g;A_M)$ et $\o$ une $M$-orbite définie sur $F$ dans $\m$, un nombre $\rho(\alpha, \o)$. Puis on définira l'intégrale orbitale pondérée grâce à des $(G,M)$-familles associées aux nombres $\rho(\alpha, \o)$ pour $\alpha\in \Sigma(\g;A_M)$. Toutefois, comme souligné  par Moeglin-Waldspurger (\cite[section II.1.4]{MW16}), la définition obtenue, pour $G$ un groupe réductif général, ne satisfait pas la formule de descente de l'induction (point 2 de la proposition \ref{prop:propertiesIOP}), les nombres d'Arthur $\rho(\alpha, \o)$ doivent être modifiés à cause de la présence des racines divisibles. Cela dit, quand $G$ est un groupe du type GL, toute racine est non-divisible, nous pouvons donc adopter sans risque la définition d'Arthur.

\subsection{Rappel de la géométrie algébrique}

Soit $S$ un schéma. On note $\O_S$ le faisceau des fonctions régulières sur $S$. Soit $x\in S$. On note $\O_{S,x}$ la fibre de $\O_S$ en $x$. Un $S$-schéma intègre de type fini et séparé sur $S$ est dite une $S$-variété. 
Soit $U$ un sous-ensemble de $S$. On note $\overline{U}$ sa clôture de Zariski.

Soit $X$ et $Y$ des $S$-schémas. On note $X\times_S Y$ le  produit fibré évident. Si $S=\text{Spec}(A)$ et $Y=\text{Spec}(B)$ pour $A$ et $B$ deux anneaux, on note aussi $X_B\eqdef X\times_S Y$. Soit $f:X\to Y$ un morphisme de schémas. Soit $T$ un sous-ensemble de $Y$. On note $f^{-1}(T)$ le produit fibré ensembliste $X\times_Y T$. Si $T$ est un schéma, on munit $f^{-1}(T)$ de sa structure naturelle de schéma. S'il existe un ouvert dense $U$ (resp. $V$) de $X$ (resp. $Y$) tel qu'il existe un isomorphisme de schémas $g:U\to V$, on écrit $g:X\dashrightarrow Y$. 


\begin{lemma}[{{\cite[corollaire (2.3.12)]{EGAIV1}}}]\label{YDLioplem:Zarclosurecommutefieldbc} 
Soit $X$ et $Y$ deux schémas. Soit  $f:X\to Y$ un morphisme quasi-compact, surjectif et plat. Alors un sous-ensemble $T\subseteq Y$ est ouvert (resp. fermé) si et seulement si $f^{-1}(T)$ est ouvert (resp. fermé). 
\end{lemma}



\begin{lemma}[{{\cite[lemme (8.10.12.1)]{EGAIV3}}}]\label{YDLioplem:extendbirmap}
Soient $X$ un schéma intègre, $Y$ un schéma intègre et normal. Alors un morphisme $f:X\to Y$ fini et birationnel est un isomorphisme.
\end{lemma}

\subsection{Interprétation des poids non-tordus en termes de représentations}
Fixons $M\in \L^G(M_0)$. On a un réseau $X^\ast(A_M)$ dans $a_M^\ast$. \'{E}crivons $\Wt(a_M^\ast)$ l'ensemble des poids extrémaux des représentations $F$-rationnelles (de dimension finie) de $G$. Puisque $\Wt(a_M^\ast)$ est d'indice fini dans $X^\ast(A_M)$, c'est aussi un réseau de $a_M^\ast$. Pour tout $\omega\in \Wt(a_M^\ast)$ on fixe $(\Lambda_\omega,V_\omega,\phi_\omega,\|\cdot\|)$ avec $\Lambda_\omega:G\rightarrow \GL(V_\omega)$ une representation irréductible, $\phi_\omega\in V_\omega(F)$ un vecteur extrémal de poids $\omega$, $\|\cdot\|=\|\cdot\|_{V_\omega}$ une norme de $V_\omega(F)$ au sens de la sous-section \ref{subsec:norme}, invariante par $K$ et pour laquelle $\phi_\omega$ est de longueur 1. Soit $P\in \P^G(M)$ tel que $\omega$ est $P$-dominant. Alors $v_P(\omega,x)=\|\Lambda_\omega(x^{-1})\phi_\omega\|$ pour tout $x\in G(F)$. 

\subsection{Définition directe}\label{subsec:defdirect}

Dans \cite{Art88}, Arthur définit une intégrale orbitale pondérée à l'aide des certaines $(G,M)$-familles $w_P$ et $r_P$. Il commence par définir ces objets pour les orbites nilpotentes, puis il les généralise au cas général en utilisant une descente au centralisateur semi-simple. Cependant, il semble que la définition d'une intégrale orbitale pondérée via la descente au centralisateur semi-simple ne soit pas favorable pour étudier la convergence lorsque les fonctions tests sont de classe Schwartz-Bruhat (voir théorème \ref{YDLiopthm:ArtdefIOPCVCinftyc}). Dans ce numéro, notre objectif est donc de généraliser l'approche d'Arthur aux orbites quelconques,  sans recourir à la descente au centralisateur semi-simple.

Fixons pour la suite $P_\square\in \P^G(M)$ un « point-base ». Soit $\o$ une $M$-orbite dans $\m$ contenant un $F$-point, elle admet une unique décomposition en $\o=A_{\o}+\o^-$ avec $A_{\o}\in \a_M(F)$ et  $\o^-$ une $M$-orbite définie sur $F$ dans $\m/\a_M$, ici par abus de notation $\m/\a_M$ désigne le supplémentaire orthogonal de $\a_M$ dans $\m$, pour la forme bilinéaire canonique. 

Pour rappel on a défini $\a_{M,\o,G-\reg}\eqdef \{A\in\a_M\mid A+Y\in \Ind_M^G(A+Y),\forall Y\in \o\}$ (équation \eqref{YDLiopeq:defaMoG-reg}). C'est une sous-variété ouverte dense définie sur $F$ de $\a_M$.

Pour tout $A \in \a_{M,\o,G-\reg}$, tout $Y \in \o$ et tout $V \in \n_\square$ ($N_\square$ est le radical unipotent de $P_\square$), il existe un unique élément
\[n=n_\square(A,Y,V)\in N_\square\]
défini par la condition
\[n^{-1}(A+Y)n=A+Y+V\]
(l'équivalence (1) $\Leftrightarrow$ (7) de la proposition \ref{YDLiopprop:whenXinIndMGX}).

Soit $\alpha\in \Sigma(\g;A_M)$. Soit $M_\alpha\in \L^G(M)$ tel que $\Sigma(\m_\alpha;A_M)=\{\alpha,-\alpha\}$. On suppose que $P_\square \cap M_\alpha$ est le sous-groupe parabolique de $M_\alpha$ tel que $\Sigma(\mathfrak{p}_\square\cap \m_\alpha;A_M)=\{-\alpha\}$. Soit $P\in \P^G(M)$ tel que $P\cap M_\alpha$ soit opposé à $P_\square \cap M_\alpha$. Pour tout $V \in \n_{\square}(F) \cap \m_\alpha(F)$, on a
$n_\square(A,Y,V )\in N_\square(F)\cap M_\alpha(F)$ pour tous $A\in \a_{M,\o,G-\reg}(F)$ et $Y \in \o(F)$. En particulier, $H_P(n_\square(A,Y,V))$ dépend uniquement de $P\cap M_\alpha$ et non du choix de $P$ dans $\P^G(M)$. D'ailleurs, si $\pi_{n_\square,n_\square\cap \m_\alpha}: n_\square\to n_\square\cap \m_\alpha$ est la projection évidente alors 
$H_P(n_\square(A,Y,V))=H_P(n_\square(A,Y,\pi_{n_\square,n_\square\cap \m_\alpha}(V)))$ pour tout $V\in \n_{\square}(F)$.

Par un développement en série de Laurent on voit qu'il existe un unique nombre réel positif $\rho(\alpha,\o)$ tel que si l'on pose 
\[r_\alpha(\lambda,A,\o)=|\alpha(A)|^{\rho(\alpha,\o)\langle \lambda,\alpha^\vee\rangle},\,\,\,\,\lambda\in ia_M^\ast,\]
alors la limite
\begin{equation}\label{eq:defr_alphalimit}
\lim_{A\in \a_{M,\o,G-\reg}(F)\to 0}r_\alpha(\lambda,A,\o)\exp(-\langle\lambda,H_P(n_\square(A,Y,V))\rangle)    
\end{equation}
existe et définit une fonction non identiquement nulle des variables $Y\in\o(F)$ et $V\in \n_{\square}(F)$. Cette propriété caractérise $\rho(\alpha,\o)$. On a aussi $\rho(\alpha,\o)=\rho(-\alpha,\o)$ et $
r_\alpha(\lambda,A,\o)=r_{-\alpha}(-\lambda,A,\o)$. 


On introduit la famille suivante, qui est une $(G,M)$-famille
\begin{align*}
w_P(\lambda,A,Y,V)= \left(\prod_{\alpha\in \Sigma(\mathfrak{p};A_M)\cap (-\Sigma(\mathfrak{p}_\square;A_M))}r_\alpha(\lambda,A,\o)\right) \exp(-\langle\lambda,H_P(n_\square(A,Y,V))\rangle) 
\end{align*}
pour $P\in \P^G(M)$, $A\in \a_{M,\o,G-\reg}(F)$, $Y\in\o(F)$ et $V\in \n_{\square}(F)$. 
Cette famille dépend du sous-groupe parabolique $P_\square$ mais, pour ne pas alourdir encore les notations, on ne le fait pas figurer dans la notation. Lorsqu'il est indispensable d'indiquer la dépendance à l'égard du point-base $P_\square$, on va noter $w_P(\lambda,A,Y,V)=w_{P|P_\square}(\lambda,A,Y,V)$.

Soit $P\in \P^G(M)$. Posons $\widetilde{\g}_P$ la $F$-variété $\{(X,Pg)\in \g\times (P\backslash G)\mid X\in (\Ad g^{-1})\p\}$. C'est la résolution partielle de Grothendieck-Springer associée à $P$. Posons $\widetilde{\g}_P(\o)\eqdef\{(X,Pg)\in \widetilde{\g}_P\mid  (\Ad g)X\in \a_M+\overline{\o}\oplus\n_P=\a_M\oplus \overline{\o^-}\oplus\n_P\}$, pour rappel $\overline{\o}$ (resp. $\overline{\o^-}$) est la clôture de Zariski de $\o$ (resp. $\o^-$) dans $\m$. 

\begin{lemma}\label{lem:geomnormalschemeGL}
$\widetilde{\g}_P(\o)$ est un fermé de $\widetilde{\g}_P$. On le munit de la structure de schéma en tant qu'un sous-schéma  fermé de $\widetilde{\g}_P$. Alors, $\widetilde{\g}_P(\o)$ est normal.  
\end{lemma}
\begin{proof}
Notons $N_P^-$ le radical unipotent du sous-groupe parabolique opposé à $P$ par rapport à $M$ (ce groupe a été noté $\overline{N_P}$, mais ici on veut réserver la notation $\overline{\cdots}$ pour la clôture de Zariski). On connaît un système de coordonnées affines pour $\widetilde{\g}_P$ : soit $w\in \uW^{G,0}$, le morphisme
\begin{equation}\label{YDLiopeq:affinecoordgw}
g_w:(X,n)\in\p\times N_P^-\mapsto ((\Ad (nw)^{-1})X,Pnw)\in \widetilde{\g}_P    
\end{equation}
est une immersion ouverte, et $\cup_{w\in\uW^{G,0}} \text{Im}(g_w)=\widetilde{\g}_P$.

On a $g_w^{-1}(\widetilde{\g}_P(\o))=(\a_M\oplus \overline{\o^-}\oplus\n_P)\times N_P^-$. C'est donc une sous-variété fermé de $\p\times N_P^-$. On en déduit que  $\widetilde{\g}_P(\o)$ est une sous-variété fermée de $\widetilde{\g}_P$.

Pour prouver que $\widetilde{\g}_P(\o)$ est un schéma normal, il suffit de prouver que $\overline{\o^-}$ est un schéma normal. Soit $\overline{F}$ une clôture algébrique de $F$. Grâce aux travaux de Kraft-Procesi (\cite[théorème en page 227]{KraPro79}) 
pour les formes des groupes généraux linéaires, on sait que $\overline{\o^-_{\overline{F}}}$ est normal. Or $\overline{\o^-_{\overline{F}}}=(\overline{\o^-})_{\overline{F}}$ selon le lemme \ref{YDLioplem:Zarclosurecommutefieldbc}. Le schéma $\overline{\o^-}$ est alors géométriquement normal, donc normal.
\end{proof}

Il y a un morphisme de schémas 
\begin{align*}
\phi_P:\widetilde{\g
}_P(\o)&\rightarrow \g\times \a_M    \\
(X,Pg)&\mapsto (X,A),
\end{align*}
avec $A$ la soustraction de la projection sur $\a_M$ de $(\Ad g)X\in\a_M\oplus \overline{\o^-}\oplus\n_P$ par $A_\o$.  
\begin{lemma}
Pour tout $(X,A)\in \g\times \a_{M,\o,G-\reg}$, la fibre $\phi_P^{-1}(X,A)$ est ou bien vide si $X$ n'est pas dans $(\Ad G)(A\oplus\overline{\o})$, ou bien un schéma ayant un seul point sinon.    
\end{lemma} 
\begin{proof}
Par définition, $\phi_P^{-1}(X,A)$ est non-vide si et seulement si \begin{align*}
X&\in (\Ad G)(A+\overline{\o}\oplus \n_P)=(\Ad G)\left(\bigcup_{\substack{\text{$\o'$ : $M$-orbite dans $\m$}\\ \o'\subseteq \overline{A+\o}}} \o'\oplus\n_P\right) \\
&=\bigcup_{\substack{\text{$\o'$ : $M$-orbite dans $\m$}\\ \o'\subseteq A+\overline{\o}}}\overline{\Ind_M^G(\o')}=\bigcup_{\o'} \overline{(\Ad G)\o'}=(\Ad G)(A+\overline{\o}).
\end{align*} 
L'avant dernière égalité vient de la définition de $\a_{M,\o,G-\reg}$ (équation \eqref{YDLiopeq:defaMoG-reg}). Supposons maintenant que $X\in (\Ad G)(A+\overline{\o})$. Soient $g_1,g_2\in G$ tels que $\phi_P(X,Pg_1)=\phi_P(X,Pg_2)=(X,A)$. Pour $i=1,2$, on a $(\Ad g_i)X\in A+\overline{\o}\oplus\n_P$. Au vu de la définition de $\a_{M,\o,G-\reg}$ on en déduit $\Ind_{M}^G((\Ad g_i)X)=(\Ad G)(\Ad g_i)X=(\Ad G)X$. Donc $(\Ad g_i)X\in (\Ad G)X\cap \p=\Ind_{M}^G((\Ad g_i)X)\cap \p$, qui est une $P$-orbite dans $\p$ selon le point 5 de la proposition \ref{chap3prop:indprop}. De ce fait il existe $p\in P$ tel que $(\Ad g_1)X=(\Ad pg_2)X$. Autrement dit $pg_2g_1^{-1}\in G_{(\Ad g_1)X}\subseteq P$, selon le point 6 de la proposition \ref{chap3prop:indprop}. Il vient $Pg_1=Pg_2$, ce qu'il fallait.
\end{proof}

Posons $\widetilde{\g}_P(\o)_{G-\reg}$ le produit fibré $(\g\times\a_{M,\o,G-\reg})\times_{\g\times\a_M}\widetilde{\g}_P(\o)$, considéré comme un sous-schéma ouvert de $\widetilde{\g}_P(\o)$.
Pour tout $P,Q\in \P^G(M)$ il existe donc un unique isomorphisme 
\begin{equation}\label{eq:G-Sresolutionindclass}
\begin{split}
 f_{Q,P}:\widetilde{\g}_P(\o)_{G-\reg}&\longrightarrow \widetilde{\g}_Q(\o)_{G-\reg}    \\
(X,Pg_1)&\longmapsto (X,Qg_2)    
\end{split}    
\end{equation}
tel que $\phi_Q|_{\widetilde{\g}_Q(\o)_{G-\reg}}\circ f_{Q,P}=\phi_P|_{\widetilde{\g}_P(\o)_{G-\reg}}$, obtenu en composant la flèche de haut avec l'inverse de  la flèche de gauche du produit fibré 
\[\begin{tikzcd}	{\widetilde{\g}_P(\o)_{G-\reg}\times_{\g\times\a_{M,\o,G-\reg}}\widetilde{\g}_Q(\o)_{G-\reg}} & {\widetilde{\g}_Q(\o)_{G-\reg}} \\
	{\widetilde{\g}_P(\o)_{G-\reg}} & {\g\times\a_{M,\o,G-\reg}}
	\arrow[from=1-1, to=1-2]
	\arrow[from=1-2, to=2-2]
	\arrow[from=2-1, to=2-2]
	\arrow[from=1-1, to=2-1]
\end{tikzcd}\]
L'élément $g_2\in Q\backslash G$ est défini par la relation $(\Ad g_2^{-1})(A\oplus\overline{\o}\oplus \n_Q)=(\Ad g_1^{-1})(A\oplus\overline{\o}\oplus \n_P)$.

Posons $\widetilde{\g}_{P,\Ind}(\o)$ la sous-partie $\{(X,Pg)\in \widetilde{\g}_P\mid  (\Ad g)X\in \Ind_{M}^G(\o)\cap(\o\oplus\n_P)=(\o\oplus\n_P)_{G-\reg}\}$ de $\widetilde{\g}_P(\o)$. Posons $\widetilde{\g}_{P,\Ind}(\o)$ la sous-partie $\{(X,Pg)\in \widetilde{\g}_P\mid  (\Ad g)X\in \Ind_{M}^G(\o)\cap(\o\oplus\n_P)=(\o\oplus\n_P)_{G-\reg}\}$ de $\widetilde{\g}_P(\o)$. Posons $\widetilde{\g}_P(\o)^{\sm}\eqdef\widetilde{\g}_{P,\Ind}(\o)\cup \widetilde{\g}_{P}(\o)_{G-\reg}$.

\begin{lemma}
$\widetilde{\g}_P(\o)^{\sm}$ est un ouvert de $\widetilde{\g}_P(\o)$.    
\end{lemma}
\begin{proof}
Cela se justifie facilement avec le système de coordonnées affines $(g_w)_{w\in \uW^{G,0}}$ (équation \eqref{YDLiopeq:affinecoordgw}).   
\end{proof}

On munit $\widetilde{\g}_P(\o)^{\sm}$ de la structure de schéma en tant qu'un sous-schéma ouvert de $\widetilde{\g}_P(\o)$.    

L'isomorphisme $f_{Q,P}$ définit un morphisme birationnel $f_{Q,P}:\widetilde{\g}_P(\o)^\sm\dashrightarrow\widetilde{\g
}_Q(\o)^\sm$.

\begin{lemma}\label{YDLioplem:prolongerbir}Le morphisme birationnel $f_{Q,P}:\widetilde{\g}_P(\o)^\sm\dashrightarrow\widetilde{\g}_Q(\o)^\sm$ est en fait un isomorphisme $\widetilde{\g}_P(\o)^\sm\to\widetilde{\g}_Q(\o)^\sm$, dont la restriction donne un isomorphisme $\widetilde{\g}_{P,\Ind}(\o)\to \widetilde{\g}_{Q,\Ind}(\o)$.
\end{lemma}

\begin{proof}
On a un diagramme commutatif
\setlength{\perspective}{2pt}
\[\begin{tikzcd}[row sep={40,between origins}, column sep={80,between origins}]
      &[-\perspective] {\widetilde{\g}_P(\o)_{G-\reg}\times_{\g\times\a_{M,\o,G-\reg}}\widetilde{\g}_Q(\o)_{G-\reg}} \arrow{rr}\arrow{dd} \arrow{dl} &[\perspective] &[-\perspective] {\widetilde{\g}_Q(\o)_{G-\reg}} \arrow[dd,]\arrow[dl,"\phi_Q|_{\widetilde{\g}_Q(\o)_{G-\reg}}"] \\[-\perspective]\widetilde{\g}_P(\o)_{G-\reg}
    \arrow[crossing over,"\phi_P|_{\widetilde{\g}_P(\o)_{G-\reg}}" near end]{rr} \arrow[dd] \arrow[urrr,crossing over,"f_{Q,P}"] & & \g\times \a_{M,\o,G-\reg} \\[\perspective]
      & \widetilde{\g}_P(\o)\times_{\g\times \a_{M}}
\widetilde{\g}_Q(\o) \arrow{rr} \arrow[dl] & &  \widetilde{\g}_Q(\o) \arrow[dl,"\phi_Q"] \\[-\perspective]
    \widetilde{\g}_P(\o) \arrow[rr,"\phi_P"]\arrow[urrr,crossing over,dotted]   && \g\times \a_{M} \arrow[from=uu,crossing over]
\end{tikzcd}\]
dont les flèches sont celles évidentes. Les six faces du cube sont cartésiennes.

Le morphisme $\widetilde{\g}_P\to \g$ est projectif, et $\g\times \a_M$ est un $F$-schéma séparé, on en tire que $\widetilde{\g}_P\to\g\times \a_M$ est projectif. Comme $\widetilde{\g}_P(\o)\to \widetilde{\g}_P$ est une immersion fermée. Il vient que $\phi_P$ est projectif. De même $\phi_Q$ est projectif. 

Soit $Z\eqdef\overline{\widetilde{\g}_P(\o)_{G-\reg}\times_{\g\times\a_{M,\o,G-\reg}}\widetilde{\g}_Q(\o)_{G-\reg}}$ la clôture de Zariski de $\widetilde{\g}_P(\o)_{G-\reg}\times_{\g\times\a_{M,\o,G-\reg}}\widetilde{\g}_Q(\o)_{G-\reg}$ dans $\widetilde{\g}_P(\o)_{\times_{\g\times\a_{M}}}\widetilde{\g}_Q(\o)$. C'est une sous-$F$-variété fermée. Le morphisme naturel 
\[h_P:Z\to \widetilde{\g}_P(\o)\]
est la composée $Z\to \widetilde{\g}_P(\o)_{\times_{\g\times\a_{M}}}\widetilde{\g}_Q(\o)\to \widetilde{\g}_P(\o)$. La première flèche étant une immersion fermée, elle est projective. La deuxième flèche étant un changement de base de $\phi_Q$, elle est également projective. Ainsi $h_P$ est projectif. Son image est donc un fermé de $\widetilde{\g}_P(\o)$. Or son image contient $\widetilde{\g}_P(\o)_{G-\reg}$ qui est un ouvert dans $\widetilde{\g}_P(\o)$. Cette image vaut donc $\widetilde{\g}_P(\o)$.

Soit $C_P\eqdef h_P\left(h_P^{-1}(\widetilde{\g}_P(\o)^\sm)\cap h_Q^{-1}
(\widetilde{\g}_Q(\o)^\sm)\right)\subseteq \widetilde{\g}_P(\o)^\sm$. Soit $h_P'$ le changement de base de $h_P$ à $C_P$ :
\[h_P':h_P^{-1}(\widetilde{\g}_P(\o)^\sm)\cap h_Q^{-1}(\widetilde{\g}_Q(\o)^\sm)\to C_P.\]
Soit $(X,Qg)\in C_Q\subseteq \widetilde{\g}_{P,\Ind}(\o)\cup \widetilde{\g}_{P}(\o)_{G-\reg}$. 

Supposons d'abord que $(X,Qg)\in \widetilde{\g}_Q(\o)_{G-\reg}$. Autrement dit $(\Ad g)X\in A_{M,\o,G-\reg}\oplus\overline{\o^-}\oplus \n_Q$.
Soit $(Y,Pg_1)\in \widetilde{\g}_P(\o)$. On sait que $(Y,Pg_1)\in h_P'(h_Q'^{-1}(((X,Qg)))$ si et seulement si $Y=X$ et $(\Ad g_1)Y\in A_{M,\o,G-\reg}\oplus\overline{\o^-}\oplus \n_P$. D'après la définition de $a_{M,\o,G-\reg}$ et l'équivalence (1) $\Leftrightarrow$ (3) de la proposition \ref{YDLiopprop:whenXinIndMGX} on sait que de tel élément $(Y,Pg_1)$ existe, i.e. $h_P'(h_Q'^{-1}(((X,Qg)))\not= \emptyset$. Soit maintenant $(X,Pg_1),(X,Pg_2)\in h_P'(h_Q'^{-1}(((X,Qg)))$. Toujours d'après la définition de $a_{M,\o,G-\reg}$ et l'équivalence (1) $\Leftrightarrow$ (3) de la proposition \ref{YDLiopprop:whenXinIndMGX} on sait qu'il existe $n\in N_P$ tel que $(\Ad hg_1)X=(\Ad g_2)X\in A_{M,\o,G-\reg}\oplus\overline{\o^-}\oplus \n_P$. On a $ng_1g_2^{-1}\in G_{(\Ad g_2)X}$. Puis selon le point 6 de la proposition  \ref{chap3prop:indprop} on obtient $ng_1g_2^{-1}\in P$. D'où $Pg_1=Pg_2$. En conclusion $h_P'(h_Q'^{-1}(((X,Qg)))$ est un singleton.

Supposons ensuite que $(X,Qg)\in \widetilde{\g}_{Q,\Ind}(\o)$. Autrement dit $(\Ad g)X\in (\o\oplus \n_Q)_{G-\reg}$. Soit $(Y,Pg_1)\in \widetilde{\g}_P(\o)$. On sait que $(Y,Pg_1)\in h_P'(h_Q'^{-1}(((X,Qg)))$ si et seulement si $Y=X$ et $(\Ad g_1)Y\in (\o\oplus \n_P)_{G-\reg}$. Comme $(\Ad G)(\o\oplus \n_Q)_{G-\reg}=\Ind_M^G(\o)=(\Ad G)(\o\oplus \n_P)_{G-\reg}$ on sait que de tel élément $(Y,Pg_1)$ existe, i.e. $h_P'(h_Q'^{-1}(((X,Qg)))\not= \emptyset$. Soit maintenant $(X,Pg_1),(X,Pg_2)\in h_P'(h_Q'^{-1}(((X,Qg)))
$. D'après le point 6 de la proposition \ref{chap3prop:indprop} on sait qu'il existe $p\in P$ tel que $(\Ad pg_1)X=(\Ad g_2)X\in (\o\oplus \n_P)_{G-\reg}$. On a $pg_1g_2^{-1}\in G_{(\Ad g_2)X}$. Puis toujours selon le point 6 de la proposition  \ref{chap3prop:indprop} on obtient $pg_1g_2^{-1}\in P$. D'où $Pg_1=Pg_2$. En conclusion $h_P'(h_Q'^{-1}(((X,Qg)))$ est de nouveau un singleton.

Des deux paragraphes précédents, on tire que $C_P=\widetilde{\g}_P(\o)^\sm$, et que le morphisme $h_P'$ est quasi-fini. Puisque $h_P'
$ est aussi projectif (c'est un changement de base d'un morphisme projectif), donc propre, il est fini. 
\`{A} l'aide du lemme \ref{YDLioplem:extendbirmap}, on conclut que $h_{P}'$ est un isomorphisme.

Enfin l'isomorphisme $h_{Q}'\circ h_{P}'^{-1}:\widetilde{\g}_P(\o)^\sm\to\widetilde{\g}_Q(\o)^\sm$ prolonge le morphisme birationnel $f_{Q,P}:\widetilde{\g}_P(\o)^\sm\dashrightarrow\widetilde{\g}_Q(\o)^\sm$. Par restriction $h_{Q}'\circ h_{P}'^{-1}$ induit un isomorphisme $\widetilde{\g}_{P,\Ind}(\o)\to \widetilde{\g}_{Q,\Ind}(\o)$.
\end{proof}



\begin{lemma}\label{lem:weightwadjpara}
Soient $P_1,P_2,P_3\in\P^G(M)$. Soient $A\in\a_{M,\o,G-\reg}(F)$, $Y_1,Y_2\in \o(F)$, $V_1\in \n_{P_1}(F), V_2\in \n_{P_2}(F)$ et $k_1,k_2\in K$ tels que $k_1^{-1}(A+Y_1+V_1)k_1=k_2^{-1}(A+Y_2+V_2)k_2$. Alors
\[w_{P_3|P_1}(\lambda,A,Y_1,V_1)=w_{P_3|P_2}(\lambda,A,Y_2,V_2)w_{P_2|P_1}(\lambda,A,Y_1,V_1),\,\,\,\,\forall \lambda\in ia_M^\ast.\]   
\end{lemma}

\begin{proof}
Pour $P,Q\in\P^G(M)$ notons $\text{dist}(P, Q)$, la distance de $P$ à $Q$, c'est-à-dire la longueur minimale d'une chaîne de sous-groupes paraboliques adjacents de $\P^G(M)$ qui relie $P$ à $Q$. Si $P_2=P_3$ il n'y a rien à faire. Sinon une récurrence simple réduit le cas général au cas où $\text{dist}(P_2,P_3)=1$.

Soit $\Sigma(\mathfrak{p}_3;A_M)\cap (-\Sigma(\mathfrak{p}_2;A_M))=\{\alpha\}$. Alors $w_{P_3|P_1}(\lambda,A,Y_1,V_1)w_{P_2|P_1}(\lambda,A,Y_1,V_1)^{-1}$ vaut
\begin{align*}
v_{P_3}(\lambda,n_1(A,Y_1,V_1))v_{P_2}&(\lambda,n_1(A,Y_1,V_1))^{-1}r_\alpha(\lambda,A,\o)\\
&=v_{P_3}(\lambda,n_2(A,Y_2,V_2))r_\alpha(\lambda,A,\o)=w_{P_3|P_2}(\lambda,A,Y_2,V_2),
\end{align*}
la première égalité vient du fait que $n_{1}(A,Y_1,V_1)k_1k_2^{-1}n_{2}(A,Y_2,V_2)^{-1}\in G_{A+Y}(F)\subseteq M(F)$.
\end{proof}

Pour $P\in \P^G(M_0)$ on a $(P \backslash G)(F)=P(F)\backslash G(F)\simeq K\cap P(F)\backslash K$. Ainsi, dans la situation du lemme précédent, on remarque que l'application $((\Ad k_2^{-1})(A+Y_2+V_2),P_2k_2)\mapsto ((\Ad k_1^{-1})(A+Y_1+V_1),P_1k_1)$ n'est rien d'autre que la restriction de l'isomorphisme de schémas $f_{P_1,P_2}$ du lemme \ref{YDLioplem:prolongerbir}. On en déduit sur-le-champ, pour tous $Y\in\o(F)$ et $V\in \n_{\square}(F)$ avec $Y+V\in \Ind_M^G(\o)(F)$, l'existence de la limite et sa non-annulation
\[w_P(\lambda,Y,V)\eqdef \lim_{A\in \a_{M,\o,G-\reg}(F)\to 0}w_P(\lambda,A,Y,V)\not=0,\]
par une récurrence sur $\text{dist}(P, P_\square)$, le cas adjacent étant une conséquence de la définition. Cela donne une autre $(G,M)$-famille $(w_P(\lambda,Y,V))_{P\in \P^G(M)}$. 

Soit maintenant $\omega\in \Wt(a_M^\ast)$ qui est $P$-dominant. Eu égard à ce qui précède, on a $w_P(\omega,A,Y,V)=\|W_{P,\omega}(A,Y,V)\|$ pour $W_{P,\omega}(A,Y,V)$ un $F$-morphisme défini sur $\a_{M,\o}\times \o\times\n_{\square}$, s'étendant en un $F$-morphisme sur un ouvert de $\a_{M}\times \o\times\n_{\square}$ contenant $\{0\}\times (\o\times\n_{\square})_{G-\reg}$. Ici par abus de notation $(\o\times\n_{\square})_{G-\reg}\eqdef\{(Y,V)\in \o\times\n_{\square} \mid Y+V\in (\o\oplus\n_{\square})_{G-\reg}\}$
En général $w_P(\omega,A,Y,V)$ est l'exponentielle de la valeur de $\lambda$ en un point dans $a_M$, on voit en conséquent que si $\{\omega_1,\dots,\omega_n\}$ est composée d'éléments $P$-dominant de $\Wt(a_M^\ast)$, et $\lambda=\sum_{i=1}^n\lambda_i\omega_i$ avec $\lambda_i\in\C$,
alors $w_P(\lambda,A,Y,V)=\prod_{i=1}^n\|W_{P,\omega_i}(A,Y,V)\|^{\lambda_i}$. \`{A} la faveur de l'équation \eqref{eq:GMcalcul}, le poids $w_M(A,Y,V)$, dans le domaine $(Y,V)\in (\o\times \n_{\square})_{G-\reg}(F)$ et $A$ assez proche de $0$, est une somme finie 
\[\sum_{\omega\text{ fini}}c_\omega\left(\prod_{(P,\omega)\in\Omega_\omega}\log\|W_{P,\omega}(A,Y,V)\|\right)\]
où $c_\omega\in \C$ et chaque $\Omega_\omega$ est un multiensemble fini composé d'éléments de la forme $(P,\omega)$ avec $P\in \P^G(M)$ et $\omega\in \Wt(a_M^\ast)$.

On prendra dans la suite $A\in \a_{M,\o,G-\reg}(F)$. Alors $M_{A+Y}=L_{A+Y}$ pour tout $L\in\L^G(M)$. On note $r_\alpha(\lambda,A,Y)\eqdef r_\alpha(\lambda,A,\o)$ pour tout $\alpha\in \Sigma(\g;A_M)$. Posons pour $P\in \P^G(M)$
\[r_P(\lambda,A,Y)\eqdef\prod_{\alpha \in \Sigma(\mathfrak{p};A_{M})}r_\alpha\left(\frac{\lambda}{2},A,Y\right).\]
Pour tout $L\in \L^G(M)$ on a une $(L,M)$-famille $(r_R^L(\lambda,A,Y))_{R\in \P^L(M)}$, où $L$ prend le rôle de $G$ ci-dessus. De l'autre côté pour tout $Q\in \P^G(L)$ on a une $(L,M)$-famille $(r_R^Q(\lambda,A,Y))_{R\in\P^L(M)}=(r_{RN_Q}(\lambda,A,Y))_{R\in\P^L(M)}$. On sait que $r_M^L(A,Y)=r_M^Q(A,Y)$ pour tout $Q\in \P^G(L)$ (cf. \cite[section 7]{Art82II}). 

\begin{lemma}\label{lem:Art88lem5.3} Pour tous $(A,Y,V)\in \a_{M,\o,G-\reg}(F)\times \o(F)\times \n_{\square}(F)$,
\[\sum_{L\in \L^{M_Q}(M)}r_M^L(A,Y)v_L^G(n_\square(A,Y,V))=w_M^{G}(A,Y,V).\]
\end{lemma}
\begin{proof}
Soit $P\in \P^G(M)$. Un argument combinatoire donne
\begin{align*}
\prod_{\alpha\in \Sigma(\mathfrak{p};A_{M})}&r_\alpha\left(\frac{\lambda}{2},A,Y\right)\\
&=\left(\prod_{\alpha\in \Sigma(\mathfrak{p};A_{M})\cap (-\Sigma(\mathfrak{p}_{\square};A_{M}))}r_\alpha\left(\lambda,A,Y\right)\right)\left(\prod_{\alpha\in \Sigma(\mathfrak{p}_{\square};A_{M})}r_\alpha\left(\frac{\lambda}{2},A,Y\right)\right).    
\end{align*}
Le produit $r_P(\lambda,A,Y)v_P(\lambda,n_R(A,Y,V))$ est alors $w_P(\lambda,A,Y,V)$ fois $\prod_{\alpha\in \Sigma(\mathfrak{p}_{\square};A_{M})}r_\alpha\left(\frac{\lambda}{2},A,Y\right)$ qui est indépendant de $P$. On conclut en utilisant la formule de produit.
\end{proof}

Soient $f\in\S(\g(F))$, $M\in \L^G(M_0)$, $Q\in \F^G(M)$, et $Y\in \m(F)$. On suppose que $M_Y=G_Y$. Posons 
\begin{equation}\label{YDLiopeq:defIOPArtdirect}
\widetilde{J}_{M}^Q(Y,f)=|D^{\g}(Y)|^{1/2}\int_{M_Y(F)\backslash G(F)} f\left((\Ad g^{-1})Y\right)v_M^Q(g)\,dg.    
\end{equation}
La proposition \ref{prop:LX=GX} nous garantit que $v_{M,X}^Q=v_M^Q$ avec $X$ le représentant standard de $\Ind_M^G(Y)$. Autrement dit $\widetilde{J}_{L}^Q(Y,f)=J_{L}^Q(Y,f)$ dans cette situation. Ce qui donne à part la convergence de l'intégrale \eqref{YDLiopeq:defIOPArtdirect}. 

Pour rappel on a posé $\a_{M,Y,G-\reg}\eqdef \a_{M,(\Ad M)Y,G-\reg}$ (équation \eqref{YDLiopeq:defaMYG-reg}).

Pour $X$ un espace topologique, on note $C^0(X)$ l'espace des fonctions continues sur $X$ à valeurs complexes.

\begin{theorem}\label{YDLiopthm:ArtdefdirectIOP}
Soient $f\in\S(\g(F))$, $M\in \L^G(M_0)$, $Q\in \F^G(M)$, et $Y\in \m(F)$. Alors
\begin{equation}\label{eq:NewArthurWOI}
 \widetilde{J}_M^Q(Y,f)\eqdef\lim_{ \substack{A\to 0\\A\in \a_{M,Y,G-\reg}(F)}}\sum_{L\in \L^{M_Q}(M)}r_M^L(A,Y)\widetilde{J}_{L}^Q(A+Y,f)   
\end{equation}
existe. Il existe $c>0$ et $a\in \N$ tels que $|\widetilde{J}_L^Q(Y,-)|$ soit majoré par $c\|-\|_{a,0}$. C'est donc une distribution tempérée.
\end{theorem}

\begin{remark}~{}
\begin{enumerate}
    \item Dans le cas où $M_Y=G_Y$, l'équivalence (5) $\Leftrightarrow$ (7) de la proposition \ref{YDLiopprop:whenXinIndMGX} nous affirme que $n_\square(A,Y,V)$ admet une limite en $A=0$, d'où
    \[r_M^L(A,Y)=\begin{cases}
        1, & \text{si }L=M\\ 0,& \text{si } L\not=M,
    \end{cases}\]
    ce qui justifie l'écriture.
    \item $\widetilde{J}_M^Q(Y,f)$ ne dépend que de la classe de $M(F)$-conjugaison de $Y$. On peut donc noter $\widetilde{J}_M^Q(\o,f)$ à la place de $\widetilde{J}_M^Q(Y,f)$, où $\o=(\Ad M(F))Y$. 
\end{enumerate}

\end{remark}
\begin{proof}
Sans perte de généralité on peut supposer que $Q=G$. 
Fixons un point-base $P_\square\in \P^G(M)$. La somme $|D^{\g}(A+Y)|^{-1/2}\sum_{L\in \L^{G}(M)}r_M^L(A,Y)\widetilde{J}_{L}^G(A+Y,f)$ vaut
\begin{align*}
&\int_{M_{Y}(F)\backslash G(F)} f((\Ad g^{-1})(A+Y))\sum_{L\in \L^{G}(M)}r_M^L(A,Y)v_L^G(g)\,dg\\
&\hspace{1cm}=\gamma^{G}(P_\square)\int_{M_Y(F)\backslash M(F)}\int_{N_\square(F)}\int_{K} f(k^{-1}n^{-1}m^{-1}(A+Y)mnk)\sum_{L\in \L^{G}(M)}r_M^L(A,Y)v_L^G(n)\,dk\,dn\,dm  \\
&\hspace{1cm}=\gamma^{G}(P_\square)\cdot \text{Jac}\cdot\int_{M_Y(F)\backslash M(F)}\int_{\n_\square(F)}\int_{K} f(k^{-1}(A+m^{-1}Ym+V)k)\\
&\hspace{5cm}\times\sum_{L\in \L^{G}(M)}r_M^L(A,Y)v_L^G(n_\square(A,m^{-1}Ym,V))\,dk\,dV\,dm\\
&\hspace{1cm}=\gamma^{G}(P_\square)\cdot \text{Jac}\cdot\int_{M_Y(F)\backslash M(F)}\int_{\n_\square(F)}\int_{K} f(k^{-1}(A+m^{-1}Ym+V)k) w_M^{G}(A,m^{-1}Ym,V)\,dk\,dV\,dm
\end{align*}
où $\text{Jac}=|\det(\ad(A+Y);\n_P(F))|^{-1}$. Fixons une sous-partie $C$ compacte de $\a_M(F)$ contenant $0$. Des discussions précédentes, on déduit l'existence de $c>0$, $e\in \N_{>0}$, $B$ un polynôme à une variable et à coefficients positifs, tels que $B(0)>3$ et $|w_M^{G}(A,m^{-1}Ym,V)|<c\left(\log B(\|m^{-1}Ym\|_{\g}) +\log B(\|V\|_{\g})\right)^e$ pour tous $A\in C$ et $(m^{-1}Ym,V)\in (((\Ad M)Y)\times \n_\square)_{G-\reg}(F)$. 

Posons $f'(-)\eqdef \max_{A\in C}|f(A+-)|$, on obtient une fonction
\[f'\in \begin{cases}
\S(\g(F)) & \text{si $F$ est non-archimédien} \\
C^0(\g(F)) & \text{si $F$ est archimédien}
\end{cases}\]
vérifiant que $\|f'\|_{a,0}<+\infty$ pour tout $a\in \N$. Par une réduction immédiate, on est conduit au problème de convergence de 
\begin{align*}
\int_{M_Y(F)\backslash M(F)}f'\left(m^{-1}Ym\right)(\log B(\|m^{-1}Ym\|_{\m}))^e\,dm.   
\end{align*}
On a déjà rencontré ce type d'intégrale dans la preuve du théorème \ref{thm:IOPfullstatement}. Le théorème de convergence dominée implique l'existence de \eqref{eq:NewArthurWOI} en $A=0$.
\end{proof}

\subsection{Définition via descente semi-simple}\label{subsec:defviadescentss}

On offre encore une définition d'une intégrale orbitale pondérée. Celle-ci s'accorde avec sa version pour les groupes réductifs proposée par Arthur.

Soient $M\in\L^G(M_0)$ et $\o\subseteq \m$ une orbite sous $M$-conjugaison contenant un $F$-point. Soit $Y\in \o(F)$. Soit $A\in \a_{M,Y,G-\reg}(F)$.

Fixons $P_\square\in \P^{G}(M)$, il nous fournit le point-base $P_{\square,Y_\ss}\in \P^{G_{Y_\ss}}(M_{Y_\ss})$. Les points 4 et 5 de la proposition \ref{YDLiopprop:whatisaMoG-reg} disent aussi que 
\[\a_{M,Y,G-\reg}\subseteq \a_{M_{Y_\ss},G_{Y_\ss}-\reg}=\a_{M_{Y_\ss},Y_\nilp,G_{Y_\ss}-\reg}.\]
On a alors les fonctions $r_\alpha(\lambda,A,Y_{\nilp})$ relativement au groupe $G_{Y_\ss}$. Posons pour $P\in \P^G(M)$
\[r_P[\lambda,A,Y]\eqdef\prod_{\alpha \in \Sigma(\mathfrak{p}_{Y_{\ss}};A_{M_{Y_{\ss}}})}r_\alpha(\lambda,A,Y_{\nilp})\]
Notons que la restriction de toute racine de $\Sigma(\mathfrak{p}_{Y_{\ss}};A_{M_{Y_{\ss}}})$ à $a_M\subseteq a_{M_{Y_{\ss}}}$ appartient à $\Sigma(\mathfrak{p};A_{M})$. Cela définit une $(G,M)$-famille. Pour tout $L\in \L^G(M)$ on a une $(L,M)$-famille $(r_R^L[\lambda,A,Y])_{R\in \P^L(M)}$ où $L$ prend le rôle de $G$ ci-dessus. De l'autre côté pour tout $Q\in \P^G(L)$ on a une $(L,M)$-famille $(r_R^Q[\lambda,A,Y])_{R\in\P^L(M)}=(r_{RN_Q}[\lambda,A,Y])_{R\in\P^L(M)}$. On sait que $r_M^L[A,Y]=r_M^Q[A,Y]$ pour tout $Q\in \P^G(L)$. 

De même, on peut poser pour $P\in \P^G(M)$ 
\[w_P[\lambda,A,Y_{\ss}+U,V]\eqdef w_{P_{Y_\ss}}(\lambda,A,U,V)\]
où $\lambda\in ia_M^\ast$, $U\in(\Ad M_{Y_\ss}(F))(Y_\nilp)$ et $V\in \n_{\square,Y_\ss
}(F)$. Encore une fois la fonction de droite $w_{P_{Y_\ss}}$ est définie relativement au groupe $G_{Y_\ss}$. La famille $(w_P[-,A,Y_{\ss}+U,V])_{P\in\P^G(M)}$ définit une $(G,M)$-famille.

Soient $f\in\S(\g(F))$, $M\in \L^G(M_0)$, $Q\in \F^G(M)$, et $Y\in \m(F)$. On suppose d'abord que $M_Y=G_Y$. Posons 
\[\widetilde{J}_{M}^Q[Y,f]=|D^{\g}(Y)|^{1/2}\int_{M_Y(F)\backslash G(F)} f\left((\Ad g^{-1})Y\right)v_M^Q(g)\,dg.\]
Idem $\widetilde{J}_{M}^Q[Y,f]=J_
{M}^Q(Y,f)$ dans cette situation, d'où la convergence.

Pour $F$ non-archimédien on note $C_c^\infty(\g(F))=\S(\g(F))$  l'espace des fonctions localement constantes à support compact et à valeurs complexes sur $\g(F)$ ; pour $F$ archimédien on note $C_c^\infty(\g(F))$ l'espace des fonctions lisses (au sens réel) à support compact sur $\g(F)$. 

\begin{theorem}\label{YDLiopthm:ArtdefIOPCVCinftyc}
Soient $f\in C_c^\infty(\g(F))$, $M\in \L^G(M_0)$, $Q\in \F^G(M)$, et $Y\in \m(F)$. Alors
\[\widetilde{J}_M^Q[Y,f]\eqdef\lim_{\substack{A\to 0\\ A\in \a_{M,Y,G-\reg}(F)}}\sum_{L\in \L^{M_Q}(M)}r_M^L[A,Y]\tilde{J}_{L}^Q[A+Y,f]\]
existe. 
\end{theorem}

\begin{remark}~{}
\begin{enumerate}
    \item Dans le cas où $M_Y=G_Y$, on a $
M_{Y_\ss}=G_{Y_\ss}$, d'où
    \[r_M^L[A,Y]=\begin{cases}
        1, & \text{si }L=M\\ 0,& \text{si } L\not=M,
    \end{cases}\]
    ce qui justifie l'écriture.
    \item Si l'on suit l'approche d'Arthur (\cite[équation (3*) en page 224]{Art88}), il est nécessaire de remplacer le domaine de la limite « $\a_{M,Y,G-\reg}(F)$ » dans la définition de $\widetilde{J}_M^Q[Y,f]$ par « $\a_{M,G-\reg}(F)$ ». Cependant, le point 2 de la proposition \ref{YDLiopprop:whatisaMoG-reg} nous indique que l'existence de la limite $\lim_{A\in \a_{M,Y,G-\reg}(F)\to 0}$ constitue une assertion plus forte.
    \item $\widetilde{J}_M^Q[Y,f]$ ne dépend que de la classe de $M(F)$-conjugaison de $Y$. On peut donc noter $\widetilde{J}_M^Q[\o,f]$ à la place de $\widetilde{J}_M^Q[Y,f]$, où $\o=(\Ad M(F))Y$. 
\end{enumerate}
\end{remark}
\begin{proof}

Notons $\sigma\eqdef Y_\ss$. 
Pour tout $B\in \F^{M_Q}(M)$ fixons $R=R^B\in \P^{M_{B_\sigma}}(M_\sigma)$. La somme $\sum_{L\in \L^{M_Q}(M)}r_M^L[Y,A]\widetilde{J}_{L}^Q[A+Y,f]$ vaut  
\begingroup
\allowdisplaybreaks
\begin{align*}
&|D^{\g}(A+Y)|^{1/2}\int_{G_{A+Y}(F)\backslash G(F)} f\left((\Ad g^{-1})(A+Y)\right)\left(\sum_{L\in \L^{M_Q}(M)}r_M^L[Y,A] v_L^Q(g)\right)\, dg \\
&=|D^{\g}(A+Y)|^{1/2}\int_{y}\int_{x}\int_{m} f\left((\Ad (xy)^{-1})\left[A+\sigma+(\Ad m^{-1})Y_\nilp\right]\right)\\
&\hspace{8cm}\times\left(\sum_{L\in \L^{M_Q}(M)}r_M^L[Y,A] v_L^Q(mxy)\right)\, dm\, dx\,dy
\end{align*}
\endgroup
avec $(m,x,y)\in (M_{Y}(F)\backslash M_\sigma(F))\times (M_\sigma(F)\backslash G_\sigma(F))\times (G_\sigma(F)\backslash G(F))$. En appliquant la formule de produit on voit que
\[v_L^Q(mxy)=v_L^Q(xy)=\sum_{B\in \F^{M_Q}(L)}v_L^B(\sigma,x)v_B^Q{}'(K_{\sigma S}(x)y),\]
avec $v_L^B(\sigma,x)$ le poids associé au $(L_{B_\sigma},L_{\sigma})$-famille $(v_{P_\sigma}^{B_\sigma}(\lambda,x))_P$. Inversons la somme sur $L\in \L^{M_Q}(M)$ et celle sur $B\in \F^{M_Q}(L)$, on tombe sur la somme sur $B\in \F^{M_Q}(M)$ et $L\in \L^{M_B}(M)$. Effectuons ensuite le changement de variables $x=n_Rn_Bk$, $(n_R,n_B,k)\in N_R(F)\times N_{B_{\sigma}}(F)\times K_{\sigma }$ puis 
\begin{equation}\label{Art88eq:lemma5.3}
(\Ad n_R^{-1})\left[A+\sigma+(\Ad m^{-1})Y_\nilp\right]=A+\sigma+(\Ad m^{-1})Y_\nilp+V_R,\,\,\,\,V_R\in \n_{R}(F)   
\end{equation}
(dont la réciproque est notée $V\mapsto n_{R}(A+\sigma,(\Ad m^{-1})Y_{\nilp},V)$) et
\begin{equation*}
(\Ad n_B^{-1})\left[A+(\Ad m^{-1})Y_\nilp+V_R\right]=A+(\Ad m^{-1})Y_\nilp+V_R+V_B,\,\,\,\,V_B\in \n_{B_{\sigma}}(F)   
\end{equation*}
(dont la réciproque est notée $V\mapsto n_{B_{\sigma}}(A+\sigma,(\Ad m^{-1})Y_{\nilp},V)$), qui introduisent successivement des facteurs Jacobiens $|D^{\m_{B_\sigma}}((\Ad m^{-1})(A+Y_\nilp))|^{-1/2}|D^{\m_\sigma}((\Ad m^{-1})(A+Y_\nilp))|^{1/2}=|D^{\m_{B_\sigma}}(A)|^{-1/2}$ (cela vient de la valeur absolue du déterminant de $\ad ((\Ad m^{-1})(A+Y_\nilp))$ agissant sur $\n_R(F)$) puis $|D^{\g_{\sigma}}(A)|^{-1/2}|D^{\m_{
B_\sigma}}(A)|^{1/2}$. 
Notons ultérieurement 
\begin{lemma}
Pour $n_R,V_R$ assujettis à l'équation \eqref{Art88eq:lemma5.3} on a 
\[\sum_{L\in \L^{M_B}(M)}r_M^L[Y,A]v_L^B(\sigma,n_R)=w_M^{B}[A,(\Ad m^{-1})Y,V_R].\]
\end{lemma}
\begin{proof}
C'est la même preuve que lemme \ref{lem:Art88lem5.3}.
\end{proof}

On va pouvoir en conclure que la somme $\sum_{L\in \L^{M_Q}(M)}r_M^L[Y,A]\widetilde{J}_{L}^Q[A+Y,f]$ vaut  
\begin{equation}\label{eq:compareArt}
\begin{split}
&|D^{\g}(A+Y)|^{1/2}|D^{\g_\sigma}(A)|^{-1/2}\int_{y\in G_\sigma(F)\backslash G(F)}\sum_{B\in\F^{M_Q}(M)}\int_{m\in M_{Y}(F)\backslash M_\sigma(F)}\int_{V_R\in \n_R(F)}\\
&\hspace{4cm}\Phi_{B,y}^{\sigma}\left(A+(\Ad m^{-1})Y_\nilp+V_R\right)w_M^B[A,(\Ad m^{-1})Y,V_R,\sigma]\,dV_R\,dm\,dy  
\end{split}    
\end{equation}
avec
\begin{align*}
\Phi_{B,y}^{\sigma}(M)\eqdef   \gamma^{(M_Q)_\sigma}(B_\sigma) \int_{K_{\sigma S}}\int_{\n_{B_\sigma}(F)}f\left((\Ad (ky)^{-1})\left[\sigma+M+V_B\right]\right)v_B^Q{}'(ky)\,dV_B\,dk,\,\,\,\, M\in \m_{B_{\sigma}}(F).
\end{align*}
Il est clair que $\Phi_{B,y}^{\sigma}\in C_c^\infty(\m_{B_{\sigma}}(F))$ dépend de façon lisse en $y$. On peut réduire le domaine d'intégration de $y\in G_\sigma(F)\backslash G(F)$ à une partie compacte de $G(F)$ indépendante de $A$ selon un lemme de compacité de Harish-Chandra déjà mentionné (cf. \cite[lemme 14.1]{Kottbook}). Il est à présent clair que l'intégrale (\ref{eq:compareArt}) converge abosolument. \qedhere
\end{proof}

\subsection{Premières propriétés des intégrales orbitales pondérées}

Avant de confronter ces définitions d'une intégrale orbitale pondérée, on présente des propriétés des intégrales orbitales pondérées.

\begin{proposition}\label{prop:propertiesIOP}Soient  $L\in\L^G(M_0),Q\in\F^G(L),Y\in \mathfrak{l}(F)$ et $f\in\S(\g(F))$. On utilise ici la définition \ref{def:NIOPl} pour l'intégrale orbitale pondérée $J_L^Q(Y,f)$.
\begin{enumerate}
    \item Pour tous $l\in L(F)$, $w\in \uW^{G,0}$ et $k\in K$,
    \[J_{(\Ad w)L}^{(\Ad w)Q}((\Ad wl)Y,(\Ad  k)f)=J_L^Q(Y,f).\]
    \item Formule de descente de l'induction : soient $M\in\L^L(M_0)$ avec $Z\in\m(F)$. On a 
    \[J_L^Q(\Ind_M^L(Z),f)=\sum_{M_1\in\L^{L_Q}(M)}d_M^{L_Q}(L,M_1)J_M^{Q_{M_1}}(Z,f).\]
    Ici $M_1\mapsto Q_{M_1}$ est la section dans la formule de descente pour les $(G,M)$-familles.
    \item Formule de descente semi-simple : supposons que $Y$ est le représentant standard de sa $L(F)$-orbite, et $Y_{\ss}$ est $F$-elliptique dans $L$. Alors $J_L^Q(Y,f)$ égale
    \[|D^\g(Y)|^{1/2}\int_{G_{Y_\ss}(F)\backslash G(F)}\left(\sum_{R\in\F^{L_{Q,Y_\ss}}(L_{Y_\ss})}J_{L_{Y_\ss}}^{L_{Y_\ss,R}}(Y_\nilp,\Phi_{R,y})\right)\,dy,\]
    où $\Phi_{R,y}\in\S(\m_{R}(F))$ est la fonction 
    \[\Phi_{R,y}(M)=\gamma^{(L_Q)_{Y_\ss}}(R)\int_{w^{-1}K_{X_\ss} w}\int_{\mathfrak{n}_R(F)}f\left((\Ad (ky)^{-1})(\sigma+M+U)\right)v_R'(ky)\,dU\,dk,\,\,\,\,\forall M\in\m_R(F),\]
    avec $X$ le représentant standard de $\Ind_L^G(Y)$, $w$ un élément de $\uW^{G,0}$ tel que $ {}^{w^{-1}}\MR^{X}\subseteq L$, $(\Ad L)Y=\Ind_{{}^{w^{-1}}\MR^{X}}^L((\Ad w^{-1})X_\ss)$, et $(\Ad w^{-1})X_\ss=Y_\ss$ (proposition \ref{prop:paradesccomp}). Enfin $K_{X_\ss}=K_{X_\ss}^{e_{X_\ss}}$ le sous-groupe compact de $G_{X_\ss}(F)$ en bonne position relativement à $e_{X_\ss}$. Les sous-groupes de $G_{Y_\ss}(F)$ en question sont munis des mesures de Haar vérifiant les consignes de la sous-section \ref{YDLiopsubsec:normalizationmeasure}. 
\end{enumerate}
\end{proposition}

\begin{proof}
Le point 1 est facile. Le point 2 provient de l'équation \eqref{eq:GMfamilledescentformula}. On prouve le point 3. Notons $X$ le représentant standard de $\Ind_L^G(Y)(F)$. Posons $\sigma\eqdef X_\ss$ et  $\mu\eqdef X_\nilp$. \'{E}crivons ${}^wA\eqdef (\Ad w)A$ pour $w\in \uW^{G,0}$ et $A\in \L^G(M)$. Soit $w \in \uW^{G,0}$ tel que $ {}^{w^{-1}}\MR^{X}\subseteq L$, $(\Ad L)Y=\Ind_{{}^{w^{-1}}\MR^{X}}^L((\Ad w^{-1})X_\ss)$, et $(\Ad w^{-1})X_\ss=Y_\ss$. L'élément $\sigma$ est $F$-elliptique dans ${}^wL$.  Le sous-groupe groupe de Levi $({}^wL)_\sigma$ contient $\uM_0^\sigma$, on obtient alors la décomposition d'Iwasawa $G_{\sigma}(F)=R(F)K_{\sigma}$ pour tout $R\in \F^{G_\sigma}(({}^wL)_\sigma)$. 

L'intégrale $J_L^Q(Y,f)$ vaut
\begin{align*}
|D^\g(X)|^{1/2}\int_{y\in G_\sigma(F)\backslash G(F)}\int_{x\in (G_\sigma)_\mu(F)\backslash G_\sigma(F)}f&\left((\Ad(xy)^{-1})X\right)v_{{}^wL,X}^{{}^wQ}(xy)\,dx\,dy.  
\end{align*}

Pour $P\in \P^G({}^wL)$ et $x\in G_\sigma(F)$ on note $K_{P_\sigma}(x)$ un élément de $K_{\sigma }$ tel que $xK_{P_\sigma}(x)^{-1}$ appartient à $P_\sigma(F)$. Il est uniquement déterminé modulo $K_{\sigma }\cap P_\sigma(F)$ à gauche. Nous avons
\begin{align*}
v_{P,X}^{{}^wQ}(\lambda,xy)&=\exp(\langle\lambda,  -H_P(w_{P,X}xy)\rangle) \\
&=\exp(\langle\lambda,-H_P(w_{P_\sigma,\mu}x)\rangle)\exp(\langle\lambda,-H_P(K_{P_\sigma}(w_{P_\sigma,\mu}x)y)\rangle)\\
&=\exp(\langle\lambda,-H_{P_\sigma}(w_{P_\sigma,\mu}x)\rangle)\exp(\langle\lambda,-H_P(K_{P_\sigma}(w_{P_\sigma,\mu}x)y)\rangle),
\end{align*}
le lemme \ref{lem:wP} est invoqué au passage. En conséquence

\begin{lemma}Avec les notations précédentes,
\[v_{{}^wL,X}^{{}^wQ}(x)=\begin{cases*} v_{({}^wL)_\sigma,\mu}^{({}^wQ)_\sigma}(x), & si $a_{{}^wQ}=a_{({}^wQ)_\sigma}$\\ 0,& sinon.
\end{cases*}\]
\end{lemma}
\begin{proof}
La preuve est la même qu'en \cite[lemme 8.3]{Art88}.
\end{proof}

De ce fait en notant pour tout $R\in\F^{({}^wL_Q)_\sigma}(({}^wL)_\sigma)$
\[v_R'(z,T)=\sum_{\{Q\in\F^{{}^wL_Q}({}^wL):Q_\sigma=R,a_Q=a_R\}}v_Q'(z,T),\,\,\,\,z\in G(F)\]
on aura
\begin{align*}
v_{{}^wL,X}^{{}^wQ}(xy)=\sum_{R\in\F^{({}^wL_Q)_\sigma}(({}^wL)_\sigma)}v_{({}^wL)_\sigma,\mu}^R(x)v_R'(K_{R}(w_{R,\mu}x)y).
\end{align*}
\`{A} ce stade nous avons
\begin{equation}\label{eq:IOPssR}
\begin{aligned}
J_L^Q(Y,f)=|D^\g(X)|^{1/2}\int_{y\in G_\sigma(F)\backslash G(F)}\sum_{R\in\F^{({}^wL_Q)_\sigma}(({}^wL)_\sigma)}&\int_{x\in (G_\sigma)_\mu(F)\backslash G_\sigma(F)}f\left((\Ad(xy)^{-1})X\right)\\
\times &v_{({}^wL)_\sigma,\mu}^R(x)v_R'(K_{R}(w_{R,\mu}x)y,T)\,dx\,dy.
\end{aligned}
\end{equation}

Fixons $R'\in\F^{({}^wL_Q)_\sigma}(({}^wL)_\sigma)$. Posons, pour $y\in G_\sigma(F)\backslash G(F)$, $\Phi_{R',y}\in\S(\m_{R'}(F))$ la fonction
\[\Phi_{R',y}(M)=\gamma^{({}^wL_Q)_\sigma}(R')\int_{K_\sigma}\int_{\mathfrak{n}_{R'}(F)}f\left((\Ad (ky)^{-1})(\sigma+M+U)\right)v_{R'}'(ky)\,dU\,dk,\,\,\,\,\forall M\in\m_{R'}(F).\]
Il est clair que $\Phi_{R',y}$ dépend de manière lisse de $y$. On note ensuite $\mu_{R'}$ le représantant standard de $\Ind_{({}^wL)_\sigma}^{({}^wL)_{\sigma,{R'}}}((\Ad w)Y_\nilp)(F)$, pour rappel $({}^wL)_{\sigma,R'}$ est le facteur de Levi de $R'$ contenant $({}^wL)_{\sigma}$. Bien sûr $\mu_{G_\sigma}=\mu$. Comme $(\Ad w_{R,\mu}^{-1})R'\in \cLS^{G_\sigma}(\mu)$, on en déduit $(\Ad w_{R,\mu})\mu \in\mathfrak{r}'(F)$. Dans la dernière intégrale portée sur $(G_\sigma)_\mu(F)\backslash G_\sigma(F)$ de l'équation \eqref{eq:IOPssR}, nous effectuons d'abord le changement de variable
\[w_{R',\mu}x=z,\,\,\,\,x\in (G_\sigma)_\mu(F)\backslash G_\sigma(F), z\in (G_\sigma)_{(\Ad w_{R',\mu})\mu}(F)\backslash G_\sigma(F),\]
puis
\[z^{-1}(w_{R',\mu}\mu w_{R',\mu}^{-1})z=k_\sigma^{-1}(m_{R'}^{-1}\mu_{R'}m_{R'}+U_{R'})k_\sigma,\]
avec $(m_{R'},U_{R'})\in (({}^wL)_{\sigma
,{R'},\mu_{R'}}(F)\backslash ({}^wL)_{\sigma,{R'}}(F)\times \n_{R'}(F))_{G_\sigma-\reg}$, un ouvert dense de $({}^wL)_{\sigma,{R'},\mu_{R'}}(F)\backslash({}^wL)_{\sigma,{R'}}(F)\times \n_{R'}(F)$, et $k_\sigma\in K_{\sigma}$. En observons que le Jacobien introduit vaut 1, puis $K_{R'}(w_{{R'},\mu}x)=K_{R'}(z)=k_\sigma$, nous obtenons 
\begin{align*}
J_L^Q(Y,f)&=|D^\g(X_\ss)|^{1/2}\int_{G_\sigma(F)\backslash G(F)}\left(\sum_{{R'}\in\F^{({}^wL_Q)_\sigma}(({}^wL)_\sigma)}J_{({}^wL)_{\sigma}}^{({}^wL)_{\sigma,{R'}}}((\Ad w)Y_\nilp,\Phi_{{R'},y})\right)\,dy \\   
&=|D^\g(Y_\ss)|^{1/2}\int_{G_{Y_\ss}(F)\backslash G(F)}\left(\sum_{R\in\F^{(L_Q)_{Y_\ss}}(L_{Y_\ss})}J_{L_{Y_\ss}}^{L_{Y_\ss,R}}(Y_\nilp,\Phi_{R,y})\right)\,dy.
\end{align*}
La preuve s'achève.
\end{proof}

Ces trois propriétés sont satisfaites par la définition $\widetilde{J}_L^Q[Y,-]$ pour les fonctions lisse à support compact. Pour les points 1 et 2 cela est facile à voir, pour le point 3 on se réfère à \cite[corollaire 8.7]{Art88}.

\subsection{Lemmes supplémentaires}

Soit dans ce numéro $D=\R$, $\C$, ou $\mathbb{H}$, l'algèbre des quaternions de Hamilton. On note $X\mapsto \overline{X}$ l'involution principale sur $D$.

\begin{definition}
Soit $V$ un $D$-module à droite libre de rang fini. Un produit scalaire sur $V$ est une application $\langle-,-\rangle:V\times V\to D$ vérifiant
\begin{enumerate}
    \item $\langle v,v\rangle\geq 0$ pour tout $v\in V$ et $\langle v,v\rangle=0$ si et seulement si $v=0$ ;
    \item $\langle v_1d_1+v_2d_2,v\rangle=\langle v_1,v\rangle d_1+\langle v_2,v\rangle d_2$ pour tous $v_1,v_2,v\in V$ et $d_1,d_2\in D$ ;
    \item $\langle v_1,v_2\rangle=\overline{\langle v_2,v_1\rangle}$ pour tous $v_1,v_2\in V$.
\end{enumerate}
On dit que $v_1,\dots,v_m\in V$ est une base orthonormée si $\langle v_i,v_j\rangle=\delta_{ij}$.
\end{definition}

\begin{lemma}~{}
\begin{enumerate}
    \item Soit $V$ un $D$-module à droite libre de rang fini muni d'un produit scalaire, alors il existe une  base orthonormée.
    \item Soit $E$ un sous-$D$-module à droite. Posons $E^\perp\eqdef\{v\in V : \langle e,v\rangle=\langle v,e\rangle=0 \text{ pour tout $
e\in E$}\}$. Alors $V=E\oplus E^\perp$.
\end{enumerate}
\end{lemma}

\begin{lemma}[Lemme de Hensel]
Soient $R$ un anneau hensélien, $\mathfrak{m}_R$ son idéal maximal, et $S$ un $R$-schéma lisse. Alors l'application naturelle $S(R)\to S(R/\mathfrak{m}_R)$ est surjective. 
\end{lemma}

\subsection{Comparaison}\label{subsec:WOIcomparaisonLiealg}

On compare d'abord $\widetilde{J}_M^Q(Y,f)$ et $\widetilde{J}_M^Q[Y,f]$. Il y a une égalité entre les deux comme ce que montre le théorème suivant. 

\begin{theorem}\label{thm:IOPcomparaisonII-III}
Soient $f\in C_c^\infty(\g(F))$, $M\in \L^G(M_0)$, $Q\in \F^G(M)$, et $Y\in \m(F)$. Alors $\widetilde{J}_M^Q(Y,f)=\widetilde{J}_M^Q[Y,f]$.
\end{theorem}
\begin{proof}
On a vu, dans le cas où $G_Y=M_Y$, que $\widetilde{J}_M^Q(Y,f)=J_M^Q(Y,f)= \widetilde{J}_M^Q[Y,f]$. Il suffit donc d'établir :
\begin{lemma}
Soient $M\in \L^G(M_0)$, $L\in \L^G(M)$, $Y\in \m(F)$, et $A\in \a_{M,Y,G-\reg}(F)$. Alors 
\[r_M^L(A,Y)=r_M^L[A,Y].\]
\end{lemma}
\begin{proof}
Par la formule de descente d'une $(G,M)$-famille on peut supposer que $Y_\ss$ est $F$-elliptique dans $\m(F)$. Il s'agit d'établir l'égalité entre $\prod_{\alpha\in \Sigma(\mathfrak{p}_{Y_\ss};A_{M_{Y_\ss}})}r_{\alpha}(\lambda,A,Y)$ et $\prod_{\alpha\in \Sigma(\mathfrak{p};A_{M})}r_{\alpha}(\lambda,A,Y)$. Comme 
$A_{Y_\ss}=A_M$, la restriction naturelle $\text{res}:\Sigma(\mathfrak{p}_{Y_\ss};A_{M_{Y_\ss}})\rightarrow \Sigma(\mathfrak{p};A_M)$ est injective. On voit bien $r_{\alpha}(\lambda,A,Y)=r_{\text{res}(\alpha)}(\lambda,A,Y)$. Soit $\alpha\in \Sigma(\mathfrak{p};A_M)$ qui n'est pas dans l'image de $\text{res}$, par des calculs directs on voit que $n_{\square}(A,Y,V)$ admet une limite en $A=0$ pour $V\in \n_{\square}(F)\cap\m_{\alpha}(F)$, en d'autres mots $r_{\alpha}(\lambda,A,Y)=1$. En définitive $\prod_{\alpha\in \Sigma(\mathfrak{p}_{Y_\ss};A_{M_{Y_\ss}})}r_{\alpha}(\lambda,A,Y)=\prod_{\alpha\in \Sigma(\mathfrak{p};A_{M})}r_{\alpha}(\lambda,A,Y)$.
\end{proof}

La preuve s'achève.    
\end{proof}

Au même titre on aimerait comparer $\widetilde{J}_M^Q(Y,f)$ et $J_M^Q(Y,f)$.

\begin{theorem}\label{thm:IOPcomparaisonI-II}
Soient $f\in \S(\g(F))$, $M\in \L^G(M_0)$, $Q\in \F^G(M)$, et $Y\in \m(F)$. Alors $\widetilde{J}_M^Q(Y,f)=J_M^Q(Y,f)$.
\end{theorem}

\begin{proof}
Notons $X$ le représentant standard de $\Ind_M^G(Y)(F)$. En conjuguant éventuellement $M$ par $\uW^{G,0}$ et ayant recours à la formule de descente de l’induction on peut supposer que $Y=X_\ss$ et $M$ le facteur de Levi semi-standard d'un élément de $\cR^G(X)$. Puis en manipulant la formule de descente semi-simple pour $\widetilde{J}_M^Q(Y,f)=\widetilde{J}_M^Q[Y,f]$ et pour $J_M^Q(Y,f)$ on peut supposer que $Y=X_\ss=0$. Soit $P_\square \in\P^G(M)$. On est conduit à confronter
\begin{equation}\label{eq:compareArt'}
\begin{split}
\widetilde{J}_M^Q(0,f)=\lim_{A\to 0}\gamma^G(P_\square)|D^{\g}(A)|^{1/2}|\det(\ad(A);\n_\square(F))|^{-1}&\int_{\n_\square(F)}\int_{K}\\
&f\left(k^{-1}(A+V)k\right)w_{M}^{Q}\left(A,0,V\right)\,dk\,dV  
\end{split}
\end{equation}
et 
\begin{equation}\label{eq:compareMe}
\begin{split}
J_M^Q(0,f)=\gamma^G(P_\square)\int_{\n_\square(F)}\int_{K}f\left(k^{-1}Vk\right)v_{M,X}^{Q}\left(p_\square(V)k\right)\,dk\,dV    
\end{split}
\end{equation}
avec $p_\square(V)\in
 G_X(F)\backslash P_\square(F)$ défini par $V=p_\square(V)^{-1}Xp_\square(V)$ sur un ouvert dense de $\n_{\square}(F)$. 

Il est évident que 
\[|D^{\g}(A)|^{1/2}|\det(\ad(A);\n_\square(F))|^{-1}=1,\]
et, au moins pour la convergence simple, 
\[f\left(A+V\right)\underset{A\to 0}{\longrightarrow}f\left(V\right).\]
Le vrai problème est donc de comparer les poids :

\begin{lemma}
Soit $X\in \g(F)$ le représentant standard de sa $G(F)$-orbite nilpotente. Soient $M$ le facteur de Levi semi-standard d'un élément de $\cR^G(X)$,  et $P_\square\in \P^G(M)$. Pour tous $g\in G(F)$, $m\in M(F)$, $V\in \n_{\square}(F)$ et $k\in K$ tels que  
\begin{equation}\label{eq:poidscomparaisondefineg}
k^{-1}Vk=g^{-1}Xg,    
\end{equation}
on a pour tous $P\in \P^G(M)$ et $\lambda\in ia_{M}^\ast$
\begin{equation}\label{eq:poidscomparaison}
w_{P|P_{\square}}(\lambda,A=0,0,V)=\exp(\langle\lambda,R_{P_\square,X}(g)-R_{P,X}(g)\rangle).     
\end{equation}
\end{lemma}

\begin{remark}~{}
Ce lemme a été énoncé par Chaudouard \cite[proposition 6.4.1]{Ch17}. Cependant, une erreur a été identifiée en page 218 où il est mentionné qu'il est possible de supposer « $g\in \widetilde{P}$, ..., et $k=w_{P_1}$ ». Il convient de noter qu'une telle simplification est généralement irréalisable. Voici un contre-exemple, dans les notations de son article : soit $V=W_1\oplus W_2\oplus W_3$ avec $W_i$ un espace vectorielle sur $F$ de dimension 1 munit d'un vecteur non-nul. On a ainsi une base ordonnée de $V$ qui est tel que le vecteur de $W_i$ plus petit que celui de $W_j$ si et seulement si $i$ plus petit que $j
$. Soit $\g=\text{End}(V)$, un élément est représenté par une matrice dans la base ordonnée. Soit
    \begin{align*}
    X=\begin{bmatrix}
     0 &  0 &  1  \\
    0  &  0 &0\\
     0 &  0&0
    \end{bmatrix}       \,\,\,\,\text{et}\,\,\,\,Y=\begin{bmatrix}
     0 & 1  &  0    \\
     0 & 0  & 0  \\
    0  & 0 & 0  
     \end{bmatrix}  .
    \end{align*}
    Aussi $P_1=\widetilde{P_1}=\text{stab}((0)\subseteq W_1\subseteq V)$, $P_2=\text{stab}((0)\subseteq W_2\oplus W_3\subseteq V)$, et $\widetilde{P_2}=\text{stab}((0)\subseteq W_1\oplus W_2\subseteq V)$. Posons $\widetilde{P}\eqdef \widetilde{P_1}\cap \widetilde{P_2}$. Alors on n'a pas $Y\in (\Ad {\widetilde{P}})X$.

L'auteur tient à remercier Chaudouard pour lui avoir proposé un correctif lors d'une communication. Par la suite, une alternative à sa méthode originelle est explorée.
\end{remark}

\begin{proof}
On reformule d'abord lemme \ref{lem:actcent} comme suit : si $g_1,g_2\in G(F)$ sont tels que $g_1^{-1}Xg_1=g_2^{-1}Xg_2$ alors $R_{P_\square,X}(g_1)-R_{P,X}(g_1)=R_{P_
\square,X}(g_2)-R_{P,X}(g_2)$, on est donc libre de changer l'élément $g$ dans l'équation \eqref{eq:poidscomparaisondefineg}. Entamons la preuve : le lemme est trivial pour $P=P_\square$, puis selon le lemme \ref{lem:weightwadjpara} et le lemme \ref{lem:actcent} il suffit de traiter le cas où $P$ et $P_\square$ sont adjacents. Empruntons donc les notations de la sous-section \ref{subsec:situationelliptique}, sans commentaire. Déjà $r>1$ sinon $P=P_\square$, supposons que $P_\square=P_1$ et $P=P_2$, notons $Q$ le plus petit sous-groupe parabolique contenant $P_1$ et $P_2$ puis $P^{-}=\widetilde{P_1}\cap \widetilde{P_2}$ qui est pour rappel un sous-groupe parabolique. Par décomposition d'Iwasawa, on est ramené au cas où $g \in \widetilde{P_1}(F)$, on a alors $(\Ad g^{-1})X\in \widetilde{\mathfrak{p}_1}(F)$. Or $X\in \n_{\widetilde{P_1}}(F)$, on a $(\Ad w_{P_1}g^{-1})X\in \n_{\widetilde{P_1}}\cap \Ind_M^G(0)$. Il existe alors $p_1\in P_1(F)$ tel que $(\Ad p_1w_{P_1}g^{-1})X=V=(\Ad kg^{-1})X$. Il vient $g(w_{P_1}^{-1}p_1^{-1}w_{P_1})(w_{P_1}^{-1}k)g^{-1}=gw_{P_1}^{-1}p_1^{-1}kg^{-1}\in G_X(F)\subseteq \widetilde{P_1}(F)$, puis $w_{P_1}^{-1}k\in \widetilde{P_1}(F)\cap K(F)$. Quitte à translater $g$ à droite par $k^{-1}w_{P_1}$ on peut désormais supposer que $g\in \widetilde{P_1}(F)$ et $k=w_{P_1}$.

En tout conjuguant par $w_{P_1}^{-1}$ on voit que les deux côtés de l'égalité voulue \eqref{eq:poidscomparaison} ne dépend des quantités de l'équation \eqref{eq:poidscomparaisondefineg} que via leur projection sur $M_{\widetilde{Q}}$ et $\m_{\widetilde{Q}}$ (on entend la composante sur $M_{\widetilde{Q}}$ de $g$ par la décomposition $g\in \widetilde{P_\square}(F)\subseteq \widetilde{Q}(F)=M_{\widetilde{Q}}(F)N_{\widetilde{Q}}(F)$). On se ramène à prouver l'énoncé suivant : soient $\widetilde{m} \in \widetilde{P_1}(F) \cap M_{\widetilde{Q}} (F)$ et $\widetilde{\alpha}$ l'unique élément de $\Delta_{w_{P_1}^{-1}P_2w_{P_1}}^{\widetilde{Q}}$, $\widetilde{\lambda}\in ia_{\widetilde{M}}^\ast$ avec $\widetilde{M}=M_{\widetilde{P_1}}$, alors 
\begin{equation}\label{eq:(6.4.10)}
\lim_{\a_{\widetilde{M},\reg}(F)\ni \widetilde{A}\to 0}|\widetilde{\alpha}(\widetilde{A})|^{\rho(\widetilde{\alpha},0)\langle\widetilde{\lambda},\widetilde{\alpha}^{\vee}\rangle}\exp(-\langle\widetilde{\lambda},H_{w_{P_1}^{-1}P_2w_{P_1}}(n_{\widetilde{{P_1}}})\rangle)    =|\Nrd(U_{1,3})|\exp(\frac{1}{\deg D}\langle\lambda,\widetilde{\alpha}^\vee\rangle),
\end{equation}
où $\widetilde{V}\in \n_{\widetilde{P_1}}(F)$ est un point de l'orbite $(\Ad G)X$, et $n_{\widetilde{{P_1}}}\in M_{\widetilde{Q}} (F)\cap N_{\widetilde{P_1}}(F)$ est défini par 
\begin{equation}
n_{\widetilde{{P_1}}}^{-1}\widetilde{A}n_{\widetilde{{P_1}}}=\widetilde{A}+\widetilde{V},    
\end{equation}
enfin $U$ est l'élément qui correspond à $Y\eqdef\widetilde{V}$ dans le lemme \ref{lem:diffadjacentR_P}.

Il est clair que la limite \eqref{eq:(6.4.10)} ne dépend de $\widetilde{A}$, $\widetilde{m}$, $X$ etc. que via leur projection sur la composante $\text{Aut}_D (W_1\oplus W_2\oplus W_3)$ dans la décomposition en facteurs irréductibles de $M_{\widetilde{Q}}$ (si $k_{W}\in K\cap \text{Aut}_D (W_1\oplus W_2\oplus W_3)$ est tel que $U_W=k_{W}Y_{W}k_W^{-1}\in  \Hom_D(W_1,W_3)$ dans le lemme \ref{lem:diffadjacentR_P}, avec $U_W$ et $Y_W$ leur porjection sur $\text{Aut}_D (W_1\oplus W_2\oplus W_3)$, alors en considérant $k_W$ comme un élément de $K$ par l'extension par la matrice identité en dehors de $W_1\oplus W_2\oplus W_3$ on a $U=k_{W}Yk_W^{-1}\in n_{\widetilde{P_1}}\cap n_{\widetilde{P_2}}$). On a va donc supposer $M_{\widetilde{Q}}=\text{Aut}_D (W_1\oplus W_2\oplus W_3)$. Dans la suite on projecte les objets sur $M_{\widetilde{Q}}$, et par abus de notation on continue à les noter par les mêmes symboles. Nous avons
\begin{align*}
\widetilde{P_1}=\left\{\begin{bmatrix}
\ast & \ast  &  \ast    \\
& \ast  &\ast \\
& \ast  & \ast
\end{bmatrix}   \right\}     ,\,\,\,\,\widetilde{P_2}=\left\{\begin{bmatrix}
\ast  &  \ast & \ast    \\
\ast& \ast  & \ast \\
&  &  \ast
\end{bmatrix}\right\} 
 ,\,\,\,\,P^-=\left\{\begin{bmatrix}
\ast  &  \ast & \ast    \\
& \ast  & \ast \\
&  &  \ast
\end{bmatrix}\right\},
\end{align*}

\begin{align*}
X=\begin{bmatrix}
0  & 0  &  \text{Id}_{r_1\times r_1}    \\
& 0 &0\\
&  &  0
\end{bmatrix}       \,\,\,\,\text{et}\,\,\,\,\widetilde{m}=\begin{bmatrix}
\ast  &  \ast & \ast    \\
& \ast  & \ast \\
&  \ast &  \ast
\end{bmatrix}  .
\end{align*}
Or un élément $\widetilde{A}\in\a_{\widetilde{M},\reg}(F)$ s'écrit
\begin{align*}
\widetilde{A}=\begin{bmatrix}
 a \text{Id}_{r_1\times r_1} &   &     \\
&  b \text{Id}_{r_2\times r_2} &\\
&  &  b \text{Id}_{r_1\times r_1}
\end{bmatrix}    
\end{align*}
avec $a\not= b$ deux éléments de $F$. On a $\widetilde{\alpha}(\widetilde{A})=b-a$. On calcule facilement $n_{\widetilde{{P_1}}}$ et on trouve $n_{\widetilde{{P_1}}}=\text{Id} -\widetilde{\alpha}(\widetilde{A})^{-1}\widetilde{V}$.

On prétend à ce stade qu'il existe $k\in K\cap M_{\widetilde{P_1}}\subseteq K$ de la forme
\[
k=\begin{bmatrix} 
\text{Id}_{r_1\times r_1} & \mspace{-10mu} 
  \begin{matrix}  \,&  \, \end{matrix} \\
\begin{matrix} \, \\ \,  \end{matrix} & \mspace{-10mu}
  \begin{bmatrix}
    \vphantom{\begin{matrix} \, \\ \,  \end{matrix}}
    \vcenter{\hbox{\hspace*{7pt}$ k'$\hspace*{7pt}}} 
  \end{bmatrix}
\end{bmatrix},
\]
avec $k'\in \begin{cases*}\GL_{r_2+r_1}(\O_D)  & \text{si $F$ est non-archimédien} \\
 \{g\in \GL_{r_2+r_1}(D):g\cdot{}^t\overline{g}=\text{Id}\}& \text{si $F$ est archimédien} 
\end{cases*}$, tel que $k\widetilde{V}k^{-1}\in  \n_{\widetilde{P_1}}\cap \n_{\widetilde{P_2}}$ comme dans le lemme \ref{lem:diffadjacentR_P}. Par des calculs directs on voit que l'assertion est équivalente à : étant donnée une matrice $\widetilde{V}\in \Mat_{r_1\times (r_2+r_1)}(D)$, il existe $k'\in \begin{cases*}\GL_{r_2+r_1}(\O_D)  & \text{si $F$ est non-archimédien} \\
 \{g\in \GL_{r_2+r_1}(D):g\cdot{}^t\overline{g}=\text{Id}\}& \text{si $F$ est archimédien} 
\end{cases*}$ tel que les premières $r_2$ colonnes de $\widetilde{V}k'$ soient 0. C'est un exercice élémentaire en algèbre linéaire. Pour $F$ archimédien, on prend sur $D^{r_2+r_1}$ le produit scalaire tel que la base canonique soit une base orthonormée, on procède par récurrence pour définir la $i$-ème colonne ($1 \leq i \leq r_2$) de $k'$ comme un vecteur de longueur 1 dans $D^{r_2+r_1}$, orthogonal aux lignes de $\widetilde{V}$ et aux premières $i-1$ colonnes de $k'$, ensuite on complète cela en une base orthonormée de $D^{r_2+r_1}$. Pour $F$ non-archimédien, quitte à multiplier les colonnes de $\widetilde{V}$ par des scalaires on peut supposer que $\widetilde{V}$ est une matrice à coefficient dans $\O_D$ et chaque ligne non-nulle contient au moins un coefficient dans $\O_D^\times$. Soit $S$ le $\O_D$-schéma définie par $S(E)=\{k''\in \GL_{r_2+r_1}(E) : \text{les premières $r_2$ colonnes de $\widetilde{V}k''$ sont 0}\}$ pour toute $\O_D$-algèbre $E$. C'est un schéma lisse. Il est connu que $\O_D$ est un anneau hensélien (\cite[p.324]{Pi82}). Soit $\mathfrak{m}_D$ l'idéal maximal de $\O_D$. On procède par récurrence pour définir la $i$-ème colonne ($1 \leq i \leq r_2$) d'un élément de $S(\O_D/\mathfrak{m}_D)$ comme un vecteur solution non-nulle des équations associées aux lignes de $\widetilde{V} \text{ mod }\mathfrak{m}_D$ et linéairement indépendant aux premières $i-1$ colonnes de cet élément, ensuite on complète cela en une matrice de $\GL_{r_2+r_1}(\O_D /\mathfrak{m}_D)^{r_2+r_1}$, et on relève cela en un élément de $S(\O_D)$ par le lemme de Hensel.

Enfin, nous sommes en mesure d'étudier l'égalité \eqref{eq:(6.4.10)}. La matrice $U_{1,3}$ en question est les dernières $r_1$ colonnes de $\widetilde{V}k'$, et donc 
\[n_{\widetilde{{P_1}}}= \begin{bmatrix} 
\text{Id}_{r_1\times r_1} &   
  -\widetilde{\alpha}(\widetilde{A})^{-1}[0 \,\, U_{1,3}] k'^{-1} \\
\begin{matrix} \, \\ \,  \end{matrix} & 
  { \text{Id}}_{(r_2+r_1)\times (r_2+r_1)}
\end{bmatrix}.\]

Clairement, il suffit d'obtenir l'égalité \eqref{eq:(6.4.10)} pour l'élément particulier $\widetilde{\lambda}=\omega\in X^\ast(\widetilde{M})$ est le caractère $m\mapsto \Nrd(m_{2,3})$ avec $m_{2,3}$ la composante sur $\Aut_D(W_2\oplus W_3)$ de $\widetilde{M}=\Aut_D(W_1)\times \Aut_D(W_2\oplus W_3)$. On retiendra qu'on a $\langle \omega,\widetilde{\alpha}^\vee\rangle =\deg D$. Prenons
la puissance extérieure
$\bigwedge_D^{r_1+r_2}(W_1\oplus W_2\oplus W_3)$, qu'on munit de la base déduite de la base de $W_1\oplus W_2\oplus W_3$ extraite de $e$. Soit $\phi_\omega$ l'unique
élément  de la base qui appartient à $\bigwedge^{r_1+r_2}(W_2\oplus W_3)$, c'est un vecteur extrémal de poids $\omega$. Soit $\|\cdot\|$ une norme au sens des normes vectorielles (ici l'espace en question est vu comme espace vectoriel sur $F$ et non sur $D$) telle que $\phi_\omega$ soit de longueur 1. On a 
\begin{align*}
\exp(-\omega,H_{w_{P_1}^{-1}P_2w_{P_1}}(n_{\widetilde{{P_1}}})\rangle)&=\|\bigwedge^{r_1+r_2}(n_{\widetilde{{P_1}}}^{-1})\phi_\omega\|\\
&=|\widetilde{\alpha}(\widetilde{A})|^{-r_1\deg D}|\Nrd(U_{1,3})||\Nrd(k'^{-1})|+o_{\widetilde{A}=0}(|\widetilde{\alpha}(\widetilde{A})|^{-r_1\deg D})
\end{align*}
De ce fait, quand $\widetilde{A}$ tend vers $0$, l'expression $|\widetilde{\alpha}(\widetilde{A})|^{r_1\deg D}\exp(-\langle\widetilde{\lambda},H_{w_{P_1}^{-1}P_2w_{P_1}}(n_{\widetilde{{P_1}}})\rangle)$ admet une limite qui vaut $|\Nrd(U_{1,3})|$. Ce qu'il fallait.
\end{proof}

Revenons sur la comparaison des intégrales orbitales pondérées. Le lemme précédent avèrent la convergence simple
\[w_{M}^{Q}\left(A,0,V\right)\underset{A\to 0}{\longrightarrow}v_{M,X}^{Q}\left(p_\square(V)k\right).\]

Prenons $C$ une sous-partie compacte de $\a_{M,G-\reg}(F)$ contenant $0$. Posons $f'(-)\eqdef \max_{A\in C}|f(A+-)|$, on obtient une fonction
\[f'\in \begin{cases}
\S(\g(F)) & \text{si $F$ est non-archimédien} \\
C^0(\g(F)) & \text{si $F$ est archimédien}
\end{cases}\]
vérifiant que $\|f'\|_{a,0}<+\infty$ pour tout $a\in \N$. Or il existe une constante $c$ tel que pour tout $A\in C$, on ait 
\[\left|w_{M}^{Q}\left(A,0,V\right)\right|\leq c\left|v_{M,X}^{Q}\left(p_\square(V)k\right)\right|+1.\] 
Au moyen du théorème de convergence dominée et la preuve du théorème \ref{thm:IOPfullstatement}, on met un point final à la comparaison.
\end{proof}

\subsection{Intégrale orbitale pondérée semi-locale}
Soit exceptionnellement dans cette partie $F$ un corps global. Soit $S$ un sous-ensemble fini non-vide de places de $F$. On note $F_v$ le complété local de $F$ en $v\in S$, et $F_S\eqdef \prod_{v\in S}F_v$. Soit $G$ un groupe du type GL sur $F$, $M_0$ un sous-groupe minimal de $G$, $K_S=\prod_{v\in S}K_v$ tel que $K_v$ est un sous-groupe compact de $G(F_v)$ en bonne position par rapport à $M_0$, et $M\in \L^G(M_0)$.

On a la théorie des $(G,M)$-familles sur $F$ (cf. \cite[section 7]{Art88I}). Notons $M_S=\prod_{v\in S}M_v$ et le voit comme sous-groupe de Levi de $G_S=\prod_{v\in S}G_v$ défini sur $F_S$. On note de plus $\L^{G_S}(M_S)$ (resp. $\P^{G_S}(M_S)$ ; $\F^{G_S}(M_S)$) l'ensemble des produits $\prod_{v\in S}L_v$ (resp. $\prod_{v\in S}P_v$ ; $\prod_{v\in S}Q_v$) où $L_v$ est un sous-groupe de Levi défini sur $F_v$ contenant $M_v$ (resp. $P_v$ est un sous-groupe parabolique défini sur $F_v$ et de facteur de Levi $M_v$ ; $Q_v$ est un sous-groupe parabolique défini sur $F_v$ contenant $M_v$), par une récurrence on montre l'existence des applications, toujours en moyennant certains choix (cf. \cite[section 9]{Art88I})
\begin{align*}
    d_M^G :\L^{G_S}(M_S)&\rightarrow [0,+\infty[ \\
    s:\L^{G_S}(M_S)&\rightarrow \F^{G_S}(M_S)
\end{align*}
de sorte que, pour tout $(L_v)_{v\in S}\in \L^{G_S}(M_S)$, on ait
\begin{enumerate}
    \item si $s\left((L_v)_{v\in S}\right)=\left((Q_v)_{v\in S}\right)$ alors $(Q_v)_{v\in S}\in \P^{G_S}(L_S)$ ;
    \item $d_M^G\left((L_v)_{v\in S}\right)\not =0$  si et seulement si l'une des flèches naturelles
    \[\bigoplus_{v\in S}a_{M_v}^{L_v}\longrightarrow a_M^G\]
    et
    \[\bigoplus_{v\in S}a_{L_v}^{G_v}\longrightarrow a_M^G\]
    est un isomorphisme, auquel cas les deux sont isomorphismes et $d_M^G\left((L_v)_{v\in S}\right)$ est le volume dans $a_M^G$ du parallélotope formé par les bases orthonormées des $a_{M_v}^{L_v}$ ;
    \item Formule de scindage : supposons que nous disposons, pour tout $v\in S$, de $(c_{P_v})_{P_v\in\P^{G_v}(M_v)}$ une $(G_v,M_v)$-famille sur $F_v$. Définissons $c_P=\prod_{v\in S}c_{P_v}$. Alors $(c_P)_{P\in \P^G(M)}$ est une $(G,M)$-famille, et
    \begin{equation}\label{eq:GMfamillescindageformula}
    c_{M}=\sum_{(L_v)_{v\in S}\in \L^{G_S}(M_S)}d_M^G\left((L_v)_{v\in S}\right)\prod_{v\in S} c_{M_v}^{Q_v}.    
    \end{equation}
\end{enumerate}

Soit $\o\subseteq \m$ une classe de $M$-conjugaison contenant un $F$-point. Soit $Y\in \o(F_S)$. On prendra dans la suite $A\in \a_{M,\o,G-\reg}(F_S)$. Alors $M_{A+Y}=G_{A+Y}$. Fixons $P_\square\in \P^{G}(M)$, il nous fournit le point-base $P_{\square,v}\in \P^{G_{v}}(M_{v})$ pour tout $v\in S$. Pour tout $v$ il y a $A_v\in \a_{M_{v},\o_v}(F_v)$ et $Y_v\in \m_v(F_v)$, on a alors les fonctions $r_P(\lambda_v,A_v,Y_{v})$ et $r_P[\lambda_v,A_v,Y_{v}]$ introduites dans les sous-sections \ref{subsec:defdirect} et \ref{subsec:defviadescentss}. On pose alors
\[r_P(\lambda,A,Y)=\prod_{v\in S}r_{P_v}(\lambda_v,A_v,Y_v)\,\,\,\,\text{et}\,\,\,\,r_P[\lambda,A,Y]=\prod_{v\in S}r_{P_v}[\lambda_v,A_v,Y_v]\,\,\,\,,\lambda\in ia_{M,\C}.\]

Soit $\S(\g(F_S))\eqdef\widehat{\bigotimes}_{v\in S}\S(\g(F_v))$, avec $\bigotimes$ le produit tensoriel d'espaces vectoriels topologiques, muni de sa topologie usuelle, cf. \cite[définition 43.1 ou 43.2]{BookTreve} ($\S(\g(F_v))$ étant nucléaire, d'après \cite[remarque, p.53]{Bruhat1961}, la $\epsilon$-topologie et la $\pi$-topologie sur $\bigotimes_{v\in S}\S(\g(F_v))$ sont donc la même), puis $\widehat{\bigotimes}$ le complété pour cette la topologie.  

Soient $f\in\S(\g(F_S))$ et $Q\in \F^G(M)$. On suppose que $M_Y=G_Y$. Posons 
\[\widetilde{J}_{M}^Q(Y,f)=|D^{\g}(Y)|_S^{1/2}\int_{M_Y(F_S)\backslash G(F_S)} f\left((\Ad g^{-1})Y\right)v_M^Q(g)\,dg.\]
Abondonnons l'hypothèse $M_Y=G_Y$, posons
\begin{equation}
 \widetilde{J}_M^Q(Y,f)\eqdef\lim_{\substack{A\to 0\\ A\in \a_{M,Y,G-\reg}(F_S)}}\sum_{L\in \L^{M_Q}(M)}r_M^L(A,Y)\widetilde{J}_{L}^Q(A+Y,f),   
\end{equation}
et
\begin{equation}
 \widetilde{J}_M^Q[Y,f]\eqdef\lim_{\substack{A\to 0\\ A\in \a_{M,Y,G-\reg}(F_S)}}\sum_{L\in \L^{M_Q}(M)}r_M^L[A,Y]\widetilde{J}_{L}^Q(A+Y,f).   
\end{equation}

\begin{theorem}
$\widetilde{J}_M^Q(Y,-)$ et $\widetilde{J}_M^Q[Y,-]$ sont bien définis, $\widetilde{J}_M^Q(Y,-)=\widetilde{J}_M^Q[Y,-]$, et c'est une distribution tempérée, i.e. une fonctionnelle continue sur $\S(\g(F_S))$.
\end{theorem}
\begin{proof}
C'est une conséquence de la formule de scindage et l'énoncé local puisque, pour $f=\bigotimes_{v\in S} f_v\in\S(\g(F_S))$ décomposable,   
\begin{align*}
\sum_{L\in \L^{M_Q}(M)}&r_M^L(A,Y)\widetilde{J}_{L}^Q(A+Y,f)\\
&=\sum_{(T_v)_{v\in S}\in \L^{M_{Q,S}}(M_S)}d_M^{M_Q}\left((T_v)_{v\in S}\right)\prod_{v\in S}\sum_{L_v\in \L^{T_v}(M_v)}r_{M_v}^{L_v}(A_v,Y_v)\widetilde{J}_{L_v}^{Q_{L_v}}(A_v+Y_v,f_v),    
\end{align*}
et similairement pour $\widetilde{J}_M^Q[Y,f]$.
\end{proof}

\newcommand{\noop}[1]{}


\begin{thebibliography}{BHLS10}

\bibitem[Ar81]{Art81}
J. Arthur.
\newblock {\em The trace formula in invariant form}.
\newblock Ann. of Math. (2) 114, 1-74. (1981).

\bibitem[Ar82]{Art82II}
J. Arthur.
\newblock {\em On a family of distributions obtained from Eisenstein series II: Explicit formulas}.
\newblock Amer. J. Math. 104, 1289-1336 (1982).

\bibitem[Ar86]{Art86}
J. Arthur.
\newblock {\em On a family of distributions obtained from orbits}.
\newblock Canad. J. Math. 38, 179-21 4. (1986).

\bibitem[Ar88a]{Art88}
J. Arthur.
\newblock {\em The local behaviour of weighted orbital integrals}.
\newblock Duke Math. J . 56, 223-293. (1988).

\bibitem[Ar88b]{Art88I}
J. Arthur.
\newblock {\em The invariant trace formula {I}. Local theory}.
\newblock J. Amer. Math. Soc., 1(2):323–383 (1988).

\bibitem[Br61]{Bruhat1961}
F. Bruhat
\newblock {\em Distributions sur un groupe localement compact et applications à l’étude des représentations des groupes $p$-adiques}.
\newblock Bulletin de la Société Mathématique de France, 43-75 (1961).

\bibitem[BT72]{BruTit72}
F. Bruhat and J. Tits
\newblock {\em Groupes réductifs sur un corps local : I. Données radicielles valuées}.
\newblock Publications Mathématiques de l'IHÉS, Volume 41, pp. 5-251 (1972).

\bibitem[CGP15]{CGPbook}
B. Conrad and O. Gabber and G. Prasad
\newblock {\em Pseudo-Reductive Groups}.
\newblock 2nd ed. Cambridge: Cambridge University Press (2015).

\bibitem[Ch02]{Ch02a}
P.-H. Chaudouard.
\newblock {\em La formule des traces pour les algèbres de Lie}.
\newblock Math. Ann., 322(2):347-382 (2002).

\bibitem[Ch17]{Ch17}
P.-H. Chaudouard.
\newblock {\em Sur la contribution unipotente dans la formule des traces d'Arthur pour les groupes généraux linéaires}.
\newblock Israël Math. J., 218(1):175--271 (2017).

\bibitem[Ch18]{Ch18}
P.-H. Chaudouard.
\newblock {\em Sur certaines contributions unipotentes dans la formule des traces d'Arthur}.
\newblock Amer. J. Math., 140(3):699--752 (2018).

\bibitem[CM93]{CM93}
D. H. Collingwood and W. M. McGovern
\newblock {\em Nilpotent Orbits in Semisimple Lie Algebras: An Introduction}.
\newblock Chapman and Hall/CRC (1993).

\bibitem[EGAIV1]{EGAIV1}
A. Grothendieck
\newblock {\em Éléments de géométrie algébrique (rédigée avec la collaboration de Jean Dieudonné) : IV. Étude locale des schémas et des morphismes de schémas, Première partie}.
\newblock Inst. Hautes Études Sci. Publ. Math., no.20, p.5-259 (1964).

\bibitem[EGAIV2]{EGAIV2}
A. Grothendieck
\newblock {\em Éléments de géométrie algébrique (rédigée avec la collaboration de Jean Dieudonné) : IV. Étude locale des schémas et des morphismes de schémas, Seconde partie}.
\newblock Inst. Hautes Études Sci. Publ. Math., no.24, p.5-255 (1966).

\bibitem[EGAIV3]{EGAIV3}
A. Grothendieck
\newblock {\em Éléments de géométrie algébrique (rédigée avec la collaboration de Jean Dieudonné) : IV. Étude locale des schémas et des morphismes de schémas, Troisième partie}.
\newblock Inst. Hautes Études Sci. Publ. Math., no.28, p.5-255 (1966).

\bibitem[FHW18]{FiHoffWaka18}
T. Finis and W. Hoffmann and S. Wakatsuki
\newblock {\em The Subregular Unipotent Contribution to the Geometric Side of the Arthur Trace Formula for the Split Exceptional Group $G_2$}.
\newblock Geometric Aspects of the Trace Formula, Simons Symposia. Springer, Cham (2018).

\bibitem[FL16]{FiLa16}
T. Finis and E. Lapid.
\newblock {\em On the continuity of the geometric side of the trace formula}.
\newblock Acta Mathematica Vietnamica, vol. 41, no. 3, 425-455 (2016).

\bibitem[Ho12]{Hoff12}
W. Hoffmann
\newblock {\em Induced conjugacy classes, prehomogeneous varieties, and canonical parabolic subgroups}.
\newblock arXiv:1206.3068v2 (2012).

\bibitem[Ho16]{Hoff16}
W. Hoffmann
\newblock {\em The trace formula and prehomogeneous vector spaces}.
\newblock Families of Automorphic Forms and the Trace Formula, Simons Symp., Springer-Verlag, Cham, Switzerland (2016).

\bibitem[HW18]{HoffWaka18}
W. Hoffmann and S. Wakatsuki
\newblock {\em On the Geometric Side of the Arthur Trace Formula for the Symplectic Group of Rank 2}.
\newblock Memoirs of the American Mathematical Society, Volume: 255 (2018).

\bibitem[Ko82]{Kott82}
R. Kottwitz
\newblock {\em Rational conjugacy classes in reductive groups}.
\newblock Duke Math. J. 49, 785-806 (1982).

\bibitem[Ko05]{Kottbook}
R. Kottwitz
\newblock {\em Harmonic analysis on reductive p-adic groups and Lie algebras}.
\newblock in "Harmonic Analysis, the Trace Formula, and Shimura Varieties, Clay Mathematics Proceedings, Volume 4 393–522" (2005).

\bibitem[KP79]{KraPro79}
H. Kraft and C. Procesi
\newblock {\em Closures of conjugacy classes of matrices are normal}.
\newblock Inventiones mathematicae 53, 227–248 (1979).

\bibitem[KWY20]{KimWakaYama20}
H. Kim and S. Wakatsuki and T. Yamauchi
\newblock {\em An equidistribution theorem for holomorphic siegel modular forms for $GSp_4$ and its applications}.
\newblock Journal of the Institute of Mathematics of Jussieu, vol. 19, no. 2 (2020).

\bibitem[LS79]{LuszSpal79}
G. Lusztig and N. Spaltenstein
\newblock {\em Induced Unipotent Classes}.
\newblock Journal of the London Mathematical Society, Volume s2-19, Issue 1, 41–52 (1979).

\bibitem[Ma17]{Ma17}
J. Matz
\newblock {\em Weyl’s law for Hecke operators on $GL(n)$ over imaginary quadratic number fields}.
\newblock American Journal of Mathematics, Johns Hopkins University Press, Volume 139 (2017).

\bibitem[MT15]{MaTe15}
J. Matz and N. Templier
\newblock {\em Sato-Tate equidistribution for families of Hecke-Maass forms on $\SL(n,\R)/SO(n)$}.
\newblock arXiv:1505.07285 (2015).

\bibitem[MW16]{MW16}
C. Moeglin and J.-L. Waldspurger
\newblock {\em Stabilisation de la formule des traces tordue Volume 1}.
\newblock Birkhäuser Cham. Progress in Mathematics 316 (2016).

\bibitem[Pi82]{Pi82}
R. S. Pierce
\newblock {\em Associative Algebras}.
\newblock Graduate Texts in Math. 88, Springer, New York-Heidelberg-Berlin (1982).

\bibitem[Ra72]{Rao72}
R. Ranga Rao
\newblock {\em Orbital Integrals in Reductive Groups}.
\newblock Annals of Mathematics Second Series, Vol. 96, No. 3 (1972).

\bibitem[SGA3]{SGA3}
M. Demazure
\newblock {\em Sous-groupes paraboliques des groupes réductifs}.
\newblock Exposé XXVI, Réédition de Séminaire de Géométrie Algébrique du Bois Marie - 1962-64 (SGA 3), versions finales du 29 mars 2011 \url{https://webusers.imj-prg.fr/~patrick.polo/SGA3/} (2011).

\bibitem[Si79]{Sil79}
A. G. Silberger
\newblock {\em Introduction to Harmonic Analysis on Reductive P-adic Groups. (MN-23)}.
\newblock Volume 23 in the series Mathematical Notes, Princeton: Princeton University Press (1979).

\bibitem[Tr07]{BookTreve}
F. Tr{\`{e}}ves
\newblock {\em Topological Vector Spaces, Distributions and Kernels}.
\newblock Dover Publications Inc.; Illustrated edition (2007).


\bibitem[Va77]{Vara77}
V. S. Varadarajan
\newblock {\em Harmonic Analysis on Real Reductive Groups}.
\newblock Lecture Notes in Mathematics 576, Springer-Verlag Berlin Heidelberg (1977).

\bibitem[Wa95]{Walds95}
J.-L. Waldspurger
\newblock {\em Une formule des traces locale pour les algèbres de Lie p-adiques}.
\newblock Journal für die reine und angewandte Mathematik 465: 41-100. (1995).

\bibitem[Wal97]{Walds97}
J.-L. Waldspurger.
\newblock {\em Le lemme fondamental implique le transfert}.
\newblock Compositio Math. 105, 153-236, (1997).

\bibitem[YDL24]{YDL23b}
Y.-D. Lu
\newblock {\em Approche non-invariante de la correspondance de Jacquet-Langlands : analyse géométrique}.
\newblock soumis à Documenta Mathematica (en révision), arXiv:2410.10059 (2024).
\end{thebibliography}
\end{document}